\documentclass[11pt]{article}
\usepackage{authblk}
\usepackage{setspace}
\usepackage[utf8]{inputenc}

\usepackage{graphicx} 
\graphicspath{ {./images/} }
\usepackage[rightcaption]{sidecap}
\usepackage{wrapfig}
\usepackage{tikz}
\usepackage[skip=4pt]{caption}
\usepackage[flushleft]{threeparttable} 
\usepackage{bm} 
\usepackage{bbm}
\usepackage{amsmath} 
\usepackage{amsthm}
\usepackage{amssymb}
\usepackage{pdflscape}
\usepackage{array}
\usepackage{url}
\usetikzlibrary{positioning, shapes.geometric, arrows.meta, calc}
\usepackage{subcaption} 
\usepackage{titlesec}
\titlespacing*{\section}
  {0pt}{1ex}{1ex}  
\titlespacing*{\subsection}
  {0pt}{0.5ex}{0.5ex}  
\makeatletter
\renewcommand\paragraph{\@startsection{paragraph}{4}{\z@}%
  {1.5ex \@plus1ex \@minus.2ex}%
  {0.8ex}%
  {\normalfont\normalsize\bfseries}}
\makeatother

\usepackage{multirow}
\usepackage{comment}
\usepackage{array}
\usepackage{makecell}
\newtheorem{assumption}{Assumption}

\newtheorem{definition}{Definition}
\newtheorem{example}{Example}
\newtheorem{remark}{Remark}

\newtheorem{theorem}{Theorem}

\newtheorem{corollary}[theorem]{Corollary}
\newtheorem{proposition}[theorem]{Proposition}
\usepackage{colortbl}

\usepackage{color}
\usepackage{enumerate}
\usepackage[noend]{algpseudocode}
\usepackage[hidelinks]{hyperref}
\usepackage{pgf}
\usepackage{tikz}
\usetikzlibrary{arrows,automata}
\usepackage{booktabs}
\usepackage{adjustbox}
\usepackage[normalem]{ulem}
\useunder{\uline}{\ul}{}
\usepackage{lscape} 
\usepackage{placeins}
\usepackage{algorithm}
\usepackage{algpseudocode}

\newcounter{myparagraph}[subsubsection]

\usepackage{geometry} 
\title{A Primal Perspective on Distributionally Robust Optimization: An Investigation on Modeling and Solution Strategies}
\author{Yiqi Tian and Bo Zeng\\
Department of Industrial Engineering, University of Pittsburgh}
\date{}

\begin{document}

\maketitle

\singlespacing
\begin{abstract}
As a popular optimization scheme, distributionally robust optimization (DRO) protects decisions against ambiguity in probability distributions. For (single-stage) DRO, prevailing dual reformulations can become difficult when model or ambiguity-set structures are complex. We study DRO from a primal perspective, working directly with distributions in ambiguity sets on closed, potentially unbounded sample spaces. This perspective leads to an algorithmic framework, referred to as BiCS, that constructs and leverages distribution cuts to achieve strong performance. We show that BiCS is applicable to standard DRO, almost-sure DRO, DRO with various chance constraints, and DRO with ambiguity sets strengthened by local information. Numerical experiments with moment and Wasserstein ambiguity sets show that this framework demonstrates superior performance, including solving cases where the examined compact reformulations are unavailable or computationally difficult. The local-information study also makes changes in worst-case distributions directly visible.
\end{abstract}

\section{Introduction}
Randomness or uncertainty is an inherent and unavoidable component of almost all real-world decision-making problems. In many applications, such as supply chain planning, healthcare operations, and financial portfolio design, decision makers must act based on incomplete data, imperfect estimates, and uncertain future conditions. Ensuring feasibility under adverse realizations while balancing risk with economic efficiency has long been a central topic in optimization. To this end, several modeling paradigms have been developed, among which stochastic programming (SP), robust optimization (RO), and distributionally robust optimization (DRO) are the most widely adopted in both research and practice \cite{bertsimas2011theory, wiesemann2014distributionally, kuhn2025distributionally}. Compared to deterministic models, these paradigms allow randomness/uncertainty to be considered explicitly (through different representations), and offer various trade-offs between protecting against adverse outcomes and producing cost-effective solutions. 

After a modeling paradigm is chosen, an equally important question is how to solve the resulting optimization problem efficiently. For SP and RO, many fast, scalable, and intuitive exact and approximate algorithms have been developed and successfully applied in practice. We note that these strong solution methods have inspired and enabled further extensions and variants of SP and RO to meet broader practical needs and application environments, such as incorporating more complex random factor descriptions or multi-stage decision structures. Certainly, these extensions and variants require continual algorithmic advancements. Nowadays, such mutually reinforcing progress has made SP and RO sufficiently mature for real-world deployment \cite{sahinidis2004optimization, gabrel2014recent, gorissen2015practical}. Nevertheless, this is not necessarily the case for DRO, a rather recent paradigm. This is largely due to the fact that many existing solution approaches rely on mathematically sophisticated reformulations, which can be computationally demanding or challenging to adapt beyond standard or highly structured settings. As a result, many existing studies remain restricted to DRO formulations with standard or highly structured problem settings. Consequently, despite its ability to safeguard out-of-sample performance, translating DRO models into real-world applications remains challenging. 

Clearly, further research efforts are required to make DRO methodologies more accessible and intuitive, enhance their computational tractability, and expand their modeling power and flexibility to handle different settings in practice. Against this backdrop, this paper (i) reviews and synthesizes existing (single-stage) DRO research and related methodologies, while identifying key gaps and inconsistencies that warrant attention; (ii) presents a primal-based decomposition method (and its variants) that is both mathematically intuitive and computationally efficient. Building upon this strong and flexible method, we further develop new DRO extensions that can accommodate different risk attitudes, computational requirements, and data-driven practices (including the integration of domain knowledge). These modeling extensions enable us to address problems that would otherwise be impractical to solve. Overall, our results significantly expand the applicability and practical relevance of DRO in real-world decision-making environments.

The paper is organized as follows. Section~\ref{sec:Relevant literature} revisits how uncertainty/randomness is modeled under different risk-treatment paradigms, including SP, RO, DRO, and related variants, and surveys the prevailing solution schemes associated with these formulations. Section~\ref{sec:DRO} explains how to extend and adapt the RO solution scheme to DRO, especially from a primal perspective. Section~\ref{sec:AS_DRO} introduces the crucial yet previously overlooked concept of almost-sure DRO and demonstrates how our unified solution scheme effectively addresses this formulation. Section~\ref{sec:extensions} broadens the discussion to encompass various nonconvex and discontinuous settings, including DRCCP (Section~\ref{sec:DRCCP}) and locally informed ambiguity set definitions (Section~\ref{sec:local_info}). Section~\ref{sec:results} presents extensive computational experiments using both first- and second-order moment ambiguity sets as well as $\ell_1$- and $\ell_2$-norm Wasserstein ambiguity sets. Finally, Section~\ref{sec:conclusion} provides a conclusion for this paper.

\section{Review of Relevant Research}\label{sec:Relevant literature}

In this section, we first provide a conceptual overview of existing approaches to modeling uncertainty in decision making and identify a key conceptual gap that has not been fully addressed in the literature. We then summarize the current solution schemes associated with these modeling paradigms. Finally, we outline our study and list contributions to the literature.

\subsection{Modeling Uncertainties}

When little reliable distributional information is available or the decision maker focuses only on the most disruptive possible outcome, RO \cite{soyster1973convex,ben1998robust} offers a natural modeling paradigm by optimizing against the worst-case realization while ensuring feasibility across a prescribed uncertainty set. RO is not usually viewed as a risk measure in the classical sense \cite{artzner1999coherent, FollmerSchied2016}, since it does not rely on an explicit probability distribution and instead enforces worst-case protection over an uncertainty set \cite{ben1998robust}. Nevertheless, it admits a natural interpretation under a broader concept of risk, which we term \textit{risk treatment}. In this paper, we use risk treatment to describe how uncertainty is protected against in an optimization model, regardless of the availability of any probabilistic information.  It subsumes the classical concept of risk measures, which are typically defined as functionals on random variables and encode risk attitudes through probability distributions \cite{FollmerSchied2016}. From this perspective, the supremum-based protection in RO reflects a highly conservative treatment: every admissible scenario is treated as potentially critical and feasibility is enforced over the entire uncertainty set. This property makes RO particularly appealing when even low-probability adverse events can lead to severe consequences. On the other hand, RO is often criticized for producing overly conservative solutions, especially when the uncertainty set is large or loosely specified.

As additional distributional information becomes available, it is certainly more reasonable to move beyond purely set-based protection and allow risk to be treated in a probabilistic manner. If the distribution is known exactly, SP \cite{birge2011introduction,prekopa2013stochastic} targets expected performance by prioritizing typical realizations, making it ideal for settings with well-understood variability \cite{gaspars2017newsvendor,irawan2022stochastic, burdett2023stochastic}. Moreover, when the decision maker instead seeks explicit control over adverse outcomes, attention naturally shifts from expectation-based criteria to tail-based risk measures. Particularly, chance-constrained optimization (CCP) \cite{charnes1959chance,nemirovski2007convex} captures this preference by limiting the violation probability to $\epsilon$, thereby controlling the likelihood of unfavorable tail events based on the Value at Risk (VaR) principle. VaR is a quantile-based risk measure that identifies a specified distributional quantile rather than accounting for average performance \cite{liu2021theory, fissler2025elicitability}. Such VaR quantile formulations have been widely adopted in infrastructure systems, including power grids and transportation networks, where occasional violations may be acceptable while their frequency must be tightly controlled 
\cite{bienstock2014chance, kang2014value}. This risk treatment, by allowing violations,
is clearly different from RO. 
However, as VaR limits only the frequency but not the magnitude of losses, Conditional Value-at-Risk (CVaR) is introduced to account for the expected severity of losses beyond the VaR quantile \cite{rockafellar2000optimization, rockafellar2002conditional}, and thus incorporates a more informative description of tail behavior. For this reason, CVaR-based formulations are commonly adopted in applications like portfolio optimization when decision makers seek to account for both the likelihood and the magnitude of possible extreme losses.

In many applications, however, the underlying distribution cannot be specified exactly, making it ineffective or inappropriate to rely on a fixed distribution for decision making. Yet, we often have partial information available, taking the form of support bounds, moment estimates, or qualitative structural properties. In such settings, distributionally robust approaches provide a natural extension of probabilistic modeling by replacing the nominal distribution with an \textit{ambiguity set}, a set of all plausible distributions consistent with the available information. From the risk treatment perspective, distributionally robust approaches can be understood as placing a worst-case risk measure \cite{FollmerSchied2016, shapiro2017distributionally} on top of a classical risk measure evaluated under each distribution in the ambiguity set. For example, consider the worst-case expected value over all distributions in the ambiguity set.   
This paradigm, by leveraging available (and incomplete) probabilistic structure,  allows decision makers to hedge against distributional misspecification in a more effective fashion.  

If the aforementioned worst-case expected value treatment is adopted,  this leads to the mainstream DRO formulation \cite{scarf1957min,delage2010distributionally, wiesemann2014distributionally}. It imposes restrictions on the worst-case expected values of the left-hand-side (LHS) expressions.
Differently, when the risk treatment is reflected by a quantile-based consideration, the decision-making model naturally becomes a distributionally robust chance-constrained program (DRCCP), which is tailored for high-stakes settings where the reliability of the chance constraint must be guaranteed even when the underlying distribution is inaccurately estimated \cite{calafiore2006distributionally,xie2018deterministic,chen2024data}.
In many such settings, DRCCP is further relaxed or approximated through CVaR-based reformulations, leading to distributionally robust CVaR (DR-CVaR) models that optimize or constrain the worst-case average tail risk over the ambiguity set \cite{deo2025design,ren2024distributionally}. Clearly, when the ambiguity set collapses to a singleton distribution, DRCCP reduces to the classical CCP, and DR-CVaR generalizes  the classical CVaR formulation.

\subsection{A Discussion on Risk Treatments}

Although the preceding risk treatments are widely used, their relative conservatism and modeling implications are not always clear from their definitions. For a fixed uncertain quantity, the feasible region can change depending on whether the model evaluates average performance, controls tail events, or protects against the worst admissible realization. A systematic comparison therefore helps clarify how these paradigms relate.
\begin{example}
To facilitate a comparison between risk treatments, we consider the single-variable optimization problem:
\begin{equation}
\max \{ cx : \rho(a) x \leq b, x \in \mathbb{R}_+ \}
\label{eq_demo}
\end{equation}
where $a$ is a random parameter and $b,c>0$. In Table~\ref{tab:risk_treatments}, $\rho(a)$ denotes the risk treatment, and $\eta \in \mathbb{R}$ is an auxiliary threshold variable in the definitions where it appears. We assume that $a$ follows a fixed discrete distribution on the sample space $\Omega = \{1, \dots, 9\}$ with probability mass function $p=$(0, 0.05, 0.1, 0.2, 0.3, 0.2, 0.1, 0.05, 0).

\begin{table}[h]
\centering
\caption{Definitions of risk treatments $\rho(a)$}
\label{tab:risk_treatments}
\begin{tabular}{lll}
\toprule
\textbf{Risk Treatment} & \textbf{Notation} & \textbf{Mathematical Definition} \\
\midrule
Supremum (Worst-case) 
& $\sup(a)$ 
& $\sup_{\omega \in \Omega} a(\omega)$ \\
\addlinespace
Essential Supremum 
& $\operatorname{ess.sup}(a)$ 
& $\inf\{\eta \in \mathbb{R} : \mathbb{P}(a > \eta)=0\}$ \\
\addlinespace
Conditional VaR 
& $\operatorname{CVaR}_{1-\epsilon}(a)$ 
& $\inf_{\eta \in \mathbb{R}} \left\{ \eta + \frac{1}{\epsilon}\mathbb{E}[(a-\eta)^+] \right\}$ \\
\addlinespace
Value-at-Risk 
& $\operatorname{VaR}_{1-\epsilon}(a)$ 
& $\inf\{\eta \in \mathbb{R} : \mathbb{P}(a \le \eta) \ge 1-\epsilon\}$ \\
\addlinespace
Expectation 
& $\mathbb{E}[a]$ 
& $\sum_{\omega \in \Omega} a(\omega)\mathbb{P}(\omega)$ \\
\bottomrule
\end{tabular}
\end{table}

Since the constraint in \eqref{eq_demo} takes the form $\rho(a)x \le b$, a larger value of $\rho(a)$ directly tightens the constraint and reduces the feasible region of $x$. Therefore, risk treatments that yield larger values of $\rho(a)$ correspond to higher levels of conservatism in the optimization problem.

Figure~\ref{Risk Measure Comparison} provides a schematic illustration of these risk treatments and the spectrum of conservatism, ranging from the highly conservative almost-sure protection to the risk-neutral expectation-based one. Here, the shaded regions indicate which portion of the distribution is relevant to each risk operator. At the most conservative end, the supremum operator is independent of the probability distribution and depends only on the boundary of the sample space. In this discrete example, the essential supremum instead depends on the largest value with positive probability mass, regardless of how probability is distributed elsewhere. Moving along the spectrum, the VaR operator depends only on the quantile cutoff and is insensitive to the distribution's shape beyond that threshold. In contrast, CVaR and expectation both aggregate values over a range of outcomes, and therefore depend on the shape of the distribution.

\begin{figure}[h!]
    \centering
    \includegraphics[width=0.9\textwidth]{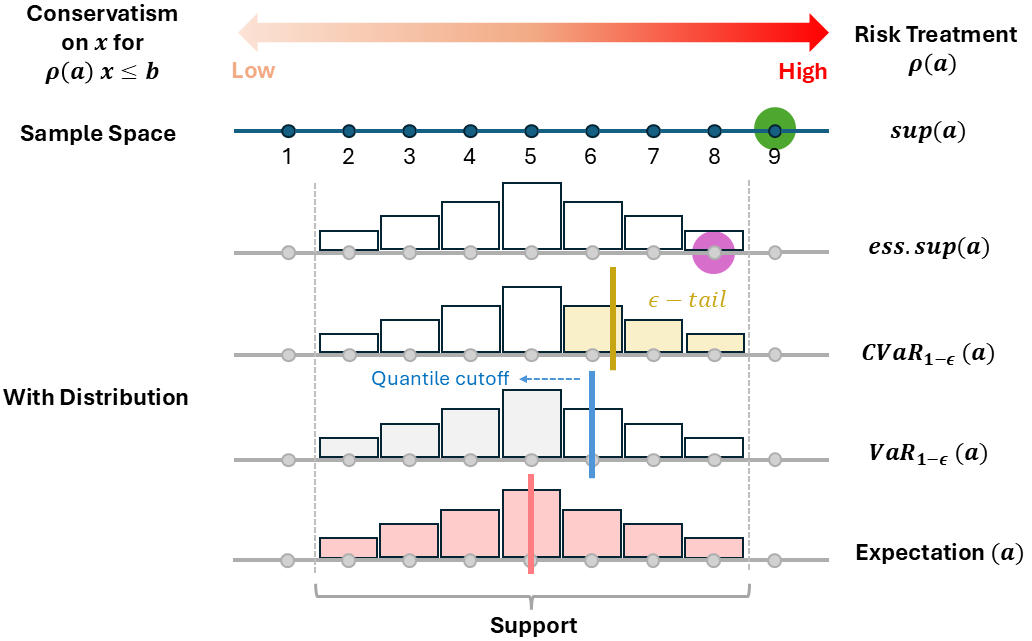}
    \caption{Schematic illustration of different risk treatments $\rho(a)$ over $\Omega$.}
    \label{Risk Measure Comparison}
\end{figure}
\end{example}

We note that Figure~\ref{Risk Measure Comparison} also helps clarify the distinction between two supremum-based risk treatments: the supremum protection, imposed over the entire sample space without distributional information, and the essential supremum (ess.sup) protection, which depends on the underlying probability distribution \cite{shapiro2021lectures}. To the best of our knowledge, this distinction has not been made explicit in the literature, partly because the sample space is often implicitly identified with the support of the distribution (which is indeed valid in some contexts). For instance, a CCP is sometimes described as resembling an RO model as $\epsilon \downarrow 0$ \cite{erdougan2006ambiguous}. That interpretation requires almost-sure feasibility to coincide with pointwise feasibility on the modeled sample space. Nevertheless, as illustrated in Figure~\ref{Risk Measure Comparison}, although points 1 and 9 belong to the sample space, they have zero probability; therefore, the support of this distribution consists only of points 2 through 8, a proper subset of the sample space.  Consequently, the supremum protection is strictly more conservative than the essential supremum protection.

When extending from a single distribution to an ambiguity set, this distinction carries over to the relationship between DRO and RO. 
One might expect DRO to reduce to RO when the ambiguity set becomes sufficiently simple (e.g., only enforcing normalization). Such a reduction is guaranteed when the ambiguity set contains all Dirac measures over the entire sample space. By contrast, an almost-sure requirement may protect only the smaller effective support induced by the admissible distributions. 
For clarity, we refer to optimization under essential-supremum protection (with respect to an ambiguity set) as the \emph{almost-sure DRO} (AS-DRO), and the currently mainstream formulation, i.e., optimization under worst-case expected-value protection, as \emph{standard DRO}. 
Almost-sure feasibility under ambiguity has previously appeared as a recourse-feasibility condition in two-stage DRO \cite{lu2024two}; here, we study this protection criterion as a standalone DRO model class.
AS-DRO can differ from RO when a nonempty violation event has probability zero under every admissible distribution. Additional support or mass restrictions, including local information, may make the effective support induced by the ambiguity set a strict subset of the sample space. This phenomenon will be discussed in detail in Section~\ref{sec:AS_DRO}, with further implications for local-information-based ambiguity sets presented in Section~\ref{sec:local_info}.

Figure~\ref{Concept} provides a high-level summary of the main approaches for modeling uncertainty. The figure organizes these methods along two dimensions: (i) the underlying risk treatment (expectation-based, tail-based, or supremum-based) and (ii) the amount of distributional information available (e.g., a pure uncertainty set, a fixed known distribution, or a family of plausible distributions). This organization highlights fundamental differences among those paradigms and clarifies how they relate to one another. Note that  \emph{Almost-Sure Stochastic Programming} (AS-SP) is AS-DRO when the ambiguity set reduces to a fixed known distribution.  

\begin{figure}[h!]
    \centering
    \includegraphics[width=0.75\textwidth]{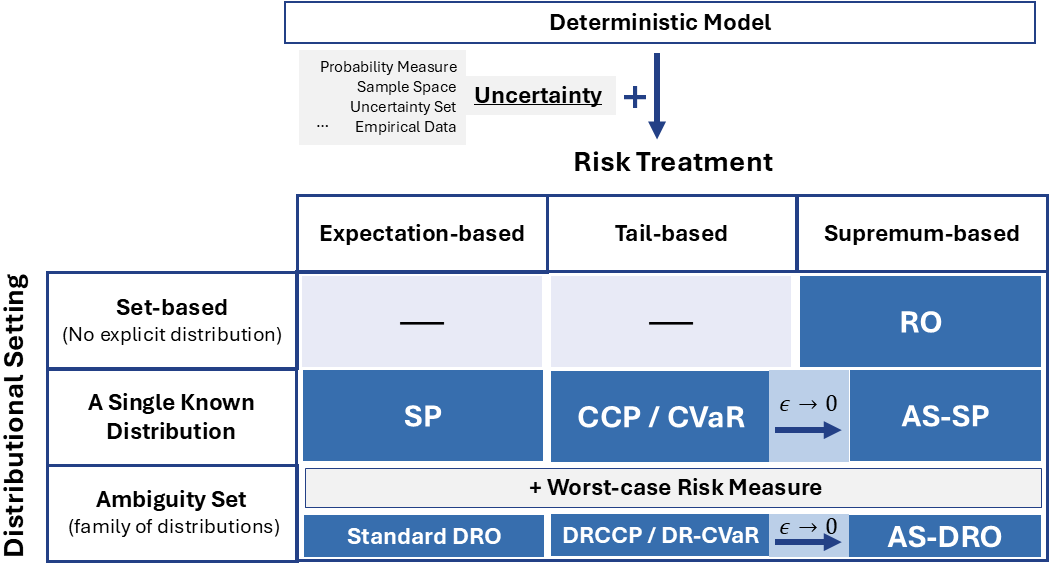}
    \caption{Taxonomy of optimization approaches for modeling uncertainty}
    \label{Concept}
\end{figure}

\subsection{Exact Solution Schemes of DROs}

Beyond the question of how uncertainty should be modeled, an equally fundamental issue is how the resulting models can be solved exactly.  Note that standard DRO and DRCCP have been formulated as bilevel $\min$--$\max$ (or, more precisely, $\min$--$\sup$) problems: the ``outer'' minimization chooses a decision, while the ``inner'' maximization selects a worst-case distribution from an ambiguity set. This bilevel structure is the core algorithmic challenge. In the following, we review methods developed for computing standard DRO and DRCCP problems, mainly focusing on exact ones.

From a theoretical perspective, the inner maximization problem in standard DRO is closely related to the classical generalized moment problem and semi-infinite optimization \cite{wiesemann2014distributionally, shapiro2017distributionally}. A predominant approach to address this challenge is to exploit strong duality and replace the inner optimization with an equivalent dual representation \cite{shapiro2017distributionally,wiesemann2014distributionally,kuhn2025distributionally}. Depending on the structure of the ambiguity set — such as moment-based sets, $\phi$-divergence models, Wasserstein metric–based sets, and event-based descriptions — the inner maximization problem admits a strong dual representation \cite{delage2010distributionally, wiesemann2014distributionally, chen2020robust}. When the resulting dual constraints are tractable, this converts the bilevel min--sup model into an equivalent single-level optimization problem. For instance, moment-based ambiguity sets typically lead to semidefinite or conic reformulations through dualizing moment constraints \cite{delage2010distributionally,bertsimas2005optimal,jiang2023optimized}, while $\phi$-divergence models yield convex programs by exploiting conjugate duality of divergence functions \cite{ben2013robust, shapiro2017distributionally}. Similarly, Wasserstein-based DRO exploits duality in optimal transport (i.e., Kantorovich duality) to reformulate the worst-case expectation problem into a tractable optimization problem involving transportation-cost penalization \cite{mohajerin2018data, gao2023distributionally, blanchet2019quantifying}. Because strong duality is the cornerstone, these reformulations typically require regularity conditions such as tail conditions, bounded moments, or appropriate continuity and integrability conditions \cite{kuhn2019wasserstein, xu2018distributionally, gao2023distributionally}. We note that this duality-based reformulation paradigm is currently the dominant and mainstream approach for solving standard DRO models in the literature \cite{kuhn2025distributionally, rahimian2022frameworks}. In many cases, the resulting reformulations are convex programs that can be solved efficiently using existing optimization solvers.

As with standard expectation-based DRO, tail-based DRCCPs constitute another well-established line of research involving ambiguity sets, despite the nonlinear and nonconvex nature of probabilistic constraints \cite{kuccukyavuz2022chance}. In these models, the inner maximization problem seeks a worst-case probability distribution that maximizes the probability of constraint violation, thereby coupling distributional ambiguity with nonconvex chance constraints and further complicating the bilevel structure. A common approach approximates chance constraints with CVaR, which provides a convex surrogate and often leads to tractable reformulations \cite{rockafellar2000optimization, nemirovski2007convex}. Moreover, for individual DRCCPs under second-order moment and support information, the worst-case CVaR approximation is exact when the constraint function is concave or quadratic \cite{zymler2013distributionally}. Beyond CVaR reformulations, under certain classes of ambiguity sets, such as moment-based sets and Wasserstein balls, DRCCPs admit exact reformulations through a combination of duality theory and mixed-integer modeling techniques \cite{zymler2013distributionally, xie2018deterministic, xie2021distributionally, ji2021data, shen2025convex}. 
Some Wasserstein-based reformulations introduce binary variables to encode violation events and yield exact mixed-integer conic models under specific uncertainty and decision structures \cite{xie2021distributionally, ji2021data}.

In addition to the aforementioned dual-based reformulation strategies, which aim to eliminate the infinite-dimensional inner problem in standard DRO and DRCCP models, duality has also been employed in developing decomposition algorithms for these models; such algorithms are generally approximation methods (e.g., \cite{xu2018distributionally} for standard DRO and \cite{jiang2022also, jiang2024also} for DRCCPs). More recent direct approaches include a probability cut framework for a structured assignment problem and a finitely convergent decomposition method under finite discrete support \cite{wang2022solution, pathy2025decomposition}. Overall, we emphasize that duality in convex programs serves as the primary driver behind existing solution approaches for standard DRO and DRCCPs. This stands in contrast to the methodological literature on RO, where both dual-based reformulations and primal-based decomposition methods are widely used due to their different strengths and applicability.

\subsection{Our Research Contributions} 
Distinct from the existing literature on DRO, we investigate several critical problems concerning its fundamental concepts and solution methods. This study aims to develop a deeper and more intuitive understanding of general DRO models, provide stronger and less conservative modeling tools suitable for practical data-driven contexts, and significantly enhance computational capabilities for solving complex instances.  Our contributions fall into two main categories, \emph{conceptual} and \emph{methodological}.

On the conceptual side, we introduce AS-DRO to bridge a structural gap between the underlying sample space and the effective support induced by probability mass allocation. This perspective highlights a form of worst-case protection that is not captured by existing expectation-based or tail-based DRO formulations. We further study DRO problems directly from a primal viewpoint, without requiring an ambiguity-set-specific compact dual reformulation of the full DRO model. This primal perspective provides a unified interpretation of worst-case protection and clarifies the structural connections among RO, DRO, AS-DRO, and DRCCPs. Building on this primal perspective, we also present a systematic classification of DRCCP formulations based on their protected performance criteria and ambiguity structures, which facilitates unified modeling and algorithmic development.

On the methodological side, we develop an exact primal-based decomposition framework for solving DRO problems, inspired by cutting-plane methods in robust optimization. The proposed algorithm operates mainly in the primal space and is supported by rigorous convergence guarantees. A key feature of the framework is its extensibility: with minimal structural modifications, it applies to AS-DRO, DRCCPs, and DRO models with ambiguity sets strengthened by additional local information, which can be used to capture a wide range of distributions that otherwise could not be well characterized by classical ambiguity sets. 

Our extensive numerical experiments show that, on the tested instances, the proposed primal-perspective framework applies across the considered models and, where applicable, substantially outperforms the examined duality-based reformulations.  Overall, the presented work provides an intuitive and straightforward treatment of distributionally robust approaches, 
making these approaches easier to implement for robust and data-driven decision making in practice.

\section{Standard DRO} \label{sec:DRO}

Consider the standard DRO problem as follows:
\begin{subequations}\label{eq_DRO}
\begin{align}
    \min_{\bm x \in \mathcal{X}} & \quad \bm c^\top \bm x\\
    \text{s.t.} & \quad \sup_{\mathbb P_m \in \mathcal{P}_m} \mathbb{E}_{\mathbb{P}_m}[f_m(\bm a_m, \bm{x})] \leq b_m, \quad \forall m \in [M] \label{DRO_constraint}
\end{align}
\end{subequations}
Here, \( \bm{x} \in \mathcal{X} \subseteq \mathbb{R}^n \) represents the decision variable, and $M$ is the number of constraints with $[M]=\{1,\dots,M\}$. In constraint $m \in [M]$, \( \bm a_m \) denotes a random vector on a measurable space \((\mathcal{A}_m, \mathcal{F}_m)\), where \( \mathcal{A}_m \subseteq \mathbb{R}^n  \) is the sample space and \( \mathcal{F}_m = \mathcal{B}(\mathcal{A}_m) \) is its Borel \(\sigma\)-algebra. The function \( f_m(\bm a_m, \bm{x}) \) is a prescribed constraint function.  
The vectors \( \bm{c} \in \mathbb{R}^n \) in the objective function and \( \bm{b} \in \mathbb{R}^M \) on the RHS are assumed to be fixed problem parameters. This assumption is made without loss of generality, noting that the objective function with uncertain coefficients can be treated as a  constraint through the epigraph reformulation technique and an uncertain RHS can be moved to LHS by introducing a new decision variable. Assume that the unknown probability distribution \( \mathbb{P}_m \) governing \( \bm a_m \) resides in an ambiguity set \( \mathcal{P}_m \), a collection of all plausible probability distributions of constraint $m$. In our study, we adopt the following generalized moment ambiguity set for all $\mathcal P_m$s, noting that this structure subsumes moment-based ambiguity sets and lifted Wasserstein ambiguity sets with finite or finitely representable nominal distributions \cite{delage2010distributionally, mohajerin2018data}.

\begin{definition} \label{ambiguity_set_define} The generalized moment ambiguity set \(\mathcal{P}\) is defined as:
\begin{align}
    \mathcal{P} := \left\{ \mathbb{P} \in \mathcal{M}_+(\mathcal{A}, \mathcal{F}) \,\middle|\, 
    \mathbb{P}(\mathcal{A}) = 1, \mathbb{E}_{\mathbb{P}} \left[ g_t^-(\bm a) \right] < \infty,
    \mathbb{E}_{\mathbb{P}} \left[ g_t(\bm a) \right] \leq \gamma_t, \forall t \in [T] \right\} \label{p_gm_ambiguity},
\end{align}
Here, \(\mathcal{M}_+(\mathcal{A}, \mathcal{F})\) denotes the set of nonnegative measures on \((\mathcal{A}, \mathcal{F})\).
The function \(g_t(\cdot): \mathcal{A} \to \mathbb{R}\) is a measurable mapping that specifies a moment condition indexed by \(t \in [T]\), where \(g_t^+ := \max\{g_t,0\}\), \(g_t^- := \max\{-g_t,0\}\), and \(\gamma_t \in \mathbb{R}\). For every \(\mathbb{P} \in \mathcal{P}\), the finite negative part and finite upper moment bound imply
\(\mathbb{E}_{\mathbb{P}}[g_t^+] = \mathbb{E}_{\mathbb{P}}[g_t] + \mathbb{E}_{\mathbb{P}}[g_t^-] \leq \gamma_t + \mathbb{E}_{\mathbb{P}}[g_t^-] < \infty\), and hence \(\mathbb{E}_{\mathbb{P}}[|g_t|] < \infty\).
\end{definition}

When Definition~\ref{ambiguity_set_define} is instantiated for constraint $m\in[M]$, we write $T_m$ for the number of moment conditions and $\gamma_{mt}$ for the bound associated with $g_t$, $t\in[T_m]$.

\begin{remark}[Generalized moment framework as a unified target]
\label{rem:regularity-unification}
A key motivation for adopting the generalized moment form in Definition~\ref{ambiguity_set_define} is that, by suitable choices of $g_t$ and $\gamma_t$, it subsumes or serves as the reformulation target for the principal families of ambiguity sets in the DRO literature.
\begin{itemize}
    \item \emph{Moment-based sets} \cite{delage2010distributionally, wiesemann2014distributionally}: directly an instance of Definition~\ref{ambiguity_set_define}, with $g_t$ encoding mean, covariance, or higher-order moment bounds.
    \item \emph{Wasserstein sets} \cite{mohajerin2018data, kuhn2019wasserstein}: through the standard lifting, a Wasserstein ball with a finite or finitely representable nominal distribution can be represented by joint distributions $\pi$ on $\mathcal A\times\mathcal A$ whose second marginal is the nominal distribution and whose transport-cost moment is bounded by $\theta^p$. The resulting finite marginal constraints and transport-cost bound fit Definition~\ref{ambiguity_set_define}; for a general nominal distribution, the fixed-marginal condition need not have a finite representation.
\end{itemize}
\end{remark}

Building on the ambiguity set in Definition~\ref{ambiguity_set_define}, we now introduce a mild regularity condition that is maintained throughout the remainder of the paper. Together with the integrability built into Definition~\ref{ambiguity_set_define}, it guarantees 
that the expectations entering the DRO problem~\eqref{eq_DRO} are well defined under each $\mathbb{P}_m \in \mathcal{P}_m$, and it underpins both the dual reformulation in Section~\ref{sec:dro-dual} and the primal decomposition in Section~\ref{BiCS section}. Conditions of this type are 
standard in the DRO literature \cite{fournier2015rate, kuhn2019wasserstein, 
gao2023distributionally, wang2025sinkhorn}.

\begin{assumption}[Regularity of the generalized moment ambiguity set]
\label{assumption:regularity}
For each $m \in [M]$, the ambiguity set $\mathcal{P}_m$ is nonempty and 
satisfies the following two conditions:
\begin{itemize}
    \item[\textup{(i)}] \textbf{Closed sample space.} 
    The sample space $\mathcal{A}_m \subseteq \mathbb{R}^n$ is nonempty and closed, the objective kernel $f_m(\cdot, \bm{x})$ is upper semicontinuous on $\mathcal{A}_m$ for every $\bm{x} \in \mathcal{X}$, and each moment function $g_t : \mathcal{A}_m \to \mathbb{R}$, $t \in [T_m]$, is lower 
    semicontinuous.
    \item[\textup{(ii)}] \textbf{Controlled growth.} 
    For every $\bm{x} \in \mathcal{X}$, the cost kernel $f_m(\cdot,\bm x)$ is dominated by the moment functions in the sense that
    \[
        |f_m(\bm a_m,\bm x)| \le \beta_{0m}(\bm x) + 
        \sum_{t \in [T_m]} \beta_{tm}(\bm x)\, |g_t(\bm a_m)|, \quad \forall \bm a_m \in \mathcal{A}_m,
    \]
    for some scalars $\beta_{0m}(\bm x) \in \mathbb{R}$ and 
    $\beta_{tm}(\bm x) \ge 0$.
    Thus, for every $\mathbb{P}_m \in \mathcal{P}_m$,
    \[
        \mathbb{E}_{\mathbb{P}_m}[|f_m(\bm a_m,\bm x)|]
        \le \beta_{0m}(\bm x) + \sum_{t \in [T_m]} \beta_{tm}(\bm x)\,
        \mathbb{E}_{\mathbb{P}_m}[|g_t(\bm a_m)|] < \infty.
    \]
\end{itemize}
\end{assumption}

With the formulation in place, we next present a couple of solution schemes for (\ref{eq_DRO}). They actually can be seen as natural extensions of their counterparts developed for solving RO problems.

\subsection{A Revisit of Solution Schemes for RO}
\label{subsect_RO}
Before presenting solution methods for DRO, we briefly revisit two classical approaches for solving RO: dual-based reformulation and primal-based decomposition. Note that this helps strengthen our understanding of the connection between RO and DRO and allows us to appreciate the development of their solution methods from a unified perspective.

The general RO model is:
\begin{subequations}\label{eq_RO}
\begin{align}
    \textbf{(RO)} \quad \min_{\bm{x} \in \mathcal{X}} \quad & \bm{c}^\top \bm{x} \\
    \text{s.t.} \quad & \sup_{\bm{a}_m \in \mathcal{A}_m} f_m(\bm{a}_m, \bm{x}) \;\;\leq\; b_m, 
    \quad \forall m \in [M]. \label{RO_constraint}
\end{align}
\end{subequations}
Currently, two main classes of solution methods are widely used for solving RO instances, i.e., dual-based reformulation and primal-based decomposition approaches \cite{ben1999robust,bertsimas2004price,bertsimas2016reformulation}. 

\noindent $(i)$ \textbf{Dual-based Reformulation Method:}
A foundational result for RO is that strong duality (with respect to certain conic uncertainty sets) can be leveraged to reformulate \eqref{RO_constraint}, which contains an inner worst-case maximization, into an equivalent regular constraint. 

\begin{theorem}[Adapted from \cite{ben1999robust}]
\label{linear-robust}
Consider a proper cone $\mathcal{K}_m$ with dual cone $\mathcal{K}_m^*$. 
Suppose that $f_m(\bm a_m, \bm x) = \bm a_m^\top \bm x$ and 
$\mathcal{A}_m = \{\bm a_m \in \mathbb{R}^n : \bm D_m \bm a_m 
\preceq_{\mathcal{K}_m} \bm d_m\}$ for $m \in [M]$, and that each 
$\mathcal{A}_m$ satisfies the conic Slater condition: there exists 
$\bm a_m^0 \in \mathbb{R}^n$ with $\bm d_m - \bm D_m \bm a_m^0 \in 
\mathrm{int}(\mathcal{K}_m)$. Then the RO problem~\eqref{eq_RO} is equivalent to the following optimization problem in $\bm x$ and $\{\bm p_m\}_{m\in[M]}$:
\begin{subequations}
\label{eq_RO_duality_reform}
\begin{align}
    \min_{\bm{x} \in \mathcal{X}} \quad & \bm{c}^\top \bm{x} \\
    \text{s.t.} \Big\{& \bm{p}_m^\top \bm{d}_m \le b_m, \; 
                       \bm{D}_m^\top \bm{p}_m = \bm{x}, \; 
                       \bm{p}_m \in \mathcal{K}_m^*\Big\} 
                     \quad \forall m \in [M].
\end{align}
\end{subequations}
\end{theorem}

\emph{Proof.} See Appendix~\ref{app:proof-thm32}.

\begin{remark}
When \(\mathcal{K}_m\) is a polyhedron or belongs to a tractable convex cone (e.g., polyhedral, second-order, or semidefinite), the reformulation \eqref{eq_RO_duality_reform} amounts to evaluating the support function of \( \mathcal{A}_m \), leading to a tractable convex optimization problem. In particular, a polyhedral uncertainty set yields a linear programming reformulation, while second-order and semidefinite uncertainty sets lead to conic programs of the corresponding type, which can be solved efficiently using standard interior-point methods \cite{nesterov1994interior}.
\end{remark}

\noindent $(ii)$ \textbf{Primal-based Decomposition Method:}
As an alternative strategy, we directly consider the primal form of \eqref{RO_constraint} and explicitly hedge against particular scenarios in $\mathcal{A}_m$. Through an iterative master-subproblem procedure, the infinite collection of robust constraints contained in \eqref{RO_constraint} can be approximated by a finite subset that is iteratively refined \cite{mutapcic2009cutting}. Particularly, given a finite set of scenarios \(\hat{\mathcal{A}}_m \subseteq \mathcal{A}_m\), we construct and solve the following \textit{master problem} (MP), which enforces only the constraints defined by $\hat{\mathcal{A}}_m$. Note that it is a relaxation of the original RO model in \eqref{eq_RO}.

\[
\begin{aligned}
\textbf{MP-RO:} \quad
&\min_{\bm x \in \mathcal{X}} \quad  \bm{c}^\top \bm{x} \\
&\text{s.t.} \quad f_m(\bm{a}_{mk}, \bm{x}) \;\le\; b_m, \quad \forall \bm{a}_{mk} \in \hat{\mathcal{A}}_m,\; m \in [M], 
\end{aligned}
\]
Let $\bm x^*$ denote an optimal solution obtained from solving MP. We next solve the following 
\textit{subproblem} (SP)  to identify a scenario $\bm{a}_m \in \mathcal{A}_m$ that is violated by $\bm x^*$ for $m\in [M]$. 
\[
\begin{aligned}
\textbf{SP-RO}_m: \quad
&\Omega_m (\bm x^*)
=\; \sup_{\bm{a}_m\in \mathcal{A}_m} \; f_m(\bm{a}_m, \bm{x}^*)
\end{aligned}
\]
If $\Omega_m (\bm x^*)>b_m$, choose any $\bm{a}^*_m\in\mathcal A_m$ satisfying $f_m(\bm{a}^*_{m},\bm{x}^*)>b_m$, and augment \(\hat{\mathcal{A}}_m\) by including $\bm{a}^*_m$. After completing an exact sweep over all subproblems, if some 
$\hat{\mathcal{A}}_m$s are updated, we resolve the master problem (MP) to start a new iteration. Otherwise, $\Omega_m(\bm x^*)\le b_m$ for every $m\in[M]$, and the current solution $\bm{x}^*$ exactly solves RO in \eqref{eq_RO}.


\begin{remark}
We note that this primal-based decomposition method is rather general, as it does not require the convexity of $\mathcal{A}_m$. Empirical studies indicate that the relative performance of reformulation-based and decomposition-based methods depends on factors such as problem size, structural properties, and solver capabilities \cite{mutapcic2009cutting,fischetti2012cutting,bertsimas2016reformulation}. Yet, the decomposition method is often preferred for large-scale problems due to its computational scalability. Moreover, from a primal perspective, it can explicitly construct worst-case realizations, offering clearer insight than the dual variables in the reformulation, which are often harder to interpret.
\end{remark}

\subsection{Dual-based Reformulation for Standard DRO}\label{sec:dro-dual}

In the existing literature on standard DRO, the dual-based reformulation is arguably the dominant strategy. As shown in~\eqref{eq_DRO} and~\eqref{eq_RO}, both formulations contain an inner worst-case maximization that is suitable for applying the dual-based reformulation technique described in the previous subsection. However, DRO introduces a key structural distinction: the inner worst-case expectation involves an additional optimization over probability measures, which gives rise to an extra layer of duality that is absent in RO. Hence, a ``bi-dual'' construction is naturally needed, dualizing over the probability measure as well as the sample space. In the current literature, this is usually done sequentially, dualizing the probability measure first and then the sample space, though the two dualizations are in principle interchangeable and the order is more a matter of convention. We next present a dual-based reformulation that captures the core structural elements of convex standard DRO, and then discuss the differences and connections between the reformulation techniques for RO and standard DRO.

For notational convenience, let $\bm{e}_{mt} \in \mathbb{R}^n$ denote the moment coefficient vector for the $t$-th moment function of constraint $m$. Let $\bm{E}_m := (\bm{e}_{m1}, \ldots, \bm{e}_{mT_m})^\top \in \mathbb{R}^{T_m \times n}$ stack these coefficient vectors as the rows 
of $\bm{E}_m$, and let $\bm{\gamma}_m := (\gamma_{m1}, \ldots, 
\gamma_{mT_m})^\top \in \mathbb{R}^{T_m}$ collect the corresponding 
moment bounds for constraint $m$.

\begin{theorem}[Adapted from \cite{wiesemann2014distributionally}]
\label{linear-dro} 
Suppose that $f_m(\bm a_m, \bm x) = \bm a_m^\top \bm x$, 
$g_t(\bm a_m) = \bm e_{mt}^\top \bm a_m$, and 
$\mathcal{A}_m = \{\bm a_m \in \mathbb{R}^n : \bm D_m \bm a_m 
\preceq_{\mathcal{K}_m} \bm d_m\}$, where $\mathcal{K}_m$ is a proper 
cone with dual cone $\mathcal{K}_m^*$. Assume further that, for each 
$m \in [M]$, (i) each $\mathcal{A}_m$ satisfies the conic Slater 
condition, i.e., there exists $\bm{a}_m^0 \in \mathbb{R}^n$ with 
$\bm d_m - \bm D_m \bm{a}_m^0 \in \mathrm{int}(\mathcal{K}_m)$, and 
(ii) $\mathcal{P}_m$ is strictly feasible, i.e., it contains a 
probability distribution under which every moment constraint holds 
strictly. Then, the standard DRO in~\eqref{eq_DRO} is equivalent to the following optimization problem in $\bm x$, $\{\bm p_m\}_{m\in[M]}$, and $\{\bm\beta_m\}_{m\in[M]}$:
\begin{subequations}
\label{eq_DRO_reform}
    \begin{align}
        \min_{\bm x \in \mathcal{X}} & \quad \bm{c}^\top \bm{x} \\
        \text{s.t.} \ & \Big\{\bm \beta_m^\top \bm \gamma_m + 
            \bm p_m^\top \bm d_m \leq b_m, \ \bm D_m^\top \bm p_m = 
            \bm x - \bm E_m^\top \bm \beta_m, \ \bm p_m \in 
            \mathcal{K}_m^*, \ \bm \beta_m \in \mathbb{R}_+^{T_m}\Big\}, 
            \ \forall m \in [M].
    \end{align}
\end{subequations}
\end{theorem}
\emph{Proof.} See Appendix~\ref{app:proof-thm34}.

\begin{remark}
    When $\mathcal{K}_m$ is a standard tractable cone (e.g., polyhedral, second-order, or semidefinite) and $\mathcal X$ also has a polynomial-size representation over the listed tractable cones, the reformulation \eqref{eq_DRO_reform} can be solved to prescribed accuracy by interior-point methods under standard conic-complexity assumptions. In particular, it reduces to an LP if both $\mathcal{K}_m$ and $\mathcal X$ are polyhedral.
\end{remark}

Theorem~\ref{linear-dro} can be viewed as a streamlined counterpart of the broader one in \cite{wiesemann2014distributionally}, which accommodates constraint functions that are convex and piecewise affine in both the decision variables and the random parameters. Ours, on the other hand, focuses on a cleaner setting with linear costs, linear moment functions, and a conic-representable sample space, to make the relationship between standard DRO and RO reformulations easier to see. By comparing \eqref{eq_DRO_reform} and  \eqref{eq_RO_duality_reform}, we note that the reformulation of standard DRO indeed shares a significant similarity to its RO counterpart. Yet, 
we also note that there is a clear difference due to the extra dual vector \( \boldsymbol{\beta} \) in the DRO's reformulation in \eqref{eq_DRO_reform},  which is introduced for dualizing the probability measure. A rather detailed comparison between reformulations of RO and DRO is provided in  Table~\ref{table:RO_DRO_comparison}.

\begin{table}[t]
\centering
\caption{Comparison between RO and Standard DRO}
\label{table:RO_DRO_comparison}
\small
\begin{tabular}{lcc}
\toprule
\textbf{Aspect} & \textbf{RO} & \textbf{DRO} \\
\midrule
Uncertainty set/sample space $\mathcal{A}_m$
& \multicolumn{2}{c}{$\{\bm a_m \in \mathbb{R}^n : \bm D_m \bm a_m \preceq_{\mathcal{K}_m} \bm d_m\}$} \\
Dual variables
& $\bm p_m$
& $\bm\beta_m,\, \bm p_m$ \\
Dualization
& Scenario dual
& Bi-dual (probability + scenario) \\
 Expression dualized over $\bm a_m$
& $f_m(\bm a_m,\bm x)$
& $f_m(\bm a_m,\bm x) - \bm\beta_m^\top g(\bm a_m)$ \\
\midrule
\multicolumn{3}{l}{\emph{Reformulation by $\mathcal{A}_m$ / $\mathcal{K}_m$}} \\
\quad Polyhedron ($\mathcal{A}_m = \{\bm a_m : \bm D_m \bm a_m \leq \bm d_m\}$) & LP & LP \\
\quad Polyhedral / second-order / semidefinite cone   & LP / SOCP / SDP & LP / SOCP / SDP \\
\bottomrule
\end{tabular}
\end{table}

\begin{remark}
A few structural observations follow from comparing the two reformulations.  First, RO dualizes only over the uncertainty set, whereas DRO requires a ``bi-dual'' process: first over the probability measure and then over the sample space. Second, the duality argument concerns the remaining maximization over $\bm a_m$: the relevant expression is $f_m(\bm a_m, \bm x)$ in RO and $f_m(\bm a_m, \bm x) - \sum_{t \in [T_m]} \beta_{mt} g_t(\bm a_m)$ in DRO. Thus, no convexity of the formulation in $\bm x$ is required for this inner dualization step. Finally, when $\mathcal{A}_m$ is discrete and finite, both reformulations simplify considerably: the RO constraints can be enumerated directly without invoking Theorem~\ref{linear-robust}, and the DRO reformulation can stop after the first dualization step over the probability measure, since all scenarios can then be listed explicitly.
\end{remark}

The aforementioned observations and comparison help clarify how RO and DRO reformulations relate to each other under comparable assumptions. Naturally, the bi-dual structure in the DRO reformulation may be subject to the same potential restrictions as the purely dual perspective in RO. While the reformulation is tractable for small-scale convex models, it becomes much more challenging in problems involving covariance terms or general bilinearities. Although modern solvers handle linear and conic programs efficiently, their performance typically deteriorates in these more demanding settings. As noted in our discussions on RO, and even more so in DRO, a key limitation of the dual perspective is its lack of intuitive interpretability. Specifically, the worst-case distribution is represented only implicitly through dual variables, which makes it difficult to explicitly characterize or analyze. In the next subsection, we take the primal perspective and develop a decomposition algorithm that often offers computational advantage and intuitive understanding.

\subsection{A Primal-based Decomposition Method for Standard DRO} \label{BiCS section}

As discussed earlier, RO and DRO share a similar worst-case structure. Given that dual reformulation techniques have been successfully extended from RO to the DRO setting, a natural question is whether primal-based methodology can also be extended from RO to DRO. Despite their structural similarity, the two paradigms differ fundamentally in how robustness is enforced: RO protects against worst-case \emph{scenarios}, whereas DRO protects against worst-case \emph{distributions}. This suggests that a genuine primal extension of cutting-plane methods to DRO should operate directly at the distribution level. 

The main challenge, however, is that when the sample space is continuous, the worst-case distribution is a probability measure defined over infinitely many points, and optimizing directly over such an infinite-dimensional space is generally computationally intractable. Theorem \ref{thm:discretize_unbounded} shows that, under the given assumptions, the expected value under any fixed distribution can be approximated arbitrarily closely by finite-support probability measures.

\begin{theorem}[Finite-support representation of an expected value]
\label{thm:discretize_unbounded}
Let $\mathcal{A}\subseteq\mathbb{R}^d$ be a measurable set, $h:\mathcal{A}\to\mathbb{R}$ a measurable function, and $\mathbb{P}$ a Borel probability measure on $\mathcal{A}$ with $\int_{\mathcal{A}}|h|\,d\mathbb{P}<\infty$. Then there exist a sequence of finite sets $\mathcal{A}_n^*=\{\bm a_1^{(n)},\ldots,\bm a_n^{(n)}\}\subseteq\mathcal{A}$ and nonnegative weights $\{w_k^{(n)}\}_{k=1}^n$ with $\sum_{k=1}^n w_k^{(n)}=1$ such that
\[
\mathbb{E}_{\mathbb{P}}[h(\bm a)]
=
\int_{\mathcal{A}}h(\bm a)\,\mathbb{P}(d\bm a)
=
\lim_{n\to\infty}\sum_{k=1}^n
w_k^{(n)}h\bigl(\bm a_k^{(n)}\bigr),
\]
or, equivalently, $\lim_{n\to\infty} \left| \mathbb{E}_{\mathbb{P}}[h(\bm a)] - \sum_{k=1}^n w_k^{(n)}h\bigl(\bm a_k^{(n)}\bigr) \right|
=0.$
\end{theorem}

\emph{Proof.} See Appendix~\ref{app:proof-thm3}.

\begin{remark}
When the sample space $\mathcal{A}_m$ is discrete and finite, Theorem~\ref{thm:discretize_unbounded} holds naturally, as any distribution is already supported on a finite set of scenarios.
\end{remark}

Theorem~\ref{thm:discretize_unbounded} provides the theoretical basis for approximating the expectation under any fixed distribution using a finite scenario set together with probability weights on those scenarios. This representation clarifies a structural point that is easily obscured in continuous formulations: a probability weight has meaning only as mass assigned to a specific scenario, so the scenario set must be available before weights can be meaningfully optimized. Existing decomposition approaches have long relied on reformulating the distributional layer and then solving the resulting finite or semi-infinite formulation by delayed constraint generation or scenario-indexed separation \cite{yang2018distributionally, luo2019decomposition}, thereby preserving a familiar RO-style iterative solution structure.

Our primal construction instead works directly with the finite-support representation and follows its natural order: scenarios first, weights second. This viewpoint leads to two design choices: how much support information is transferred from each subproblem solve and how the retained support and probability weights are represented in the master problem.

\subsubsection{Bilevel Cutting Set (BiCS) Algorithm}

Just as RO generates \emph{scenario cuts} to iteratively refine its protection against adversarial scenarios, BiCS generates \emph{distribution cuts} to iteratively refine its protection against adversarial distributions. In the primal representation suggested by Theorem~\ref{thm:discretize_unbounded}, a distribution cut naturally decomposes into two components: a finite scenario set, which serves as the core of the distribution on which probability mass is defined, and the probability weights assigned to those scenarios. The algorithm of a distribution cut therefore proceeds in two levels: the scenario set is identified first, and the weights are optimized on top of it. We call this the \textbf{Bilevel Cutting Set (BiCS)} framework.

For any finite $\hat A\subseteq\mathcal A_m$, define 
$\mathcal P_m(\hat A):=\{\mathbb P_m\in\mathcal P_m:\operatorname{supp}(\mathbb P_m)\subseteq\hat A\}$. At each iteration, the two components of a distribution cut are handled under a shared principle: the probability measure is treated as a decision variable constrained to lie in $\mathcal{P}_m$ at both levels. The \textbf{subproblem} \eqref{SubProblem}, given the  $\bm{x}^*$, searches over the full $\tilde{\mathbb{P}}_m \in \mathcal{P}_m$ and returns its supporting scenarios to enrich $\hat A_m$. The \textbf{master problem} \eqref{MasterProblem}, given the current $\hat A_m$, jointly optimizes $\bm{x}$ and probability weights in $\mathcal P_m(\hat A_m)$. Because the weights are re-optimized rather than frozen, the same $\hat A_m$ can support different worst-case distributions as $\bm{x}$ changes, yielding what we call a \emph{pooled parametric distribution cut}: scenarios as a persistent core, weights adapting to each candidate decision while remaining feasible for $\mathcal{P}_m$. We detail the framework in Algorithm~\ref{BiCSalgorithm} with convergence established in Theorem~\ref{convergence-BiCS}.

\begin{algorithm}[H]
\footnotesize
\setlength{\itemsep}{1pt}
\caption{\textbf{\textit{Bilevel Cutting Set (BiCS) Framework}}} \label{BiCSalgorithm}
\begin{algorithmic}
\State \textbf{Step 0:} Initialize the \emph{master problem (MP-DRO)} and set the optimality tolerance \( \epsilon>0 \).
\State \textbf{Step 1:} For each $m\in[M]$, select a finite initial set $\hat A_m=\{\bm a_{mk}:k=1,\ldots,K_m\}\subseteq\mathcal A_m$.
\State \textbf{Step 2:} \textbf{Solving Master Problem.} Solve the MP-DRO to obtain an optimal solution \(\bm x^*\):
\begin{subequations} \label{MasterProblem}
\begin{align} 
    \textbf{MP-DRO:}\quad \min_{\bm{x} \in \mathcal{X}} & \quad \bm{c}^\top \bm{x},\\[1ex]
    \text{s.t.} \quad & \max_{\mathbb P_m \in \mathcal{P}_m(\hat A_m)} \sum_{k=0}^{K_m} P_{mk}\, f_m(\bm{a}_{mk}, \bm{x}) \leq b_m,\hspace{2mm} \forall m \in [M].\label{inner_max_constraint}
\end{align}
\end{subequations}
\State \textbf{Step 3:} \textbf{Solving Subproblem.} For each $m$, use an $\epsilon/2$-oracle for the \emph{subproblem (SP-DRO)}:
\Statex \hspace{4em}\textbf{for each row $m\in[M]$ do}
\begin{align} \label{SubProblem}
    \textbf{SP-DRO:} \quad \Omega^*_m(\bm x^*) &= \sup_{\tilde{\mathbb{P}}_m \in \mathcal{P}_m} \; \mathbb{E}_{\tilde{\mathbb{P}}_m}\left[ f_m(\bm{a}_m,\bm x^*) \right]
\end{align}
\Statex \hspace{6em}\parbox[t]{\dimexpr\linewidth-6em\relax}{Let $\tilde{\mathbb P}_m^*\in\mathcal P_m$ be the feasible finitely supported distribution returned by the oracle, let $\tilde A_m^*:=\operatorname{supp}(\tilde{\mathbb P}_m^*)$, and let $\hat\Omega_m(\bm x^*)$ denote the associated oracle value.}
\Statex \hspace{6em}\textbf{if $\hat\Omega_m(\bm x^*)>b_m+\epsilon/2$, then} update $\hat A_m\gets\hat A_m\cup\tilde A_m^*$ and $K_m\gets|\hat A_m|-1$.
\State \textbf{Step 4:} Stopping criterion:
\Statex \hspace{4em}\textbf{if} $\max_{m\in[M]}\{\hat\Omega_m(\bm x^*)-b_m\}\le\epsilon/2$ \textbf{then}
\Statex \hspace{6em}Terminate and return $\bm x^*$.
\Statex \hspace{4em}\textbf{else}
\Statex \hspace{6em}Return to Step~2.
\end{algorithmic}
\end{algorithm}

Suppose problem~\eqref{eq_DRO} has a finite optimal value. We call $\bm x\in\mathcal X$ $\epsilon$-feasible if $\Omega_m^*(\bm x)\le b_m+\epsilon$ for all $m\in[M]$. For the candidate solutions generated by Algorithm~1 that are optimal to MP-DRO, $\epsilon$-feasibility is sufficient to establish $\epsilon$-optimality, consistent with the convention adopted in \cite{mehrotra2014cutting,luo2019decomposition}.

\begin{theorem}[Convergence of Algorithm~\ref{BiCSalgorithm}] 
\label{convergence-BiCS}
Let $\epsilon>0$. Suppose problem~\eqref{eq_DRO} has a finite optimal value, every invoked MP-DRO attains a finite optimum, and, whenever Step~3 is invoked, the oracle returns a feasible finite-support distribution $\tilde{\mathbb P}_m^*\in\mathcal P_m$ satisfying
\[
\Omega_m^*(\bm x)-\epsilon/2\le\hat\Omega_m(\bm x):=\mathbb E_{\tilde{\mathbb P}_m^*}[f_m(\bm a_m,\bm x)]\le\Omega_m^*(\bm x),
\qquad m\in[M].
\]
If Algorithm~\ref{BiCSalgorithm} terminates by the Step~4 stopping criterion, it returns an $\epsilon$-optimal solution to problem~\eqref{eq_DRO}. Moreover, if $\mathcal{X}$ is compact, 
$f_m$ is jointly continuous on $\mathcal{A}_m \times \mathcal{X}$, and the supporting scenarios of distributions added in Step~3 lie in a compact subset $\mathcal{C}_m \subseteq \mathcal{A}_m$ uniformly over generated incumbents, for each $m \in [M]$, then 
Algorithm~\ref{BiCSalgorithm} terminates in a finite number of iterations.
\end{theorem}

\begin{proof}
The proof has two parts. Part~(i) shows that any terminating incumbent is $\epsilon$-optimal, using the theorem's well-posedness premises and the Step~3 oracle condition. Part~(ii) rules out non-termination via a contradiction that draws on the additional compactness and continuity conditions.

\medskip
\noindent\textbf{Part (i): Correctness upon termination.}\;
Suppose the algorithm terminates at $\bm{x}^*$. By the Step~4 stopping criterion and the oracle condition, for every $m\in[M]$,
\[
    \Omega_m^*(\bm{x}^*) 
    \;\le\; \hat{\Omega}_m(\bm{x}^*) + \epsilon/2 
    \;\le\; b_m + \epsilon.
\]
Thus every worst-case constraint of~\eqref{eq_DRO} is satisfied to tolerance $\epsilon$. For the objective, $\mathcal P_m(\hat A_m)\subseteq\mathcal P_m$, so MP-DRO~\eqref{MasterProblem} is a relaxation of~\eqref{eq_DRO}; hence its finite optimum $\bm{c}^\top \bm{x}^*$ does not exceed the optimal value of~\eqref{eq_DRO}. Therefore $\bm{x}^*$ is $\epsilon$-optimal for~\eqref{eq_DRO}.

\medskip
\noindent\textbf{Part (ii): Finite termination.}\;
Suppose, for contradiction, that Algorithm~\ref{BiCSalgorithm} generates an infinite sequence $\{\bm{x}_k\}_{k \ge 1} \subseteq \mathcal{X}$. Every nonterminating iteration adds support for at least one row from a finite-support feasible distribution whose actual value exceeds $b_m+\epsilon/2$. Since $[M]$ is finite, the pigeonhole principle yields an index $m^*$ whose support is updated infinitely often; passing to a subsequence and using compactness of $\mathcal{X}$, we may assume that row $m^*$ is updated at every iteration and $\bm{x}_k \to \bm{x}^* \in \mathcal{X}$.

For each such $k$, let $\tilde{\mathbb{P}}^*_{m^*, k}$ be the finite-support distribution whose support is added for row $m^*$, supported on $\{\bm{a}^{(k)}_{\ell}\}_{\ell} \subseteq \mathcal{C}_{m^*}$ with weights $\{P^{(k)}_{\ell}\}_{\ell}$. Then
\begin{equation} \label{eq:oracle_approx}
    \sum_{\ell} P^{(k)}_{\ell}\, f_{m^*}(\bm{a}^{(k)}_{\ell}, \bm{x}_k) 
    \;=\; \hat{\Omega}_{m^*}(\bm{x}_k) 
    \;>\; b_{m^*} + \epsilon/2,
\end{equation}
and Step~3 absorbs $\{\bm{a}^{(k)}_{\ell}\}_{\ell}$ into $\hat{A}_{m^*}$. 
Hence, at any later iteration $k' > k$, the distribution 
$\tilde{\mathbb{P}}^*_{m^*, k}$, when extended by zero mass to scenarios added afterwards, remains feasible in the inner maximization of MP-DRO; therefore, feasibility of $\bm{x}_{k'}$ yields $\sum_{\ell} P^{(k)}_{\ell}\, f_{m^*}(\bm{a}^{(k)}_{\ell}, \bm{x}_{k'}) \;\le\; b_{m^*}.$ 

The product set $\mathcal{C}_{m^*} \times \mathcal{X}$ is compact, so $f_{m^*}$ is uniformly continuous on it. Pick $\delta > 0$ such that $\|\bm{x} - \bm{x}'\| \le \delta$ implies $|f_{m^*}(\bm{a}, \bm{x}) - f_{m^*}(\bm{a}, \bm{x}')| < \epsilon/2$ for all $\bm{a} \in \mathcal{C}_{m^*}$. Since $\bm{x}_k \to \bm{x}^*$, choose $k' > k$ large enough that $\|\bm{x}_k - \bm{x}_{k'}\| \le \delta$. Multiplying by $P^{(k)}_{\ell} \ge 0$, summing over $\ell$ and using MP-DRO feasibility, we have
\[
    \sum_{\ell} P^{(k)}_{\ell}\, f_{m^*}(\bm{a}^{(k)}_{\ell}, \bm{x}_k)
    \;<\; \sum_{\ell} P^{(k)}_{\ell} \, f_{m^*}(\bm{a}^{(k)}_{\ell}, \bm{x}_{k'}) + \epsilon/2 \sum_{\ell} P^{(k)}_{\ell}
    \;\le\; b_{m^*} + \epsilon/2,
\]
contradicting~\eqref{eq:oracle_approx}. Hence the algorithm terminates in finitely many iterations.
\end{proof}

\begin{remark}[Infeasibility detection]
If MP-DRO becomes infeasible at some iteration, Algorithm~\ref{BiCSalgorithm} terminates and problem~\eqref{eq_DRO} is declared infeasible because MP-DRO is a relaxation of~\eqref{eq_DRO}.
\end{remark}

\begin{remark}[Exact convergence]
\label{rem:exact-BiCS}
Under the remaining well-posedness premises of Theorem~\ref{convergence-BiCS}, consider the zero-tolerance variant obtained by replacing Step~0 with $\epsilon=0$. If every \textnormal{SP-DRO} oracle is exact and this variant terminates by the stopping criterion, then the returned solution is optimal for~\eqref{eq_DRO}.
\end{remark}

\begin{corollary}[Discrete $\mathcal{X}$]
\label{cor:discrete-X}
Under the conditions of Theorem~\ref{convergence-BiCS}, if $\mathcal{X}$ is compact and discrete, then Algorithm~\ref{BiCSalgorithm} terminates finitely and returns an $\epsilon$-optimal solution to~\eqref{eq_DRO}.
\end{corollary}

\begin{proof}
Since $\mathcal X$ is compact and discrete, it is finite. Suppose Algorithm~\ref{BiCSalgorithm} does not terminate. Then there exist iterations $k<\ell$ such that $\bm x^k=\bm x^\ell=:\bar{\bm x}$. Since iteration $k$ is nonterminating, for some $m\in[M]$, the oracle returns a distribution $\mathbb P_m^k$ satisfying
\[
\mathbb E_{\mathbb P_m^k}[g_m(\bar{\bm x},\bm a_m)]>b_m+\epsilon/2.
\]
Once its support is added, $\mathbb P_m^k$ remains feasible for all subsequent restricted problems. Hence, at iteration $\ell$, the left-hand side of the corresponding MP-DRO constraint at $\bar{\bm x}$ is greater than $b_m$, contradicting the feasibility of $\bm x^\ell$. Therefore, the algorithm terminates finitely, and $\epsilon$-optimality follows from Theorem~\ref{convergence-BiCS}.
\end{proof}

\subsubsection{Cutting-set Mechanism}

In RO, the adversarial response is a scenario, so row generation naturally transfers a violated scenario to the master. In DRO, Luo and Mehrotra \cite{luo2019decomposition} use scenario-indexed separation after a semi-infinite reformulation, while Yang and Wu \cite{yang2018distributionally} solve an equivalent semidefinite reformulation by delayed constraint generation. SP-DRO instead returns a finite-support distribution, so BiCS transfers its entire support in a single update. The comparison formalized below concerns the granularity of support generation: a single-scenario update versus a complete support set. Theorem~\ref{thm:discretize_unbounded} motivates finite-support approximation at the expected value level. Under the conditions of Proposition~\ref{prop:SP-O1}, a finite-support oracle returns an optimal distribution; appending its entire support in a single update recovers the full worst-case value at the current master solution. We refer to this as a \emph{cutting-set} update, and Proposition~\ref{prop:multi-vs-single} below quantifies the resulting gain over the single-scenario approach. We state the comparison under exact subproblem optimality to keep the structural distinction transparent; in the $\epsilon$-optimal setting of Theorem~\ref{convergence-BiCS}, a finite-support return instead provides a certified lower bound.

\begin{proposition}\label{prop:multi-vs-single}
Fix an iteration $k$ with master solution $\bm x^k$. Suppose SP-DRO~\eqref{SubProblem} is solved exactly at $\bm x^k$ for every row $m\in[M]$ by an oracle returning a finite-support optimal distribution. For each $m\in[M]$, let $\hat{A}_m^M$ and $\hat{A}_m^S$ denote the scenario sets obtained from the same current set $\hat{A}_m$ by a cutting-set update and a single-scenario update, respectively, and let $\phi_m(\bm x, \hat{A}) := \sup\{\mathbb{E}_{\mathbb P_m}[f_m(\bm{a}_m, \bm x)] : \mathbb P_m \in \mathcal{P}_m,\ \operatorname{supp}(\mathbb P_m) \subseteq \hat{A}\}$. Then
\begin{itemize}
    \item[(i)] for any fixed $m\in[M]$, $\phi_m(\bm x^k, \hat{A}_m^S) \;\le\; \Omega_m^*(\bm x^k) \;=\; \phi_m(\bm x^k, \hat{A}_m^M)$;
    \item[(ii)] $\mathcal{F}^* \subseteq \mathcal{F}^M$, where $\mathcal{F}^M := \{\bm x \in \mathcal{X} : \phi_m(\bm x, \hat{A}_m^M) \le b_m,\ \forall m \in [M]\}$ and $\mathcal{F}^*$ is the feasible region of \eqref{eq_DRO}.
\end{itemize}
\end{proposition}

\begin{proof}
For part~(i), fix any $m\in[M]$ and let $\hat{A}_m$ denote its scenario set before the update.

Let $\tilde{\mathbb P}_m^*$ be the finite-support optimal distribution returned by the oracle and set $\tilde A_m^*:=\operatorname{supp}(\tilde{\mathbb P}_m^*)$. The cutting-set update appends this entire support, so $\hat{A}_m^M = \hat{A}_m \cup \tilde A_m^*$ and $\tilde{\mathbb P}_m^*$ is feasible for the maximization defining $\phi_m(\bm x^k, \hat{A}_m^M)$, giving
\[
\phi_m(\bm x^k, \hat{A}_m^M) \;\ge\; \mathbb{E}_{\tilde{\mathbb P}_m^*}[f_m(\bm{a}_m, \bm x^k)] \;=\; \Omega_m^*(\bm x^k).
\]
The reverse inequality $\phi_m(\bm x^k, \hat{A}_m^M) \le \Omega_m^*(\bm x^k)$ is immediate from $\mathcal{P}_m(\hat{A}_m^M) \subseteq \mathcal{P}_m$, hence the equality. The single-scenario update appends one scenario $\bm{a}_{md}\in\tilde A_m^*$, yielding $\hat{A}_m^S = \hat{A}_m \cup \{\bm{a}_{md}\} \subseteq \mathcal{A}_m$, and the same set inclusion gives $\phi_m(\bm x^k, \hat{A}_m^S) \le \Omega_m^*(\bm x^k)$, which establishes (i).

For (ii), take any $\bm x \in \mathcal{F}^*$, so that $\Omega_m^*(\bm x) \le b_m$ for all $m \in [M]$. Since $\mathcal{P}_m(\hat{A}_m^M) \subseteq \mathcal{P}_m$, we have $\phi_m(\bm x, \hat{A}_m^M) \le \Omega_m^*(\bm x) \le b_m$ for every $m$, and hence $\bm x \in \mathcal{F}^M$.
\end{proof}

Part~(i) of Proposition~\ref{prop:multi-vs-single} shows that the cutting-set update reproduces the worst-case value $\Omega_m^*(\bm x^k)$ exactly after a single subproblem solve, whereas the single-scenario update only yields a value bounded above by $\Omega_m^*(\bm x^k)$ and may require subsequent iterations to close this gap. The traditional single-scenario method is therefore recovered as a degenerate special case of the cutting-set framework, in which the appended support is artificially restricted to a singleton at each iteration. Part~(ii) confirms that the cutting-set master remains a valid relaxation, since $\mathcal{F}^* \subseteq \mathcal{F}^M$.

A natural concern is that the master problem may grow more rapidly under cutting-set updates than under single-scenario ones. However, the scenarios added in each cutting set jointly represent the support of one adversarial distribution, rather than unrelated individual cuts, and therefore provide a more complete representation of the identified distributional violation. Theoretical bounds and practical strategies for controlling this growth are discussed later in Section~\ref{BiCS SP} and ~\ref{sec:Enhancement}.

\subsubsection{Bilevel Structure of Distribution Cuts}

The cutting-set mechanism specifies how many scenarios enter a cut, but it leaves open how probability mass is assigned to those scenarios. Algorithm~\ref{BiCSalgorithm} reflects this by treating the probability measure as a decision variable in both MP-DRO~\eqref{MasterProblem} and SP-DRO~\eqref{SubProblem}, constrained in each case to lie in $\mathcal{P}_m$. Proposition~\ref{prop:SP-O1} provides an exact admissible finite-support optimizer under its conditions, while Corollary~\ref{cor:SP-O1-oracle} provides an $\epsilon/2$-accurate admissible finite-support solution under its stated conditions. In both cases, membership in $\mathcal{P}_m$ ensures that the resulting distribution cut is valid for $\Omega_m^*(\bm x)$, as required by Theorem~\ref{convergence-BiCS}.

The master problem raises a related but distinct question: how should support information from successive oracle calls participate in future distributions? There are three natural representations. A \emph{fixed distribution cut}, as in the sampling-based cutting surface method of Mehrotra and Papp \cite{mehrotra2014cutting}, uses each generated distribution as a cut and thereby freezes its support and probability weights within that cut. A related fixed weight cutting plane construction is used by Xu et al.\ \cite{xu2018distributionally} on a discretized ambiguity set. A naive parametric construction retains each generated support as an individual block and reoptimizes weights only within that block; we call this a \emph{blockwise parametric distribution cut}. BiCS instead uses a \emph{pooled parametric distribution cut}: it combines all retained support and reoptimizes the weights jointly over the resulting set. Pooling therefore permits admissible distributions that combine scenarios generated in different iterations. The next proposition compares the three representations for a common support history.

\begin{proposition}[Ordering of distribution cuts]
\label{prop:BiCS-vs-fixed}
Fix a constraint $m\in[M]$, a positive integer $R$, and a common collection of finite support blocks $A_m^r\subseteq\mathcal A_m$, $r\in[R]$, with $\bar{\mathbb P}_m^r\in\mathcal P_m(A_m^r)$. Let $\hat A_m:=\bigcup_{r=1}^R A_m^r$ and, using the restricted expected value $\phi_m$ defined in Proposition~\ref{prop:multi-vs-single}, define
\[
\phi_m^{\mathrm{fix}}(\bm x):=\max_{r\in[R]}\mathbb E_{\bar{\mathbb P}_m^r}[f_m(\bm a_m,\bm x)],\qquad
\phi_m^{\mathrm{blk}}(\bm x):=\max_{r\in[R]}\phi_m(\bm x,A_m^r),\qquad
\phi_m^{\mathrm{pool}}(\bm x):=\phi_m(\bm x,\hat A_m).
\]
Then, for every $\bm x\in\mathcal X$, $\phi_m^{\mathrm{fix}}(\bm x)\le\phi_m^{\mathrm{blk}}(\bm x) \le\phi_m^{\mathrm{pool}}(\bm x)\le\Omega_m^*(\bm x).$
For $q\in\{\mathrm{fix},\mathrm{blk},\mathrm{pool}\}$, let $\mathcal X_m^q:=\{\bm x\in\mathcal X:\phi_m^q(\bm x)\le b_m\}$ and let $\mathcal X_m^*:=\{\bm x\in\mathcal X:\Omega_m^*(\bm x)\le b_m\}$. Then
\[
\mathcal X_m^*\subseteq\mathcal X_m^{\mathrm{pool}}
\subseteq\mathcal X_m^{\mathrm{blk}}\subseteq\mathcal X_m^{\mathrm{fix}}.
\]
\end{proposition}

\begin{proof}
For every $r\in[R]$, $\{\bar{\mathbb P}_m^r\}\subseteq\mathcal P_m(A_m^r)
\subseteq\mathcal P_m(\hat A_m)\subseteq\mathcal P_m$.
Therefore, for every $\bm x\in\mathcal X$,
\[
\mathbb E_{\bar{\mathbb P}_m^r}[f_m(\bm a_m,\bm x)]
\le\phi_m(\bm x,A_m^r)
\le\phi_m(\bm x,\hat A_m)
\le\Omega_m^*(\bm x).
\]
Taking the maximum over $r\in[R]$ gives the expected value inequalities. The feasible region inclusions follow immediately.
\end{proof}

\begin{example}[Strict feasible-region inclusions]
\label{ex:strict-fixed-cut}
Let $\mathcal X=\mathbb R_+$, $\mathcal A_m=\{0,1,2\}$,
$f_m(a,x)=a^2x$, $b_m=1$, and
\[
\mathcal P_m
=
\left\{
\mathbb P\text{ on }\mathcal A_m:
\mathbb E_{\mathbb P}[a]\le\tfrac32
\right\}.
\]
Take $R=2$, $A_m^1=\{0,1\}$, $A_m^2=\{1,2\}$, and
\[
\bar{\mathbb P}_m^1=\tfrac12\delta_0+\tfrac12\delta_1,
\qquad
\bar{\mathbb P}_m^2=\tfrac23\delta_1+\tfrac13\delta_2.
\]
The fixed cuts give $\phi_m^{\mathrm{fix}}(x)=2x$. Over
$A_m^1$ and $A_m^2$, the largest values of
$\mathbb E_{\mathbb P}[a^2]$ are $1$ and $\tfrac52$,
respectively, so $\phi_m^{\mathrm{blk}}(x)=\tfrac52x$.
After pooling, $a^2\le2a$ on $\mathcal A_m$, giving
$\mathbb E_{\mathbb P}[a^2]\le3$, with equality at
$\mathbb P=\tfrac14\delta_0+\tfrac34\delta_2$. Hence
$\phi_m^{\mathrm{pool}}(x)=3x$ and
\[
\mathcal X_m^{\mathrm{pool}}=[0,\tfrac13]
\subsetneq
\mathcal X_m^{\mathrm{blk}}=[0,\tfrac25]
\subsetneq
\mathcal X_m^{\mathrm{fix}}=[0,\tfrac12].
\]
Thus, both inclusions among the three representations in
Proposition~\ref{prop:BiCS-vs-fixed} can be strict simultaneously.
\end{example}

\begin{remark}[Decision-Dependent Uncertainty]
\label{rem:DDU}
The support--weight mechanism also accommodates decision-dependent ambiguity sets $\mathcal{P}_m(\bm x)$. Because the master problem re-optimizes the probability weights at each candidate decision, decision dependence can be incorporated directly through $\mathcal{P}_m(\bm x)$. When the sample space $\mathcal{A}_m$ is decision-independent, support points generated in past iterations remain valid as $\bm x$ varies.
\end{remark}

Together, Propositions~\ref{prop:multi-vs-single} and~\ref{prop:BiCS-vs-fixed} show that BiCS retains distributional information at two levels: the cutting set update transfers the complete support returned by the oracle, and the pooled parametric distribution cut reoptimizes admissible weights jointly over all retained support points. The three representations form a progression: a fixed cut freezes its support and weights together, a blockwise cut reoptimizes weights within each support block, and a pooled cut reoptimizes weights jointly over the union of all retained blocks. Implementing this representation requires the inner maximization in the master problem and the oracle subproblem, which are developed next.

\subsection{Anatomy of BiCS Framework}

This subsection details how each component of the BiCS framework is solved in practice. Section~\ref{BiCS Master Problem} presents two reformulation strategies for eliminating the inner maximization in the master problem, while Section~\ref{BiCS SP} develops two oracle procedures for the subproblem: a finite mathematical programming formulation and a column generation scheme. Finally, Section~\ref{sec:Enhancement} further introduces computational enhancements to improve efficiency. An overview of the framework is provided in Figure~\ref{fig:BiCS framework}. 

\begin{figure}[h!]
    \centering
    \includegraphics[width=0.9\textwidth]{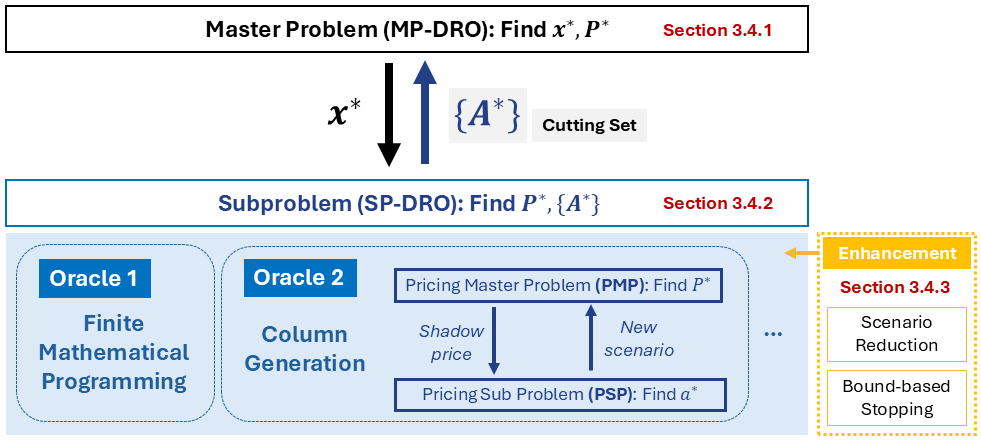}
    \caption{Anatomy of BiCS framework.}
    \label{fig:BiCS framework}
\end{figure}

\subsubsection{Solving Master Problem} \label{BiCS Master Problem}

The master problem \eqref{MasterProblem} contains an embedded worst-case expected value defined by maximization over the probability measure. By Step~1 of Algorithm~\ref{BiCSalgorithm}, $\mathcal P_m(\hat A_m)$ is nonempty throughout the algorithm; for fixed $\bm x$, the constraints in Definition~\ref{ambiguity_set_define} make this inner maximization a feasible and bounded linear program. This expected value can be represented in a single-level formulation via standard optimality conditions, such as LP duality or KKT-based reformulations. Detailed formulations are provided in Appendix~\ref{appendix:master}.

\subsubsection{Solving Subproblem} \label{BiCS SP}

While the master problem reduces cleanly to a single-level model, the subproblem~\eqref{SubProblem} is where the main computational challenge lies. For a fixed decision $\bm{x}^*$, the subproblem must identify a \emph{worst-case distribution} over a space of probability measures that is, in general, infinite-dimensional. Theorem~\ref{thm:discretize_unbounded} provides only expected-value-level approximation by finite-support measures; admissible exact and approximate finite-support distributions are supplied by Proposition~\ref{prop:SP-O1} and Corollary~\ref{cor:SP-O1-oracle}, respectively. The remaining construction and computational question is \emph{how} to identify such a finite support, and do so efficiently. We address this question from two complementary angles through the two oracles developed in the sequel.

We note that the finite mathematical programming oracle and the PMP--PSP oracle based on the column generation technique have been employed by Lu and Zeng \cite{lu2024two} to solve two-stage DRO problems with recourse, in their basic versions for a compact sample space. Here, we investigate and develop them in the context of single-stage DRO. In particular, we present analyses for closed but possibly unbounded sample spaces, which require substantial new arguments for finite-support representation, oracle construction, and convergence.

\begin{assumption}[Light-tail domination]
\label{assumption:light-tail}
For each $m \in [M]$, each moment function $g_t : \mathcal{A}_m \to \mathbb{R}$ is continuous, $t \in [T_m]$. There exists an index $t_0 \in [T_m]$ such that $g_{t_0}$ is nonnegative and coercive on $\mathcal{A}_m$, and $g_{t_0}$ strictly dominates the tail behavior of both $f_m$ and the remaining moment functions: for every $\bm{x} \in \mathcal{X}$,
\[
    \lim_{g_{t_0}(\bm{a}_m) \to \infty} 
    \frac{|f_m(\bm{a}_m, \bm{x})|}{g_{t_0}(\bm{a}_m)} = 0,
    \qquad 
    \lim_{g_{t_0}(\bm{a}_m) \to \infty} 
    \frac{|g_t(\bm{a}_m)|}{g_{t_0}(\bm{a}_m)} = 0 
    \;\; \forall t \in [T_m] \setminus \{t_0\}.
\]
\end{assumption}

Conditions of the form in Assumption~\ref{assumption:light-tail} are standard tools for controlling tails in DRO over unbounded sample spaces \cite{xu2018distributionally,blanchet2019quantifying,yue2022linear}. Our contribution is to show that, within the primal BiCS framework for generalized-moment single-stage DRO, they yield an exact finite-support formulation for Oracle~1 and, together with the additional conditions of Theorem~\ref{ConvergeProofSPCG}, finite termination of Oracle~2 without requiring the sample space to be compact.

\paragraph{Oracle 1 (Finite Mathematical Programming)}

The key idea behind Oracle~1 is conceptually straightforward. Recall that in RO, the subproblem seeks a single worst-case scenario $\bm{a}_m \in \mathcal{A}_m$. In DRO, the subproblem instead seeks an entire worst-case \emph{distribution} over $\mathcal{A}_m$, which is an inherently infinite-dimensional object. However, a classical result \cite{shapiro2001duality} shows that whenever the supremum in \eqref{SubProblem} is attained, there exists a worst-case distribution supported on at most $T_m+1$ scenarios. This means we only need to find $T_m+1$ ``representative'' scenarios and allocate probability mass among them. Oracle~1 exploits this reduction directly: it pre-allocates $T_m+1$ slots for candidate scenarios, and simultaneously optimizes over both their locations $\tilde{\bm{a}}_i \in \mathcal{A}_m$ and their probability weights $\tilde{P}_i$, treating all of them as decision variables of a single finite-dimensional mathematical program.

\begin{proposition} \label{prop:SP-O1}
Suppose Assumptions~\ref{assumption:regularity} and~\ref{assumption:light-tail} 
hold. Then for each constraint $m \in [M]$ 
and any fixed $\bm{x}^* \in \mathcal{X}$, the worst-case expectation 
$\Omega_m^*(\bm{x}^*)$ defined in \eqref{SubProblem} coincides with the 
optimal value $\Omega_m^{\mathrm{O1}}(\bm{x}^*)$ of the following finite 
mathematical program:
\begin{subequations} \label{SP-bilinear}
\begin{align}
    \textbf{\emph{SP-O1:}} \quad 
    \Omega_m^{\mathrm{O1}}(\bm{x}^*) \;=\; \max \quad 
    & \sum_{i=1}^{T_m+1} \tilde{P}_i \, f_m(\tilde{\bm{a}}_i, \bm{x}^*) \\
    \mathrm{s.t.} \quad 
    & \sum_{i=1}^{T_m+1} \tilde{P}_i = 1, \quad
      \sum_{i=1}^{T_m+1} \tilde{P}_i \, g_t(\tilde{\bm{a}}_i) \leq \gamma_{mt}, 
      \quad \forall t \in [T_m], \label{SP-bilinear-start} \\
    & \tilde{\bm{a}}_i \in \mathcal{A}_m, \quad \tilde{P}_i \geq 0,
      \quad \forall i \in [T_m+1]. \label{SP-bilinear-end}
\end{align}
\end{subequations}
\end{proposition}

\begin{proof}
We prove both inequalities. Any feasible solution 
$\{(\tilde{\bm{a}}_i,\tilde P_i)\}_{i=1}^{T_m+1}$ of SP-O1 induces the 
probability measure $\mathbb Q:=\sum_{i=1}^{T_m+1}\tilde P_i
\delta_{\tilde{\bm a}_i}\in\mathcal P_m$. Its SP-O1 objective equals 
$\mathbb E_{\mathbb Q}[f_m(\bm a_m,\bm x^*)]$, and hence
$\Omega_m^{\mathrm{O1}}(\bm x^*)\le\Omega_m^*(\bm x^*)$.

For the reverse inequality, we first establish attainment. For $R>0$, let
$C_R:=\{\bm a\in\mathcal A_m:g_{t_0}(\bm a)\le R\}$. This set is compact 
because $\mathcal A_m$ is closed and $g_{t_0}$ is continuous and coercive. 
For every $\mathbb P\in\mathcal P_m$, Markov's inequality gives
\[
    \mathbb P(C_R^c)\le
    \frac{\mathbb E_{\mathbb P}[g_{t_0}]}{R}
    \le\frac{\gamma_{mt_0}}{R},
\]
so $\mathcal P_m$ is tight. To prove weak closedness, let 
$\mathbb P_n\rightharpoonup\mathbb P$ with $\mathbb P_n\in\mathcal P_m$. 
Since $\mathcal A_m$ is closed, the Portmanteau theorem implies 
$\mathbb P(\mathcal A_m)=1$. Moreover, because $g_{t_0}$ is nonnegative and 
lower semicontinuous, $\mathbb E_{\mathbb P}[g_{t_0}]\le
\liminf_n\mathbb E_{\mathbb P_n}[g_{t_0}]\le\gamma_{mt_0}$.
For $t\ne t_0$ and any $\eta>0$, Assumption~\ref{assumption:light-tail} 
gives $R_\eta>0$ such that $|g_t|\le\eta g_{t_0}$ on 
$\{g_{t_0}>R_\eta\}$. Continuity of $g_t$ gives 
$M_\eta:=\sup_{C_{R_\eta}}|g_t|<\infty$. Thus, for $K>M_\eta$,
\[
    \sup_{\mathbb Q\in\mathcal P_m}
    \mathbb E_{\mathbb Q}[g_t^-\mathbf 1_{\{g_t^->K\}}]
    \le\eta\gamma_{mt_0},
\]
Since $\eta>0$ is arbitrary, letting $\eta\downarrow0$ proves uniform 
integrability of $\{g_t^-:\mathbb Q\in\mathcal P_m\}$. 
The same tail bound gives $\mathbb E_{\mathbb P}[g_t^-]<\infty$. For 
$g_{t,L}:=\max\{g_t,-L\}$, Portmanteau and uniform integrability yield
\[
    \mathbb E_{\mathbb P}[g_{t,L}]
    \le\liminf_n\mathbb E_{\mathbb P_n}[g_{t,L}]
    \le\gamma_{mt}+\sup_n
    \mathbb E_{\mathbb P_n}[g_t^-\mathbf 1_{\{g_t^->L\}}].
\]
Letting $L\to\infty$ gives 
$\mathbb E_{\mathbb P}[g_t]\le\gamma_{mt}$. Hence 
$\mathcal P_m$ is weakly closed and, by Prokhorov's theorem, weakly compact.

Assumption~\ref{assumption:regularity}(i) makes 
$f_m(\cdot,\bm x^*)$ upper semicontinuous. The same tail argument, now using 
its boundedness above on each $C_R$, shows that 
$\{f_m^+(\cdot,\bm x^*):\mathbb Q\in\mathcal P_m\}$ is uniformly 
integrable. Applying Portmanteau to the bounded-above upper semicontinuous 
truncations $f_{m,L}:=\min\{f_m,L\}$ and removing the truncation using uniform 
integrability and Assumption~\ref{assumption:regularity}(ii) gives
\[
    \limsup_n\mathbb E_{\mathbb P_n}[f_m(\bm a_m,\bm x^*)]
    \le\mathbb E_{\mathbb P}[f_m(\bm a_m,\bm x^*)].
\]
Thus the objective is weakly upper semicontinuous on the nonempty weakly 
compact set $\mathcal P_m$, and its supremum is attained by some 
$\mathbb P^\star\in\mathcal P_m$.

All these expectations are finite by Definition~\ref{ambiguity_set_define} 
and Assumption~\ref{assumption:regularity}(ii). Apply Richter--Rogosinski 
\cite[Lemma~3.1]{shapiro2001duality} to $1,g_1,\ldots,g_{T_m}$, and 
$f_m(\cdot,\bm x^*)$ under $\mathbb P^\star$. 
It yields atoms $\{\bm a_j\}_{j=1}^r\subseteq\mathcal A_m$, 
$r\le T_m+2$, and weights $q_j\ge0$ matching these expectations. Let 
$v_t:=\mathbb E_{\mathbb P^\star}[g_t]$ and consider the finite weights LP
\[
    \max_{p_j\ge0}\left\{\sum_{j=1}^r p_j f_m(\bm a_j,\bm x^*):
    \sum_{j=1}^r p_j=1,\ \sum_{j=1}^r p_j g_t(\bm a_j)=v_t,\ \forall t\in[T_m]\right\}.
\]
The weights $q_j$ are feasible, so this LP has a basic optimal solution with 
at most $T_m+1$ positive weights because it has $T_m+1$ equality constraints. 
It preserves $v_t\le\gamma_{mt}$ and therefore defines a measure in 
$\mathcal P_m$. Its objective is at least that of $q$, namely 
$\Omega_m^*(\bm x^*)$, and cannot exceed this supremum. Zero-padding gives 
a feasible SP-O1 solution attaining $\Omega_m^*(\bm x^*)$, proving the 
reverse inequality and the result.
\end{proof}

Proposition~\ref{prop:SP-O1} establishes that Oracle~1 exactly recovers the worst-case expectation when the supremum is attained. In some cases, however, the supremum may only be approached but not achieved (e.g., when the sample space is unbounded and lacks a coercive moment constraint). The following corollary shows that even in such cases, Oracle~1 can still produce a solution that is arbitrarily close to the true worst-case value. Such an $\epsilon/2$-accurate solution supplies the finite-support value certificate required in Step~3 of Algorithm~\ref{BiCSalgorithm}.

\begin{corollary}[SP-O1 as an $\epsilon/2$-optimal oracle]
\label{cor:SP-O1-oracle}
Suppose $\mathcal{P}_m$ is nonempty and Assumption~\ref{assumption:regularity}(ii) holds. Suppose further that $\Omega_m^*(\bm{x}^*) < \infty$. Then for every $\epsilon > 0$, there exists a feasible solution 
$\{(\tilde{\bm{a}}_i, \tilde{P}_i)\}_{i=1}^{T_m+1}$ of \emph{SP-O1} 
satisfying
\[
    \Omega_m^*(\bm{x}^*) - \epsilon/2 
    \;\le\; \sum_{i=1}^{T_m+1} \tilde{P}_i\, f_m(\tilde{\bm{a}}_i, \bm{x}^*) 
    \;\le\; \Omega_m^*(\bm{x}^*).
\]
\end{corollary}

\begin{proof}
By the definition of the supremum, there exists $\mathbb{P}^{\epsilon} \in \mathcal{P}_m$ with $\mathbb{E}_{\mathbb{P}^{\epsilon}}[f_m(\bm{a}_m, \bm{x}^*)] 
\geq \Omega_m^*(\bm{x}^*) - \epsilon/2$; this expectation is finite because the integrability derived from Definition~\ref{ambiguity_set_define}, together with Assumption~\ref{assumption:regularity}(ii), yields $\mathbb{E}_{\mathbb{P}^{\epsilon}}[|f_m|] < \infty$. Note that $\mathbb{P}^{\epsilon}$ need not be discrete.

Apply Richter--Rogosinski \cite[Lemma~3.1]{shapiro2001duality} to 
$1,g_1,\ldots,g_{T_m},$ and $f_m(\cdot,\bm{x}^*)$ under 
$\mathbb P^\epsilon$. This gives at most $T_m+2$ atoms and weights matching 
all these expectations. Fixing those atoms and solving the finite weights LP 
displayed above with 
$v_t=\mathbb E_{\mathbb P^\epsilon}[g_t]$ yields a basic optimal solution 
with at most $T_m+1$ positive weights. The resulting measure remains in 
$\mathcal P_m$ and has objective no smaller than that of 
$\mathbb P^\epsilon$. Zero-padding gives a feasible $T_m+1$-slot SP-O1 solution attaining the lower bound.

Conversely, any SP-O1 feasible tuple induces a measure in $\mathcal{P}_m$ whose $f_m$-expectation is at most $\Omega_m^*(\bm{x}^*)$, giving the upper bound.
\end{proof}

In Proposition~\ref{prop:SP-O1}, we explicitly construct the worst-case distribution by pre-specifying \(T_m+1\) slots for candidate scenarios and assigning each a probability weight. This reformulates the infinite-dimensional optimization over probability measures as a finite-dimensional mathematical program. In many cases, the resulting formulation can be handled directly by standard commercial solvers, without requiring additional algorithmic adjustments. For example, when \(f_m(\cdot,\bm x^*)\) and the functions \(g_t\) are linear in \(\tilde{\bm a}_i\), the model involves only bilinear terms of the form \(\tilde{P}_i \tilde{\bm a}_i\), which modern solvers can typically handle in practice. 

\paragraph{Oracle 2 (Column Generation)}

Although Oracle~1 is effective when $f_m(\cdot,\bm x^*)$ and the moment functions are linear, nonlinear or higher-order functions introduce high-degree interactions between the scenario and probability variables, which can substantially degrade computational performance. To address this, we instead decouple the two optimizations via a column generation (CG) scheme. The key observation is that for a fixed $\bm{x}^*$, optimizing the probability weights $P_{mk}$ over the current scenario set, i.e., the Pricing Master Problem (PMP), is simply a linear program, with shadow price $\alpha_m^*$ and $\bm{\beta}_m^*$ readily available. Introducing a new scenario $\tilde{\bm{a}}_m$ then corresponds to adding a new column with reduced cost $f_m(\tilde{\bm{a}}_m, \bm{x}^*) - \alpha_m^* - \sum_{t \in [T_m]} \beta_{mt}^* g_t(\tilde{\bm{a}}_m)$; the Pricing Subproblem (PSP) searches for a column whose reduced cost exceeds the prescribed threshold and terminates when none can be found. The overall procedure is summarized in Algorithm~\ref{SPCG}.

\begin{algorithm}[h!]
\footnotesize
\setlength{\itemsep}{1pt}
\caption{\textbf{\textit{SP - Oracle 2 (CG)}}} \label{SPCG}
\begin{algorithmic}
\State \textbf{Step 3.0:} Initialize tolerance $\epsilon>0$ and counter $l=0$. Fix row $m$, incumbent $\bm x^*$, and current scenario set $\hat A_m=\{\bm a_{mk}:k=0,\ldots,K_m\}\subseteq\mathcal A_m$.

\State \textbf{Step 3.1:} \textbf{Solving Pricing Master Problem (PMP).} Solve:
\vspace{-2mm}
\begin{subequations}\label{PricingMasterProblem}
\begin{align}
    \textbf{PMP:}\quad
    \omega_m^*(\bm x^*)=\max_{P_{mk}\ge0}\quad
    & \sum_{k=0}^{K_m}P_{mk}f_m(\bm a_{mk},\bm x^*) \label{PMP-objective}\\[-1mm]
    \text{s.t.}\quad
    & \sum_{k=0}^{K_m}P_{mk}=1,\quad
    \sum_{k=0}^{K_m}P_{mk}g_t(\bm a_{mk})\le\gamma_{mt}\ \ \forall t\in[T_m]. \label{PMP-constraints}
\end{align}
\end{subequations}
\vspace{-2mm}
\Statex \hspace{4em} \textbf{if} PMP is infeasible \textbf{then}
\Statex \hspace{6em} Let $(\alpha_m^*,\bm\beta_m^*)$, with $\bm\beta_m^*\ge\bm 0$, be a dual Farkas ray for the restricted PMP, and solve:
\begin{align} \label{PricingMasterProblem-infeasible}
    \textbf{PMP-infeasible:} \quad v^*_m = \sup_{\tilde{\bm a}_m \in \mathcal A_m} 
    & \quad - \alpha_m^* - \sum_{t \in [T_m]} \beta_{mt}^* g_t(\tilde{\bm a}_m) 
\end{align}
\vspace{-2mm}
\Statex \hspace{6em} Choose any $\tilde{\bm a}_m\in\mathcal A_m$ with $-\alpha_m^*-\sum_{t\in[T_m]}\beta_{mt}^*g_t(\tilde{\bm a}_m)>0$.
\Statex \hspace{6em} Update $\hat A_m\gets\hat A_m\cup\{\tilde{\bm a}_m\}$, $K_m\gets K_m+1$, and $l\gets l+1$.
\Statex \hspace{6em} Return to Step~3.1.
\Statex \hspace{4em} \textbf{else}
\Statex \hspace{6em} Let $(\alpha_m^*,\bm\beta_m^*)$ be an optimal dual solution and let $\{P_{mk}^*\}_{k=0}^{K_m}$ be an optimal primal solution.

\State \textbf{Step 3.2:} \textbf{Solving Pricing Subproblem (PSP).} Solve:
\vspace{-2mm}
\begin{align} \label{PricingSubProblem}
    \textbf{PSP:} \quad 
    \pi_m^*(\bm x^*, \alpha_m^*, \bm\beta_m^*) = \sup_{\tilde{\bm a}_m \in \mathcal A_m} 
    & \quad f_m(\tilde{\bm a}_m, \bm x^*) 
    - \alpha_m^* - \sum_{t \in [T_m]} \beta_{mt}^* g_t(\tilde{\bm a}_m) 
\end{align}
\vspace{-2mm}
\State \textbf{Step 3.3:} Stopping criterion:
\Statex \hspace{4em} \textbf{if} $\pi_m^*\le\epsilon/2$ \textbf{then}
\Statex \hspace{6em} Return $\mathbb P_m^{\mathrm{PMP}}:=\sum_{k=0}^{K_m}P_{mk}^*\delta_{\bm a_{mk}}$ and $\omega_m^*(\bm x^*)$.
\Statex \hspace{4em} \textbf{else}
\Statex \hspace{6em} Choose any $\tilde{\bm a}_m\in\mathcal A_m$ whose reduced cost exceeds $\epsilon/2$.
\Statex \hspace{6em} Update $\hat A_m\gets\hat A_m\cup\{\tilde{\bm a}_m\}$, $K_m\gets K_m+1$, and $l\gets l+1$.
\Statex \hspace{6em} Return to Step~3.1.
\end{algorithmic}
\end{algorithm}

Under Step~1 of Algorithm~\ref{BiCSalgorithm}, the MP-DRO (\ref{MasterProblem}) remains feasible throughout standard BiCS. The infeasibility-handling branch of Algorithm~\ref{SPCG} and the following proposition are included only for alternative initializations or extensions.

\begin{proposition}[\textbf{PMP Infeasibility}]
\label{prop:PMP-infeasible}
Suppose $\mathcal P_m$ is nonempty but the PMP restricted to $\hat A_m$ is infeasible. Let $(\alpha_m^*,\bm\beta_m^*)$ with $\bm\beta_m^*\ge\bm 0$ be a dual ray certifying this infeasibility. Then the supremum $v_m^*$ in~\eqref{PricingMasterProblem-infeasible} is strictly positive, and augmenting $\hat A_m$ with any $\tilde{\bm a}_m\in\mathcal A_m$ having strictly positive ray objective invalidates the current certificate.
\end{proposition}

\begin{proof}
Since the restricted PMP over $\hat{A}_m = \{\bm{a}_{mk}\}_{k=0}^{K_m}$ is a finite linear program, Farkas' Lemma yields a dual ray $(\alpha_m^*, \bm{\beta}_m^*) \neq (0,0)$ with $\beta_{mt}^* \geq 0$ such that
\begin{align}
    \alpha_m^* + \sum_{t \in [T_m]} \beta_{mt}^*\, g_t(\bm{a}_{mk}) &\geq 0, 
    \quad \forall k = 0,\ldots,K_m, \label{eq:ray_cols}\\
    \alpha_m^* + \sum_{t \in [T_m]} \beta_{mt}^*\, \gamma_{mt} &< 0. \label{eq:ray_rhs}
\end{align}
Suppose, for contradiction, that $\alpha_m^* + \sum_{t \in [T_m]} \beta_{mt}^*\, g_t(\bm{a}) \geq 0$ for all $\bm{a} \in \mathcal{A}_m$. Since $\mathcal P_m$ is nonempty, take any $\mathbb P \in \mathcal{P}_m$ and then take expectations to obtain
\[
    0 \;\le\; \alpha_m^* + \sum_{t \in [T_m]} \beta_{mt}^*\, \mathbb{E}_{\mathbb P}[g_t(\bm{a})]
    \;\le\; \alpha_m^* + \sum_{t \in [T_m]} \beta_{mt}^*\, \gamma_{mt} \;<\; 0,
\]
where the second inequality uses $\beta_{mt}^* \ge 0$ with $\mathbb{E}_{\mathbb P}[g_t(\bm{a})] \leq \gamma_{mt}$ from \eqref{p_gm_ambiguity}, 
and the last step is \eqref{eq:ray_rhs}. This contradiction gives $v_m^* > 0$.

For the second claim, any $\tilde{\bm{a}}_m \in \mathcal{A}_m$ with 
$-\alpha_m^* - \sum_{t \in [T_m]} \beta_{mt}^*\, g_t(\tilde{\bm{a}}_m) > 0$ 
violates \eqref{eq:ray_cols} at the new column, so $(\alpha_m^*, \bm{\beta}_m^*)$ is no longer a valid infeasibility certificate for the augmented restricted PMP. Such a point exists because $v_m^* > 0$.
\end{proof}

When the PMP is feasible, its current distribution gives 
$\omega_m^*(\bm x^*)\le\Omega_m^*(\bm x^*)$. If $\pi_m^*$ is the exact PSP supremum, then 
$f_m(\bm a,\bm x^*)\le\alpha_m^*+\pi_m^*+\sum_{t\in[T_m]}\beta_{mt}^*g_t(\bm a)$ 
for every $\bm a\in\mathcal A_m$. Integrating this inequality over any 
$\mathbb P_m\in\mathcal P_m$ and using PMP strong duality yields
\[
\omega_m^*(\bm x^*)\le\Omega_m^*(\bm x^*)\le
U_m:=\alpha_m^*+\pi_m^*+\sum_{t\in[T_m]}\beta_{mt}^*\gamma_{mt}
=\omega_m^*(\bm x^*)+\pi_m^*.
\]

\begin{theorem}[Convergence of Algorithm~\ref{SPCG}]
\label{ConvergeProofSPCG}
Fix $m\in[M]$, $\bm x^*\in\mathcal X$, and $\epsilon>0$. Suppose Assumptions~\ref{assumption:regularity} and~\ref{assumption:light-tail} hold, and Algorithm~\ref{SPCG} reaches a feasible restricted PMP at some finite iteration $l_0$. From iteration $l_0$ onward, suppose every restricted PMP primal--dual pair is solved globally, every PSP value is computed exactly, and all generated columns are retained. For $l\ge l_0$, define
\[
\rho_l(\bm a):=f_m(\bm a,\bm x^*)-\alpha_m^{*(l)}-\sum_{t\in[T_m]}\beta_{mt}^{*(l)}g_t(\bm a).
\]
For $R>0$, define $K_R:=\{\bm a\in\mathcal A_m:g_{t_0}(\bm a)\le R\}$. Assumption~\ref{assumption:regularity}(i) makes $\mathcal A_m$ closed, while Assumption~\ref{assumption:light-tail} makes $g_{t_0}$ continuous and coercive, so $K_R$ is compact. Assume that for some $R>0$, $\displaystyle \sup_{l\ge l_0}\sup_{\bm a\in\mathcal A_m:\,g_{t_0}(\bm a)>R}\rho_l(\bm a)\le\epsilon/2$ and the family $\{\rho_l|_{K_R}\}_{l\ge l_0}$ is equicontinuous. Then Algorithm~\ref{SPCG} terminates in finitely many iterations and returns an $\epsilon/2$-optimal solution to the subproblem~\eqref{SubProblem}. Specifically, the returned distribution $\mathbb P_m^{\mathrm{PMP}}\in\mathcal P_m$ and value $\omega_m^*(\bm x^*)$ satisfy
\[
\Omega_m^*(\bm x^*)-\epsilon/2
\le \omega_m^*(\bm x^*)
=\mathbb E_{\mathbb P_m^{\mathrm{PMP}}}[f_m(\bm a_m,\bm x^*)]
\le \Omega_m^*(\bm x^*).
\]
\end{theorem}

\begin{proof}
\noindent\textbf{\textup{(i)} Correctness upon termination.}
Suppose Algorithm~\ref{SPCG} terminates at iteration $\bar l\ge l_0$. The returned $\mathbb P_m^{\mathrm{PMP}}$ is feasible for the restricted PMP and therefore belongs to $\mathcal P_m$. Hence
\[
\omega_m^*(\bm x^*)
=\mathbb E_{\mathbb P_m^{\mathrm{PMP}}}[f_m(\bm a_m,\bm x^*)]
\le \Omega_m^*(\bm x^*).
\]
Because the PSP value is computed exactly, the stopping rule implies $\rho_{\bar l}(\bm a)\le\epsilon/2$ for every $\bm a\in\mathcal A_m$. Thus $(\alpha_m^{*(\bar l)}+\epsilon/2,\bm\beta_m^{*(\bar l)})$ is feasible for the full moment dual. Weak duality for the full subproblem and strong duality for the finite restricted PMP give
\[
\Omega_m^*(\bm x^*)
\le \alpha_m^{*(\bar l)}+\epsilon/2
+\sum_{t\in[T_m]}\beta_{mt}^{*(\bar l)}\gamma_{mt}
=\omega_m^*(\bm x^*)+\epsilon/2.
\]
Together, these inequalities establish the stated $\epsilon/2$-optimality certificate.

\smallskip
\noindent\textbf{\textup{(ii)} Finite termination.}
Suppose, for contradiction, that infinitely many columns are generated at iterations $l\ge l_0$. At each such iteration, the continuation rule chooses an actual point $\tilde{\bm a}_m^{(l)}$ satisfying $\rho_l(\tilde{\bm a}_m^{(l)})>\epsilon/2$. The uniform-tail condition places every such point in $K_R$.

By equicontinuity, there exists $\delta>0$ such that $\|\bm a-\bm a'\|\le\delta$ implies $|\rho_l(\bm a)-\rho_l(\bm a')|<\epsilon/2$ for every $l\ge l_0$. Since all generated columns are retained, dual feasibility at any later iteration $l'>l$ gives $\rho_{l'}(\tilde{\bm a}_m^{(l)})\le0$, whereas $\rho_{l'}(\tilde{\bm a}_m^{(l')})>\epsilon/2$. Hence the generated points are pairwise $\delta$-separated in compact $K_R$, a contradiction. Thus only finitely many columns are generated from iteration $l_0$ onward. Since $l_0$ is finite, Algorithm~\ref{SPCG} terminates finitely.
\end{proof}

\begin{remark}[Exactness]
For the zero-tolerance variant, if Algorithm~\ref{SPCG} terminates at an iteration $l$ with $\sup_{\bm a\in\mathcal A_m}\rho_l(\bm a)\le0$, then $\omega_m^*(\bm x^*)=\Omega_m^*(\bm x^*)$. Theorem~\ref{ConvergeProofSPCG} does not prove finite termination for $\epsilon=0$, because the positive separation used in its proof disappears.
\end{remark}

Proposition~\ref{prop:SP-O1} establishes an exact finite-support formulation for Oracle~1, while Corollary~\ref{cor:SP-O1-oracle} provides arbitrarily accurate feasible finite-support certificates when attainment is not guaranteed. For Oracle~2, Theorem~\ref{ConvergeProofSPCG} establishes finite termination under its stated conditions. Together, these results extend the primal oracle analysis to generalized-moment single-stage DRO over closed, potentially unbounded sample spaces without requiring the entire sample space to be compact.

From a computational perspective, the CG structure provides greater flexibility by decoupling scenario selection from probability optimization. Once a finite support is fixed, the probability assignment in the PMP reduces to a linear program, while the scenario search in the PSP can accommodate nonlinear or higher-order structures in the sample space. Actually, although high-order polynomial moment PSPs have been reported to be challenging in \cite{mehrotra2014cutting}, our numerical experiments indicate that second-order PSPs can be solved efficiently in practice. Detailed computational results are reported in Section~\ref{Numerical-Basic-DRO}.

\subsubsection{Computational Enhancements} \label{sec:Enhancement}

In practice, the efficiency of BiCS is shaped by two distinct bottlenecks. First, every iteration of Algorithm~\ref{BiCSalgorithm} appends new scenarios to the MP-DRO \eqref{MasterProblem}, so the MP-DRO can grow rapidly if all scenarios produced by the subproblem are kept. Second, when Oracle~2 is used, the CG procedure solves the subproblem to the prescribed accuracy for each fixed $\bm x^*$, even though the BiCS outer loop only needs to determine whether the full DRO constraint $\Omega_m^*(\bm x^*)\le b_m$ is satisfied. Without an outer-loop-aware stopping rule, CG can generate many scenarios whose only effect is to refine $\Omega_m^*(\bm x^*)$ beyond what is needed. The two enhancements below target these two bottlenecks separately.

One is scenario reduction, which is applicable to both oracles. It helps control the size of MP-DRO~\eqref{MasterProblem} by retaining the newly generated scenarios in a more selective fashion. One straightforward way is to retain only scenarios carrying positive probabilities, which is implemented as the default in our empirical study. In fact, we might have multiple optimal oracle solutions. Hence, a mixed-integer program that preserves optimality but minimizes the number of scenarios with positive probabilities can be designed and employed, as described in Appendix~\ref{appendix:scenario_reduction}.

Another one is bound-based early stopping for Oracle~2. This enhancement is specific to the CG-based Oracle~2. For the BiCS outer loop, CG needs only return a feasible distribution requiring a support update, an $\epsilon/2$-accurate feasible distribution, or an upper bound certifying that no update is required. Further scenario generation is unnecessary once any such return is available.

The bound identity above gives the priority rule. If $\omega_m^*>b_m+\epsilon/2$, the current PMP distribution is returned for a support update. Otherwise, if $\pi_m^*\le\epsilon/2$, the gap $U_m-\omega_m^*=\pi_m^*$ certifies that the current PMP value is $\epsilon/2$-accurate. Otherwise, if $U_m\le b_m+\epsilon/2$, the upper bound certifies that no update is required. Only when none applies is an actual scenario with reduced cost greater than $\epsilon/2$ added. Algorithm~\ref{EnhancedSPCG} replaces Step~3.3 of Algorithm~\ref{SPCG} with this rule.

\begin{algorithm}[H]
\caption{\textbf{\textit{Bound-based stopping criterion for SP - Oracle 2 (CG)}}}
\label{EnhancedSPCG}
\begin{algorithmic}
\State \textbf{Step 3.3*: Stopping Criterion with Bounds.}
\Statex \hspace{5em} Set $\mathbb P_m^{\mathrm{PMP}}:=\sum_{k=0}^{K_m}P_{mk}^*\delta_{\bm a_{mk}}$ and $U_m:=\alpha_m^*+\pi_m^*+\sum_{t\in[T_m]}\beta_{mt}^*\gamma_{mt}=\omega_m^*+\pi_m^*$.
\Statex \hspace{5em} \textbf{if} $\omega_m^*>b_m+\epsilon/2$ \textbf{then} return $(\mathbb P_m^{\mathrm{PMP}},\omega_m^*)$ to Step~3 for a support update.
\Statex \hspace{5em} \textbf{else if} $\pi_m^*\le\epsilon/2$ \textbf{then} return $(\mathbb P_m^{\mathrm{PMP}},\omega_m^*)$ as an $\epsilon/2$-accurate oracle value.
\Statex \hspace{5em} \textbf{else if} $U_m\le b_m+\epsilon/2$ \textbf{then} return $U_m$ as an upper-bound certificate.
\Statex \hspace{5em} \textbf{else} choose any $\tilde{\bm a}_m\in\mathcal A_m$ whose reduced cost exceeds $\epsilon/2$, update $\hat A_m\gets\hat A_m\cup\{\tilde{\bm a}_m\}$, $K_m\gets K_m+1$, $l\gets l+1$, and return to Step~3.1.
\end{algorithmic}
\end{algorithm}

\begin{remark}
By the bound identity above, the first branch of Algorithm~\ref{EnhancedSPCG} returns a valid support-update distribution, the second returns the standard $\epsilon/2$-accurate oracle value, and the third directly certifies $\Omega_m^*(\bm x^*)\le U_m\le b_m+\epsilon/2$ and is therefore treated as satisfying the row-wise stopping test. Hence the proof of Theorem~\ref{convergence-BiCS} extends directly to the enhanced rule, with the upper-bound inequality replacing the oracle approximation inequality for such rows; whenever standard CG terminates, the enhanced rule stops no later.
\end{remark}

Together, the two enhancements address the two practical costs of BiCS, i.e., master-problem growth and subproblem effort, without altering the applicable guarantees.

\section{Almost Sure Distributionally Robust Optimization} \label{sec:AS_DRO}

Although rarely stated explicitly, AS-DRO arises naturally when a nonempty violation event in the sample space has probability zero under every admissible distribution. Additional support or mass restrictions, including restrictions incorporated through local information, may make the support induced by the ambiguity set a strict subset of the sample space. We examine how local information is reflected in the construction of ambiguity sets in Section~\ref{sec:local_info}.

Building on the standard DRO literature \cite{wiesemann2014distributionally, rahimian2019distributionally, kuhn2025distributionally}, it is widely understood that when the ambiguity set contains all probability distributions supported on the sample space $\mathcal{A}$, the resulting DRO formulation reduces to classical RO. While the statement itself is correct, it is sometimes read in a way that implicitly equates the sample space with the support. However, generally speaking, the sample space and the support induced by an ambiguity set do not have to coincide. For example, when Wasserstein ambiguity is combined with additional support or mass restrictions, admissible distributions may concentrate on a strict subset of the sample space. Therefore, treating the sample space and the support as interchangeable hides a modeling regime in which feasibility is enforced almost surely under every admissible distribution rather than pointwise on the full sample space.

Motivated by this observation, we define \emph{AS-DRO}, where a decision $\bm x$ is feasible only if the violation event has probability zero under every admissible distribution. This formulation fills the intermediate gap between RO and standard DRO. For instance, in a supply chain context subject to rare but extreme demand spikes, RO protects against the worst-case spike regardless of its probability, whereas standard DRO may remain feasible when the spike carries negligible probability and therefore has little effect on the worst-case expectation. In contrast, AS-DRO rules out any decision for which an admissible distribution assigns positive probability to a violation.  

From a computational perspective, AS-DRO introduces additional challenges. As we saw in Section \ref{sec:DRO}, standard DRO often leads to convex formulations under common ambiguity sets. In contrast, AS-DRO enforces essential-supremum feasibility over the support, and the resulting formulation may be nonconvex depending on the structure of $f_m$ and $\mathcal X$. Moreover, AS-DRO has rarely been studied as a standalone optimization problem. Together, this possible nonconvexity and the lack of dedicated study substantially limit solution approaches. However, these challenges do not affect the underlying modeling structure. The worst-case structure of AS-DRO remains distribution-based and admits the same primal decomposition logic developed for standard DRO.

\subsection{AS-DRO Formulation}

We now formalize the AS-DRO model and discuss its structural properties. For ease of comparison, AS-DRO is defined on the same sample space and ambiguity sets as standard DRO, but differs in how feasibility is enforced. Rather than requiring constraints to hold in expectation or uniformly over the entire sample space, AS-DRO requires each constraint to be satisfied almost surely on the support induced by the ambiguity set.

\begin{definition}\label{def:Marginal-ambiguity}
For each constraint \(m\in[M]\), let \((\mathcal A_m,\mathcal F_m)\) be a measurable space and let \(\mathcal P_m\) be an ambiguity set of probability measures on \((\mathcal A_m,\mathcal F_m)\). Define \(\mathcal A=\prod_{m=1}^M\mathcal A_m\) and \(\mathcal F=\bigotimes_{m=1}^M\mathcal F_m\). For any \(\mathbb P\in\mathcal M_+(\mathcal A,\mathcal F)\), let \(\mathbb P^m\) denote its marginal on \((\mathcal A_m,\mathcal F_m)\), so that \(\mathbb P^m(B)=\mathbb P(\mathcal A_1\times\cdots\times\mathcal A_{m-1}\times B\times\mathcal A_{m+1}\times\cdots\times\mathcal A_M)\) for all \(B\in\mathcal F_m\). The joint ambiguity set induced by the marginal ambiguity sets is
\[
\mathcal P=\bigl\{\mathbb P\in\mathcal M_+(\mathcal A,\mathcal F)\,\bigm|\,\mathbb P(\mathcal A)=1,\ \mathbb P^m\in\mathcal P_m,\ \forall m\in[M]\bigr\}.
\]
\end{definition}

Throughout this section, for any ambiguity set $\mathcal Q$ and measurable event $E$, we write $\mathbf 1_E \stackrel{\mathbb Q\text{-a.s.}}{=} 1$ for every $\mathbb Q\in\mathcal Q$ to mean that $\mathbb Q(E)=1$ for every $\mathbb Q\in\mathcal Q$, or equivalently, $\inf_{\mathbb Q\in\mathcal Q}\mathbb Q(E)=1$ and $\sup_{\mathbb Q\in\mathcal Q}\mathbb Q(E^c)=0$. Accordingly, the AS-DRO constraints can be written as follow
\begin{equation}\label{as_dro-c}
    \mathbf{1}_{\{f_m(\bm{a}_m,\bm{x})\leq b_m\}}
    \stackrel{\mathbb{P}_m\text{-a.s.}}{=} 1,
    \quad \forall \mathbb{P}_m\in\mathcal{P}_m,\ \forall m\in[M].
\end{equation}

In particular, AS-DRO lies between RO and standard DRO in how strictly feasibility is enforced. This relationship is illustrated in Figure~\ref{Risk Measure Comparison}. RO represents one extreme, requiring pointwise (supremum-based) feasibility over the entire sample space. At the other extreme, standard DRO enforces feasibility only in worst-case expectation. By contrast, AS-DRO requires feasibility almost surely under every distribution in the ambiguity set, which corresponds to protection against the essential supremum of the constraint violation. Theorem~\ref{RO_DRO_Inclusion} formalizes this comparison by showing how the feasible regions of RO, AS-DRO, and standard DRO are nested.

\begin{theorem}\label{RO_DRO_Inclusion}
    Let \(\mathcal{X}_{\text{RO}}, \mathcal{X}_{\text{AS-DRO}}, \mathcal{X}_{\text{DRO}}\) denote the feasible regions for RO, AS-DRO, and standard DRO, respectively. Then, the following inclusion holds:
    \[
    \mathcal{X}_{\text{RO}} \subseteq  \mathcal{X}_{\text{AS-DRO}} \subseteq \mathcal{X}_{\text{DRO}},
    \]
    where:
    \begin{itemize}
        \item \(\mathcal{X}_{\text{RO}} = \left\{\bm{x} \in \mathcal{X} \mid f_m(\bm{a}_m, \bm{x}) \leq b_m, \, \forall \bm{a}_m \in \mathcal{A}_m, \, \forall m \in [M]\right\},\)
        \item \(\mathcal{X}_{\text{AS-DRO}} = \left\{\bm{x} \in \mathcal{X} \mid \inf_{\mathbb{P}_m \in \mathcal{P}_m} \mathbb{P}_m\left[f_m(\bm{a}_m, \bm{x}) \leq b_m\right] \stackrel{\mathbb{P}_m\text{-a.s.}}{=} 1, \, \forall m \in [M]\right\},\)
        \item \(\mathcal{X}_{\text{DRO}} = \left\{\bm{x} \in \mathcal{X} \mid \sup_{\mathbb{P}_m \in \mathcal{P}_m} \mathbb{E}_{\mathbb{P}_m}\left[f_m(\bm{a}_m, \bm{x})\right] \leq b_m, \, \forall m \in [M]\right\}.\)
    \end{itemize}
\end{theorem}

\begin{proof} 
\textbf{\(\mathcal{X}_{\text{RO}} \subseteq \mathcal{X}_{\text{AS-DRO}}\).}  If \(\bm x\in\mathcal X_{\mathrm{RO}}\), then \(f_m(\bm a_m,\bm x)\le b_m\) for every \(\bm a_m\in \mathcal A_m\); hence, for any \(\mathbb P_m\in\mathcal P_m\),
\(\mathbb P_m[f_m(\bm a_m,\bm x)\le b_m]=1\stackrel{\mathbb{P}_m\text{-a.s.}}{=} 1\). Thus \(\bm x\in\mathcal X_{\mathrm{AS-DRO}}\).\\
\textbf{\(\mathcal{X}_{\text{AS-DRO}} \subseteq \mathcal{X}_{\text{DRO}}\).} If \(\bm x\in\mathcal X_{\mathrm{AS-DRO}}\), the probability-one convention gives \(\mathbb P_m[f_m(\bm a_m,\bm x)\le b_m]\stackrel{\mathbb{P}_m\text{-a.s.}}{=}1\) for every \(\mathbb P_m\in\mathcal P_m\). Thus \(f_m(\bm a_m,\bm x)\le b_m\), \(\mathbb P_m\)-almost surely. The standing integrability condition makes the expectation well defined, so \(\mathbb E_{\mathbb P_m}[f_m(\bm a_m,\bm x)]\le b_m\) for every \(\mathbb P_m\in\mathcal P_m\); taking the supremum gives \(\bm x\in\mathcal X_{\mathrm{DRO}}\).
\end{proof}

Theorem~\ref{RO_DRO_Inclusion} confirms that the AS-DRO feasible region lies between those of RO and standard DRO. Whether AS-DRO indeed leads to a distinct feasible region, however, depends on how the ambiguity set allocates probability mass across scenarios. The following remark characterizes when the inclusions are strict, and the subsequent examples in Tables~\ref{tab:discrete_example} and~\ref{tab:continuous_example} illustrate these cases.

\begin{remark}\label{rem:strict}
The first inclusion is strict exactly when there exists \(\bm x\in\mathcal X_{\mathrm{AS-DRO}}\setminus\mathcal X_{\mathrm{RO}}\). Equivalently, with \(V_m(\bm x):=\{\bm a_m\in\mathcal A_m:f_m(\bm a_m,\bm x)>b_m\}\), every \(V_m(\bm x)\) is null under every \(\mathbb P_m\in\mathcal P_m\), while \(V_{m_0}(\bm x)\ne\varnothing\) for at least one row \(m_0\). The second inclusion is strict exactly when there exists \(\bm x\in\mathcal X_{\mathrm{DRO}}\setminus\mathcal X_{\mathrm{AS-DRO}}\): all worst-case expectation constraints are feasible, but \(\mathbb P_{m_0}(V_{m_0}(\bm x))>0\) for some row \(m_0\) and some \(\mathbb P_{m_0}\in\mathcal P_{m_0}\).
\end{remark}

\begin{table}[h]
\centering
\begin{tabular}{lccc}
\toprule
Model & Constraint & $\mathcal X$ under $\mathcal P_1$ & $\mathcal X$ under $\mathcal P_2$\\
\midrule
RO   & $\max_{a\in \mathcal A} ax\le1$                                   & $x\le\frac13$ & $x\le\frac13$\\[1mm]
DRO  & $\max_{P\in\mathcal P_i}\;\mathbb E_P[a]\,x\le1$                & $x\le 1$      & $x\le\frac12$\\[1mm]
AS-DRO & $\inf_{P\in\mathcal P_i} \mathbb E_P[\mathbf1_{\{ax\le1\}}]\ge1_-$ & $x\le 1$      & $x\le\frac13$\\
\bottomrule
\end{tabular}
\caption{Discrete sample space $\mathcal{A}=\{1,3\}$ with $x\in\mathbb R_+$ and moment inequality ambiguity sets: $\mathcal{P}_1=\Bigl\{P:\;P_1+P_3=1,\;P_1+3P_3\le 1\Bigr\} \text{ and } \mathcal{P}_2=\Bigl\{P:\;P_1+P_3=1,\;P_1+3P_3\le 2\Bigr\}.$
}
\label{tab:discrete_example}
\end{table}

\begin{table}[h]
\centering
\begin{tabular}{lccc}
\toprule
Model & Constraint & Feasible range & Largest feasible $x$\\
\midrule
RO   & $\max_{a\in \mathcal{A}} ax \le 1$                                       & $x\le\frac13$        & $\tfrac13$\\
DRO  & $\max_{P\in\mathcal P_W}\mathbb E_P[a]\,x\le1$                    & $x\le\frac1{2.5}$    & $\tfrac1{2.5}$\\
AS-DRO & $\inf_{P\in\mathcal P_W}\mathbb E_P[\mathbf1_{\{ax\le1\}}]\ge1_-$ & $x\le\frac1{2.8}$    & $\tfrac1{2.8}$\\
\bottomrule
\end{tabular}
\caption{$\mathcal{A}=[1,3]$, $x\in\mathbb R_+$, and Wasserstein ambiguity set $\mathcal{P}_W=\Bigl\{P\text{ a Borel probability measure on }[1,3] :\; \mathcal{W}_1(P,\delta_{2})\le 0.5,\; P\bigl([2.8,3]\bigr)=0\Bigr\},$ with ground cost $|a-\hat a|$ and Dirac measure $\delta_{2}$ at $2$.}
\label{tab:continuous_example}
\end{table}

Tables~\ref{tab:discrete_example} and~\ref{tab:continuous_example} illustrate Theorem~\ref{RO_DRO_Inclusion} and clarify when the inclusions become strict. In both examples, the RO feasible region is contained in that of AS-DRO, which is contained in that of standard DRO. Specifically, the Wasserstein example demonstrates that both inclusions can be strict when the ambiguity set allows probability mass to concentrate on a subset of scenarios. These cases highlight that the position of AS-DRO between RO and standard DRO is structural, but the strictness of the inclusions depends on the ambiguity set and its induced support.

\subsection{Solution Scheme}

Having clarified the modeling role of AS-DRO and its relationship with RO and standard DRO, we now turn to its solution methodology. The starting point of our approach is the close structural proximity between AS-DRO and standard DRO, which suggests that solution schemes developed for standard DRO in Section \ref{sec:DRO} can be extended to the AS-DRO setting with suitable adaptations.

To make this connection explicit, we first show that AS-DRO admits a worst-case-expectation formulation with the same structure as standard DRO. In particular, the almost-sure feasibility requirement can be enforced by reformulating the constraint as a worst-case expectation over a nonnegative violation measure. Specifically, for each constraint \(m\in[M]\), AS-DRO enforces zero worst-case violation probability, \(\sup_{\mathbb P_m\in\mathcal P_m}\mathbb P_m[f_m(\bm a_m,\bm x)>b_m]\le0\). Defining the nonnegative violation slack \(t_m(\bm a_m,\bm x):=[f_m(\bm a_m,\bm x)-b_m]_+=\max\{f_m(\bm a_m,\bm x)-b_m,0\}\), AS-DRO reduces to
\begin{subequations}\label{eq:ASDRO-sup}
\begin{align}
    \textbf{(AS-DRO) }\min_{\bm x\in\mathcal X}\quad &\bm c^\top\bm x,\\
    \text{s.t.}\quad
        &\sup_{\mathbb P_m\in\mathcal P_m}
            \mathbb E_{\mathbb P_m}[t_m]\le0,
          \quad \forall m\in[M].
\end{align}
\end{subequations}

Because \(t_m\ge0\), its expectation is zero under every \(\mathbb P_m\in\mathcal P_m\) if and only if the violation event is \(\mathbb P_m\)-null under every such distribution. Moreover, \(t_m\le |f_m|+|b_m|\), so its integrability follows from the standing integrability condition on \(f_m\). Since the only departure from standard DRO is the \([\cdot]_+\) operator, AS-DRO can be addressed using both the dual-perspective reformulation and the primal BiCS framework developed in Section~\ref{sec:DRO}, with only minor adjustments. The following sections present these two solution perspectives in turn.

\subsubsection{Reformulation Approach - Dual Perspective}\label{sec:ASDRO-dual}

We begin by examining AS-DRO from a dual perspective. The positive part satisfies \([z]_+=\max\{z,0\}\), so its semi-infinite epigraph condition separates exactly into two maximizations over the full sample space. Under the same probability-measure and sample-space duality conditions as in Theorem~\ref{linear-dro}, both maximizations admit conic representations, yielding the following robust counterpart.

\begin{proposition} \label{AS-dro-dual}
    Suppose that the conic Slater and strict moment feasibility conditions of Theorem~\ref{linear-dro} hold. 
    Let \(f_m(\bm{a}_m, \bm{x}) = \bm{a}_m^\top \bm{x}\), 
    \(g_t(\bm{a}_m) = \bm e_{mt}^\top\bm{a}_m\), 
    and 
    \(\mathcal{A}_m = \{\bm{a}_m \in \mathbb{R}^n : \bm{D}_m\,\bm{a}_m \preceq_{\mathcal{K}_m} \bm{d}_m\}\),
    where \(\mathcal{K}_m\) is a proper cone.  
    Then, the robust counterpart of \eqref{eq:ASDRO-sup} can be formulated as:
    \begin{subequations}
    \begin{align}
        \min_{\bm{x} \in \mathcal{X}}
          & \quad \bm{c}^\top \bm{x}, \\[4pt]
        \text{s.t.}\quad 
          & \bm{\beta}_m^\top \bm{\gamma}_m + \alpha_m \;\le\; 0,
            \quad \forall\,m \in [M], \\[4pt]
          & \alpha_m \;\ge\; \bm{p}_{m1}^\top \bm{d}_m-b_m,
            \quad \forall\,m \in [M], \\[4pt]
          & \bm D_m^\top\bm{p}_{m1}
               \;=\; \bm{x} \;-\; \bm{E}_m^\top \bm{\beta}_m,
            \quad \forall\,m \in [M], \\[4pt]
          & \alpha_m \;\ge\; \bm{p}_{m2}^\top \bm{d}_m,
            \quad \forall\,m \in [M], \\[4pt]
          & \bm D_m^\top\bm{p}_{m2}
               \;=\; -\,\bm{E}_m^\top \bm{\beta}_m,
            \quad \forall\,m \in [M], \\[4pt]
          & \alpha_m\in\mathbb R,\quad
            \bm{p}_{m1},\;\bm{p}_{m2} \in \mathcal{K}_m^*,\quad
            \bm{\beta}_m \in \mathbb{R}_+^{T_m},
            \quad \forall\,m \in [M],
    \end{align}
    \end{subequations}
    where \(\mathcal{K}_m^*\) denotes the cone dual to \(\mathcal{K}_m\).
\end{proposition}
\begin{proof}
The probability-measure dual used in Theorem~\ref{linear-dro} gives the semi-infinite inequality
\begin{equation*}
  \alpha_m \;\ge\;
    \bigl(f_m(\bm{a}_m,\bm{x}) - b_m\bigr)_+
     \;-\;\!\sum_{t \in [T_m]} \beta_{mt}\,g_{t}(\bm{a}_m),
     \qquad \forall\bm a_m\in\mathcal A_m,
\end{equation*}
and the constraint \(\bm\beta_m^\top\bm\gamma_m+\alpha_m\le0\), with \(\bm\beta_m\ge0\). Since \([z]_+=\max\{z,0\}\), the semi-infinite inequality is equivalent to
\begin{subequations}
\begin{align}
  \alpha_m 
   & \;\ge\;
   \sup_{\bm{a}_m \in \mathcal{A}_m}
   \Bigl\{\,\bm a_m^\top\bm x-b_m
           \;-\;\sum_{t\in[T_m]}\!\beta_{mt}\,g_t(\bm{a}_m)\Bigr\}, 
   \label{AS-DRO-violate}\\
  \alpha_m 
   & \;\ge\;
   \sup_{\bm{a}_m \in \mathcal{A}_m}
   \Bigl\{-\,\sum_{t\in[T_m]}\!\beta_{mt}\,g_t(\bm{a}_m)\Bigr\}.
   \label{AS-DRO-satisfied}
\end{align}
\end{subequations}
Under the sample-space duality conditions of Theorem~\ref{linear-dro}, the two full-space suprema have dual certificates \(\bm p_{m1}\) and \(\bm p_{m2}\), respectively, with \(\bm D_m^\top\bm p_{m1}=\bm x-\bm E_m^\top\bm\beta_m\) and \(\bm D_m^\top\bm p_{m2}=-\bm E_m^\top\bm\beta_m\). Substitution gives the displayed formulation.
\end{proof}

Overall, the dual reformulation in Proposition~\ref{AS-dro-dual} handles the non-smooth operator $[\,f_m(\bm a_m,\bm x)-b_m]_+$ exactly through the two full-space inequalities \eqref{AS-DRO-violate} and \eqref{AS-DRO-satisfied}, with separate dual certificates linked through \(\alpha_m\). Although this reformulation yields a compact conic representation, the additional dual certificates may limit scalability when solving large-scale instances with commercial solvers. Motivated by the similar worst-case structure of AS-DRO and standard DRO, we next consider AS-DRO from a primal perspective that avoids explicit dualization. Building on solution ideas developed for standard DRO (Section~\ref{BiCS section}), the preserved worst-case structure of AS-DRO can be naturally adapted within the BiCS framework.

\subsubsection{BiCS Framework — Primal Perspective}\label{sec:ASDRO-BiCS}

From a primal perspective, our key observation is that AS-DRO and standard DRO share the same underlying worst-case protection structure, differing only in the form of the protection criterion: standard DRO guards against the worst-case expectation, whereas AS-DRO guards against the essential supremum. If we view the worst-case protection mechanism as a single component, then AS-DRO and standard DRO share the same high-level optimization structure. This observation leads us to approach AS-DRO through an iterative cut-generation perspective, in direct analogy with standard DRO. Rather than reformulating the worst-case protection via dual variables, we treat the worst-case distribution as an object that can be revealed incrementally and incorporated through primal cuts. Clearly, this cut-generation operation can be naturally handled within the BiCS framework (Section~\ref{BiCS section}), leading to a decomposition into a master problem and a subproblem. We next outline how the corresponding master problem and subproblem are adapted to accommodate the AS-DRO setting.

\paragraph{Solving Master Problem}  

We begin by describing how the MP in the BiCS framework is adapted to accommodate the AS-DRO setting. As in standard DRO, probability mass is assigned to scenarios in the current finite collection; AS-DRO changes only the quantity attached to each scenario, replacing \(f_m\) by its nonnegative violation slack. Under this representation, the MP can be formulated as follows.

\begin{subequations}\label{eq:ASDRO-master-refined}
\begin{align}
  \textbf{MP-ASDRO:} \hspace{2mm}\min_{\bm x \in\mathcal X,\;t_{mk}}\quad &\bm c^\top\bm x\\
  \text{s.t.}\quad
    &\sup_{{\mathbb P_{m} \in \mathcal{P}_m(\hat A_m)}} \sum_{k=0}^{K_m} P_{mk}\,t_{mk}\;\le\;0, \quad \forall m\in[M],\\
    &t_{mk}=\bigl[f_m(\bm a_{mk},\bm x)-b_m\bigr]_+,
      \quad k=0,\ldots,K_m,\;m\in[M].
\end{align}
\end{subequations}

As in MP-DRO, $\mathcal P_m(\hat A_m)$ is kept nonempty. If valid finite bounds are available, the positive-part relation can be represented by a standard Big-\(M\) construction.

Importantly, this change does not alter the solution approach for the embedded probability-assignment problem. Given a finite scenario collection \(\{\bm a_{mk}\}_{k=0,\ldots,K_m,\;m\in[M]}\), the violation term $t_{mk} = [\,f_m(\bm a_{mk}, \bm x) - b_m\,]_+ $ depends only on \(\bm x\). The maximization over probability weights therefore remains linear and admits the same LP-duality or KKT representations described in Section~\ref{BiCS Master Problem}.

\paragraph{Solving Subproblem}  

Given a candidate solution \(\bm x^*\) from the master problem, the role of the subproblem is to determine whether there exists a distribution in the ambiguity set under which the AS-DRO constraint is violated. This step reveals the key conceptual distinction in risk treatment between AS-DRO and standard DRO. In standard DRO, the subproblem searches over admissible distributions to maximize the expected constraint value. In contrast, AS-DRO imposes worst-case protection at the level of constraint violations, characterized by an essential-supremum criterion. 

Once \(\bm x^*\) is fixed, however, whether a realization induces a violation is fully determined by the scenario itself. Consequently, the subproblem can still operate over scenarios, but with a modified objective: identifying those scenarios that trigger violations under \(\bm x^*\). If no violating distribution is returned under the applicable oracle guarantee, no scenario is added. Otherwise, the returned distribution supplies scenarios for a distribution cut.

Specifically, each subproblem seeks a distribution in \(\mathcal P_m\) that maximizes the expected violation slack under a fixed candidate solution \(\bm x^*\). Under the corresponding conditions of Standard Oracle~1 in Section~\ref{BiCS SP}, its finite-support formulation becomes
\begin{equation}\label{SP-bilinear-ASDRO-all}
\textbf{SP-O1-ASDRO:}\quad
\Omega_{m,\mathrm{AS}}^*(\bm{x}^*)=
\max\left\{\sum_{i=1}^{T_m+1}\tilde P_i
\bigl[f_m(\tilde{\bm a}_i,\bm x^*)-b_m\bigr]_+
\;\middle|\;\eqref{SP-bilinear-start}\text{--}\eqref{SP-bilinear-end}\right\}.
\end{equation}

Alternatively, \emph{Oracle 2} applies the Standard DRO PMP/PSP construction with \(h_m(\bm a_m,\bm x):=[f_m(\bm a_m,\bm x)-b_m]_+\) in place of \(f_m\); the PMP and PSP formulations are summarized in Table~\ref{tab:RO_ASDRO_DRO_comparison}. The exact AS-DRO target is zero: a fixed positive tolerance certifies only the corresponding expected-slack bound, whereas a zero-tolerance run is exact only if it terminates.

\vspace{2mm}

In summary, we defined AS-DRO and provided its solution schemes in this section. We compared its level of conservativeness against pointwise RO and expectation‑based DRO, showing that AS-DRO provides intermediate protection as the ambiguity set’s support can differ from the sample space. We then embedded AS-DRO into the standard DRO solution framework, showing that it fits naturally within the BiCS primal decomposition with only minor adjustments, including when the resulting AS-DRO formulation is nonconvex. This observation broadens our understanding of the primal perspective and highlights its ability to handle more challenging problem classes. In the next section, we extend this primal approach to nonconvex and discontinuous settings, showing that the same primal-based BiCS framework continues to apply with only modest additional modifications.

\begin{landscape}
\begin{table}[htbp]
\centering
\footnotesize  
\renewcommand{\arraystretch}{1.2}      
\setlength{\tabcolsep}{3pt}           
\caption{Comparison of Robust Optimization (RO), Almost-Sure DRO (AS-DRO), and Standard DRO.}
\label{tab:RO_ASDRO_DRO_comparison}

\begin{tabular}{|>{\centering\arraybackslash}m{1.2cm}|
                >{\centering\arraybackslash}m{5.6cm}|
                >{\centering\arraybackslash}m{7.8cm}|
                >{\centering\arraybackslash}m{7.3cm}|}
\hline
\textbf{} & \textbf{RO} & \textbf{AS-DRO} & \textbf{Standard DRO} \\
\hline

\textbf{Primal}
& \parbox{5.6cm}{
\[
\begin{aligned}[t]
\min_{\bm{x}\in \mathcal{X}} \quad & \bm{c}^\top \bm{x},\\
\text{s.t.}\quad &
 \sup_{\bm{a}_m \in \mathcal{A}_m} f_m(\bm{a}_m,\bm{x}) \,\le b_m
\end{aligned}
\]
}
& \parbox{7.8cm}{
\[
\begin{aligned}[t]
\min_{\bm{x}\in \mathcal{X}} \quad & \bm{c}^\top \bm{x},\\
\text{s.t.}\quad &
 \sup_{\mathbb{P}_m\in \mathcal{P}_m} 
 \mathbb{E}_{\mathbb{P}_m}[\,t_m\,]\le 0,\\
& t_m = \max\{\,f_m(\bm{a}_m,\bm{x})-b_m,\;0\}
\end{aligned}
\]
}
& \parbox{7.3cm}{
\[
\begin{aligned}[t]
\min_{\bm{x}\in \mathcal{X}} \quad & \bm{c}^\top \bm{x},\\
\text{s.t.}\quad &
 \sup_{\mathbb{P}_m\in \mathcal{P}_m} 
 \mathbb{E}_{\mathbb{P}_m}[f_m(\bm{a}_m,\bm{x})] 
 \,\le b_m
\end{aligned}
\]
}
\\ \hline

\textbf{Dual}
& \parbox{5.6cm}{
\[
\begin{aligned}[t]
\min_{\bm{x}\in \mathcal{X}} \quad & \bm{c}^\top \bm{x},\\
\text{s.t.}\quad &
 \bm{p}_m^\top \bm{d}_m \,\le b_m,\\
& \bm{D}_m^\top \bm{p}_m = \bm{x},\\
& \bm{p}_m \in \mathcal{K}_m^*
\end{aligned}
\]
}
& \parbox{7.8cm}{
\[
\begin{aligned}[t]
\min_{\bm{x}\in \mathcal{X}} \quad & \bm{c}^\top \bm{x},\\
\text{s.t.}\quad &
 \bm{\beta}_m^\top \bm{\gamma}_m + \alpha_m \,\le 0,\\
& \alpha_m \,\ge \bm{p}_{m1}^\top \bm{d}_m-b_m,\\
& \bm D_m^\top \bm{p}_{m1}
    = \bm{x} - \bm{E}_m^\top \bm{\beta}_m,\\
& \alpha_m \,\ge \bm{p}_{m2}^\top \bm{d}_m,\\
& \bm D_m^\top \bm{p}_{m2}
    = -\,\bm{E}_m^\top \bm{\beta}_m,\\
& \bm{p}_{m1},\bm{p}_{m2} \in \mathcal{K}_m^*,\ 
  \bm{\beta}_m\in \mathbb{R}_+^{T_m},\ 
  \alpha_m\in\mathbb R
\end{aligned}
\]
}
& \parbox{7.3cm}{
\[
\begin{aligned}[t]
\min_{\bm{x}\in \mathcal{X}} \quad & \bm{c}^\top \bm{x},\\
\text{s.t.}\quad &
 \bm{\beta}_m^\top \bm{\gamma}_m + \bm{p}_m^\top \bm{d}_m \,\le b_m,\\
& \bm{D}_m^\top \bm{p}_m 
  = \bm{x} - \bm{E}_m^\top \bm{\beta}_m,\\
& \bm{p}_m\in \mathcal{K}_m^*,\ 
  \bm{\beta}_m\in \mathbb{R}_+^{T_m}
\end{aligned}
\]
}
\\ \hline

\textbf{MP}
& \parbox{5.6cm}{
\[
\begin{aligned}[t]
\min_{\bm{x}\in \mathcal{X}} \quad & \bm{c}^\top \bm{x},\\
\text{s.t.}\quad &
 f_m(\bm{a}_{mk},\bm{x}) \le b_m,  \bm{a}_{mk}\in \hat A_m
\end{aligned}
\]
}
& \parbox{7.8cm}{
\[
\begin{aligned}[t]
\min_{\bm{x}\in \mathcal{X}} \quad & \bm{c}^\top \bm{x},\\
\text{s.t.}\quad &
 \sup_{\mathbb P_m\in \mathcal{P}_m(\hat A_m)}
 \sum_{k=0}^{K_m} P_{mk}\,t_{mk}\,\le 0,\\
& t_{mk} = \max\{\,f_m(\bm{a}_{mk},\bm{x}) - b_m,\;0\},\quad k=0,\ldots,K_m
\end{aligned}
\]
}
& \parbox{7.3cm}{
\[
\begin{aligned}[t]
\min_{\bm{x}\in \mathcal{X}} \quad & \bm{c}^\top \bm{x},\\
\text{s.t.}\quad &
 \sup_{\mathbb P_m\in \mathcal{P}_m(\hat A_m)}
 \sum_{k=0}^{K_m} P_{mk}\,f_m(\bm{a}_{mk},\bm{x})
 \,\le b_m, \\
 & \bm{a}_{mk}\in \hat A_m
\end{aligned}
\]
}
\\ \hline

\textbf{SP}
& \parbox{5.6cm}{
\[
\begin{aligned}[t]
\sup \quad & f_m(\bm{\tilde{a}}_m,\bm{x}^*),\\
\text{s.t.}\quad & \bm{\tilde{a}}_m\in \mathcal{A}_m
\end{aligned}
\]
}
& \parbox{7.8cm}{
\[
\begin{aligned}[t]
\sup_{\mathbb{P}_m\in \mathcal{P}_m} \quad 
& \mathbb{E}_{\mathbb{P}_m}[\tilde{t}_m],\\
\text{s.t.}\quad &
 \tilde{t}_m = \max\{\,f_m(\bm{\tilde{a}}_m,\bm{x}^*)-b_m,\;0\},\\
& \bm{\tilde{a}}_m\in \mathcal{A}_m
\end{aligned}
\]
}
& \parbox{7.3cm}{
\[
\begin{aligned}[t]
\sup_{\mathbb{P}_m\in \mathcal{P}_m} \quad 
& \mathbb{E}_{\mathbb{P}_m}[\,f_m(\bm{\tilde{a}}_m,\bm{x}^*)],\\
\text{s.t.}\quad & \bm{\tilde{a}}_m\in \mathcal{A}_m
\end{aligned}
\]
}
\\ \hline

\textbf{PMP}
& Not applicable
& \parbox{7.8cm}{
\[
\max_{\mathbb P_m\in \mathcal{P}_m(\hat A_m)}\ 
 \sum_{k=0}^{K_m} P_{mk}\,t_{mk} (\bm x^*)
\]
}
& \parbox{7.3cm}{
\[
\max_{\mathbb P_m\in \mathcal{P}_m(\hat A_m)}\ 
 \sum_{k=0}^{K_m} P_{mk}\,f_m(\bm{a}_{mk},\bm{x}^*)
\]
}
\\ \hline

\textbf{PSP}
& Not applicable
& \parbox{7.8cm}{
\[
\begin{aligned}[t]
\sup \quad & \tilde{t}_m \,-\,\alpha_m^* 
    \,-\,\sum_{t\in[T_m]}\beta_{mt}^*\,g_t(\tilde{\bm{a}}_m),\\
\text{s.t.}\quad &
 \tilde{t}_m = \max\{\,f_m(\tilde{\bm{a}}_m,\bm{x}^*) - b_m,\;0\},\\
& \tilde{\bm{a}}_m \in \mathcal{A}_m
\end{aligned}
\]
}
& \parbox{7.3cm}{
\[
\begin{aligned}[t]
\sup \quad & f_m(\tilde{\bm{a}}_m,\bm{x}^*) \,-\,\alpha_m^*
    \,-\,\sum_{t\in[T_m]}\beta_{mt}^*\,g_t(\tilde{\bm{a}}_m),\\
\text{s.t.}\quad &
 \tilde{\bm{a}}_m \in \mathcal{A}_m
\end{aligned}
\]
}
\\ \hline

\end{tabular}

\vspace{0.5ex}
\begin{minipage}{\linewidth}
\footnotesize
\textbf{Note:} 
All constraints involving \(m\) are implicitly assumed to hold 
\(\forall\,m \in [M]\), unless otherwise stated; finite-scenario rows use \(k=0,\ldots,K_m\). The dual rows are stated for the linear-conic moment setting of Section~\ref{sec:DRO}. MP, SP, PMP, and PSP denote the master problem, subproblem, pricing master problem, and pricing subproblem, respectively.
\end{minipage}
\end{table}
\end{landscape}

\section{Extensions} \label{sec:extensions}

Building on the BiCS framework for formulations that may be nonconvex, including AS-DRO, we next examine how the same algorithmic principles extend to broader settings. In this section, we focus on two representative and practically important generalizations: (i) DRCCP and (ii) Standard DRO models incorporating local information. In both cases, the ambiguity set remains convex in the probability measure; the main difficulties instead arise from indicators of chance events or regional membership and the resulting formulations over the sample space or decision variables, which may be nonconvex or discontinuous. We show that these models nevertheless admit a natural primal decomposition aligned with the BiCS paradigm.

\subsection{Distributionally Robust Chance Constrained Programs} \label{sec:DRCCP}

We begin with DRCCP, a prototypical extension in which nonconvexity arises from probabilistic feasibility requirements under distributional ambiguity. In terms of risk treatment, DRCCP constitutes a natural progression beyond AS-DRO. Both settings apply a worst-case risk measure to an ambiguity set of distributions; however, they differ fundamentally in their treatment of constraint risk. AS-DRO imposes almost-sure feasibility through the essential supremum, whereas DRCCP uses a quantile to impose a high probability requirement. This shift provides a more flexible reliability requirement while introducing nonconvexity as an algorithmic challenge. As such, solving DRCCPs is important both from a practical modeling perspective and from an algorithmic standpoint.

Formally, for each group $m\in[M]$, let $I_m\in\mathbb N$ be the number of component constraints, let $h_m\in[I_m]$ be the minimum number that must be satisfied, and let $\epsilon_m\in[0,1)$. The vector $\bm a_m$ denotes the joint uncertainty for group $m$, and component $i\in[I_m]$ is described by $f_{mi}:\mathcal A_m\times\mathcal X\to\mathbb R$ and $b_{mi}\in\mathbb R$. We assume that $\mathcal P_m$ is nonempty and that $\{\bm a_m\in\mathcal A_m:f_{mi}(\bm a_m,\bm x)\le b_{mi}\}$ is measurable for every $\bm x\in\mathcal X$ and $i\in[I_m]$. The general DRCCP is
\begin{subequations}\label{chanceDRO}
\begin{align}
    \min \quad & \bm{c}^\top \bm{x}, \\
    \text{s.t.} \quad & \bm{x} \in \mathcal{X},\\
    & \inf_{\mathbb P_m \in \mathcal{P}_m} \,\mathbb{P}_m\Biggl\{
        \left|\left\{i\in[I_m]:
        f_{mi}(\bm a_m,\bm x)\le b_{mi}\right\}\right|
        \ge h_m
      \Biggr\}
      \;\ge\; 1 - \epsilon_m,
      \;\;\forall\, m \in [M].
\end{align}
\end{subequations}

The formulation \eqref{chanceDRO} encompasses several well-known models in DRCCP. When each ambiguity set $\mathcal{P}_m$ reduces to a singleton, \eqref{chanceDRO} reduces to the CCP \cite{charnes1959chance,nemirovski2007convex}. If each group contains a single constraint ($I_m = h_m = 1$), we recover the standard \emph{individual DRCCP}; if each group contains multiple constraints that must all hold simultaneously ($I_m = h_m \ge 2$), we obtain the classical \emph{joint DRCCP}. These two settings are widely studied, particularly under moment-based or Wasserstein ambiguity sets \cite{xie2018deterministic,zymler2013distributionally,xie2021distributionally}.

Furthermore, models involving \emph{multiple} DRCCPs ($M \ge 2$) have received less explicit theoretical attention, despite their relevance in applications such as multiperiod power systems and cancer treatment planning \cite{zaghian2018chance,wen2024multiple,huang2025distributionally}. DRCCP techniques predominantly focus on a single worst-case probability constraint because it facilitates tight reformulations and tractable characterizations. From a primal modeling perspective, however, aggregation is not intrinsic. Different stages, subsystems, or clinical objectives naturally induce distinct chance constraints, each with its reliability target $\epsilon_m$ and distributional information $\mathcal{P}_m$, as captured by formulation \eqref{chanceDRO}.

More generally, when $1 \le h_m < I_m$, we refer to the groupwise cardinality structure as a \emph{combinatorial DRCCP}, in which at least $h_m$ out of $I_m$ constraints must be satisfied. Because satisfaction may arise from many different subsets of component constraints, this structure complicates exact reformulation. A primal representation based on the satisfaction event keeps this event explicit and enables its direct integration into the BiCS framework under the conditions stated below. Table~\ref{Chance Cate} summarizes the DRCCP categories encompassed by our general formulation \eqref{chanceDRO}. We next discuss how these models can be addressed within the BiCS framework.

\begin{table}[ht!]
    \centering
    \caption{Classification of DRCCP Formulations}
    \label{Chance Cate}
    \begin{threeparttable}
    \small
    \renewcommand{\arraystretch}{1.12}
    \begin{tabular}{@{}p{0.28\textwidth}p{0.23\textwidth}p{0.40\textwidth}@{}}
    \toprule
    DRCCP category & Condition & Interpretation \\
    \midrule
    \multicolumn{3}{@{}l}{\emph{Number of chance constraints}} \\
    \hspace{1em}Single DRCCP & $M=1$ & One probabilistic requirement \\
    \hspace{1em}Multiple DRCCPs & $M\ge 2$ & Several reliability requirements \\
    \midrule
    \multicolumn{3}{@{}l}{\emph{Satisfaction structure within group $m$}} \\
    \hspace{1em}Individual DRCCP & $I_m=h_m=1$ & One inequality in each group \\
    \hspace{1em}Joint DRCCP & $I_m=h_m\ge 2$ & All inequalities in a group must hold \\
    \hspace{1em}Combinatorial DRCCP & $1\le h_m<I_m$ & At least $h_m$ of $I_m$ inequalities must hold \\
    \midrule
    \multicolumn{3}{@{}l}{\emph{Distributional information}} \\
    \hspace{1em}Classical CCP & $\mathcal P_m=\{\mathbb P_m^0\}$ for all $m$ & Known distribution \\
    \hspace{1em}DRCCP & $|\mathcal P_m|>1$ for some $m$ & Ambiguity over distributions \\
    \bottomrule
    \end{tabular}
    \begin{tablenotes}\footnotesize
    \item \textbf{Note:} $M$ is the number of chance constraints, $I_m$ is the number of inequalities in group $m$, and $h_m$ is the minimum number of inequalities that must be satisfied in that group. The distribution $\mathbb P_m^0$ denotes a known probability distribution.
    \end{tablenotes}
    \end{threeparttable}
\end{table}

Before presenting the solution scheme, we distinguish the inclusion chain among RO, AS-DRO, DRCCP-J, and DRCCP-I from the separate comparison between Standard DRO and DRCCP-I, whose feasible regions are generally unordered. This comparison complements Figure~\ref{Risk Measure Comparison} in positioning the models by their risk treatment. For the following inclusion comparison, we restrict attention to the individual and joint DRCCP settings. Here, DRCCP-J denotes the cross-row joint chance requirement over the $M$ individual rows ($I_m=1$ for all $m$); it is distinct from the within-group joint DRCCP characterized by $I_m=h_m\ge2$ in Table~\ref{Chance Cate}.

For this comparison, write $f_m:=f_{m1}$ and $b_m:=b_{m1}$. ~\eqref{eq:DRCCP-J} and ~\eqref{eq:DRCCP-I} impose, respectively,

\begin{align}
\inf_{\mathbb{P}\in\mathcal{P}}
\mathbb{P}\left(
    \bigcap_{m\in[M]}
    \left\{f_m(\boldsymbol{a}_m,\boldsymbol{x})\le b_m\right\}
\right)
&\ge 1-\epsilon,
\tag{DRCCP-J}\label{eq:DRCCP-J}
\\
\inf_{\mathbb{P}_m\in\mathcal{P}_m}
\mathbb{P}_m
\left\{f_m(\boldsymbol{a}_m,\boldsymbol{x})\le b_m\right\}
&\ge 1-\epsilon,
\qquad \forall m\in[M].
\tag{DRCCP-I}\label{eq:DRCCP-I}
\end{align}

\begin{theorem} \label{thm:RO_DRCCP_Inclusion}
Let $\mathcal X_{\mathrm{RO}}$, $\mathcal X_{\text{AS-DRO}}$, 
$\mathcal X_{\mathrm{DRCCP\text{-}J}}$, and $\mathcal X_{\mathrm{DRCCP\text{-}I}}$
denote the feasible regions of RO, AS-DRO, joint DRCCP, and individual DRCCP, respectively.
Assume $0<\epsilon<1$ and $\epsilon_m=\epsilon$ for all $m\in[M]$.
Let the nonempty ambiguity sets $\{\mathcal P_m\}_{m=1}^M$ and $\mathcal P$ satisfy
Definition~\ref{def:Marginal-ambiguity}.
Then the following inclusions hold:
\[
\mathcal X_{\mathrm{RO}}
\;\subseteq\;
\mathcal X_{\text{AS-DRO}}
\;\subseteq\;
\mathcal X_{\mathrm{DRCCP\text{-}J}}
\;\subseteq\;
\mathcal X_{\mathrm{DRCCP\text{-}I}}.
\]
\end{theorem}

\begin{proof}
For each $m\in[M]$, define $F_m(\bm x):=\{\bm a_m\in\mathcal A_m:f_m(\bm a_m,\bm x)\le b_m\}$ and its cylinder event $E_m(\bm x):=\{\bm a\in\mathcal A:\bm a_m\in F_m(\bm x)\}$. Let $E(\bm x):=\bigcap_{m\in[M]}E_m(\bm x)$.

\smallskip\noindent
\textbf{Step 1: \(\mathcal{X}_{\mathrm{RO}} \subseteq \mathcal{X}_{\text{AS-DRO}}\).}
This is identical to Step~1 of Theorem~\ref{RO_DRO_Inclusion}.

\smallskip\noindent
\textbf{Step 2: $\mathcal X_{\text{AS-DRO}}
\subseteq \mathcal X_{\mathrm{DRCCP\text{-}J}}$.}
Let $\bm x\in\mathcal X_{\text{AS-DRO}}$. By Proposition~\ref{prop:ASDRO-equivalence}, $\mathbb P(E(\bm x))=1$ for every $\mathbb P\in\mathcal P$. Hence
\[
\inf_{\mathbb P\in\mathcal P}\mathbb P(E(\bm x))=1\ge 1-\epsilon,
\]
which implies $\bm x\in\mathcal X_{\mathrm{DRCCP\text{-}J}}$.

\smallskip\noindent
\textbf{Step 3: $\mathcal X_{\mathrm{DRCCP\text{-}J}}
\subseteq \mathcal X_{\mathrm{DRCCP\text{-}I}}$.}
Let $\bm x\in\mathcal X_{\mathrm{DRCCP\text{-}J}}$ and fix $m\in[M]$. Since $E(\bm x)\subseteq E_m(\bm x)$,
\[
\inf_{\mathbb P\in\mathcal P}\mathbb P(E_m(\bm x))
\ge \inf_{\mathbb P\in\mathcal P}\mathbb P(E(\bm x))
\ge 1-\epsilon.
\]
Every $\mathbb P\in\mathcal P$ has an $m$th marginal in $\mathcal P_m$. Conversely, for any $\mathbb Q_m\in\mathcal P_m$, choose $\mathbb Q_j\in\mathcal P_j$ for $j\ne m$ and form $\mathbb Q:=\bigotimes_{j=1}^M\mathbb Q_j\in\mathcal P$. Therefore,
\[
\inf_{\mathbb P_m\in\mathcal P_m}\mathbb P_m(F_m(\bm x))
=\inf_{\mathbb P\in\mathcal P}\mathbb P(E_m(\bm x))
\ge 1-\epsilon_m.
\]
Thus $\bm x\in\mathcal X_{\mathrm{DRCCP\text{-}I}}$.
\end{proof}

\begin{remark}[On the relationship between DRCCP-I and DRO]
In general, there is no simple inclusion relationship between the feasible regions induced by DRCCP-I and DRO. For each $m\in[M]$, the individual DRCCP constraint controls the worst-case probability of violating the $m$th constraint, but places no restriction on the magnitude of violations when they occur. By contrast, the DRO constraint
\[
    \sup_{\mathbb{P}_m \in \mathcal{P}_m}
    \mathbb{E}_{\mathbb{P}_m}\bigl[f_m(\bm a_m,\bm x)\bigr] \;\le\; b_m
\]
bounds the worst-case expected value, while allowing constraint violations to occur with strictly positive probability. In this case, there can exist decisions that are feasible under DRCCP-I, such as those with extremely rare but arbitrarily large violations, that are infeasible under DRO. On the other hand, there can also exist decisions that satisfy DRO, due to small expected values, but violate the DRCCP-I constraint because the violation probability exceeds $\epsilon_m$. Without additional assumptions that quantitatively link violation probabilities to violation magnitudes, there is no general inclusion relationship between the two models. We therefore view DRCCP-I and DRO as complementary modeling paradigms that represent different approaches to risk control.
\end{remark}

With these inclusion relationships established, we now turn to the solution of DRCCPs. These results clarify the role of DRCCPs in risk treatment but do not provide a solution method. We next present the dual reformulation and primal BiCS perspectives.

\subsubsection{Reformulation Approach: Dual Perspective}

Existing reformulation approaches for DRCCPs depend on the structures of the ambiguity set and satisfaction event. One common approximation route replaces the original chance constraint based on VaR with a CVaR surrogate. This substitution is generally not an exact reformulation. Instead, it replaces the original probability requirement based on a quantile with an average tail risk criterion, thereby enforcing violation control only indirectly and often conservatively.

Motivated by this source of inexactness, and inspired by the enumerated reformulation strategy developed for AS-DRO, we adopt a direct reformulation that explicitly models constraint satisfaction at the event level. In particular, we use the indicator function $z(\bm a,\bm x)$ of the joint satisfaction event, which equals one exactly when all component constraints are satisfied. This indicator representation preserves the intuitive logic of the original VaR model. Under the local conditions stated in Proposition~\ref{prop:drccp_dual}, it yields an equivalent finite-dimensional reformulation. The reformulation for a single DRCCP is presented in Proposition \ref{prop:drccp_dual}, with further discussion of extensions to other DRCCP settings provided in Remark \ref{DRCCP_remark}.

\begin{proposition}\label{prop:drccp_dual}
Let $M=1$ in \eqref{chanceDRO}, write $I:=[I_1]$, $h_1=I_1$, and $\epsilon:=\epsilon_1$, and let $\mathcal A:=\mathcal A_1:=\prod_{i\in I}\mathcal A_i$ carry its product Borel $\sigma$-algebra. For each $i\in I$, let $f_{1i}(\bm a,\bm x)=\bm a_i^\top \bm x$, $b_i:=b_{1i}$, and $\mathcal{A}_i=\{\bm a_i\in\mathbb{R}^n:\bm D_i\bm a_i\preceq_{\mathcal{K}_i}\bm d_i\}$, where $\mathcal{K}_i$ is a proper cone with dual cone $\mathcal{K}_i^*$. Let $T:=T_1$, $\mathcal P:=\mathcal P_1$, $\bm\gamma:=(\gamma_t)_{t\in[T]}$, and $g_t(\bm a)=\sum_{i\in I}\bm E_{it}\bm a_i$ for $t\in[T]$, where $\bm E_{it}\in\mathbb R^{1\times n}$.

For $\bm x\in\mathcal X$, define the satisfaction region $S(\bm x):=\{\bm a=(\bm a_i)_{i\in I}\in\mathcal A: \bm a_i^\top\bm x\le b_i,\ \forall i\in I\}$, and, for each $k\in I$, the strict violation region $V_k(\bm x):=\{\bm a\in\mathcal A:
\bm a_k^\top\bm x>b_k\}.$ Assume that (i) $V_k(\bm x)$ is nonempty for every $\bm x\in\mathcal X$ and $k\in I$; (ii) $\mathcal P$ contains a distribution that strictly satisfies every moment inequality; and (iii) for every $\bm x\in\mathcal X$ and $\bm\beta\in\mathbb R_+^T$, the regional conic minimization problems over $S(\bm x)$ and $\operatorname{cl}_{\mathcal A}V_k(\bm x)$, $k\in I$, have finite optimal values, zero duality gaps, and attained dual optima. Then \eqref{chanceDRO} is equivalent to the following finite-dimensional bilinear reformulation:
\begin{subequations}\label{chance-reformulation}
\begin{align}
\min_{\bm x\in \mathcal{X}} \quad 
& \bm c^\top \bm x, \notag \\[0.3em]
\text{s.t.}\quad 
& \alpha - \bm\beta^\top \bm\gamma \ge 1-\epsilon, \\[0.3em]
& \alpha \le 1 - \sum_{i \in I} \bm p_{0i}^\top \bm d_i - \sum_{i \in I} \mu_{0i} b_i,\\[0.3em]
& - \bm D_i^\top \bm p_{0i} - \mu_{0i}\bm x
   = \sum_{t\in[T]} \bm E_{it}^\top \beta_t,
   \qquad \forall i\in I, \\[0.3em]
& \alpha \le 0 - \sum_{i\in I} \bm p_{ki}^\top \bm d_i + \mu_{1k}b_k,
   \qquad \forall k\in I, \\
& -\bm D_i^\top \bm p_{ki}
  = \sum_{t\in[T]}\bm E_{it}^\top \beta_t,
  \qquad \forall k\in I,\ \forall i\in I\setminus\{k\},\label{eq:dual-violation-k-other} \\
& -\bm D_k^\top \bm p_{kk} + \mu_{1k}\bm x  = \sum_{t\in[T]}\bm E_{kt}^\top \beta_t, \quad \forall k\in I,\\
& \alpha\in\mathbb R,\quad \bm p_{0i}, \bm p_{ki} \in \mathcal K_i^*, \quad \bm\beta\in \mathbb R_+^{T}, \quad
  \mu_{0i}\ge 0,\quad
  \mu_{1k}\ge 0 \quad \forall k\in I, \forall i\in I.
\end{align}
\end{subequations}
\end{proposition}

\begin{proof}
Since $M=1$, the DRCCP constraint can be written as
\[
\inf_{\mathbb P\in\mathcal P}
\mathbb P\{\bm a_i^\top\bm x\le b_i,\ \forall i\in I\}
\ge 1-\epsilon.
\]
For fixed $\bm x$, define the exact satisfaction indicator $s(\bm a,\bm x):=\mathbf 1_{S(\bm x)}(\bm a).$ The chance constraint is therefore
$\inf_{\mathbb P\in\mathcal P}
\mathbb E_{\mathbb P}[s(\bm a,\bm x)]\ge1-\epsilon$.
Since $-s(\cdot,\bm x)$ is bounded, assumption~(ii), together with the
interior argument in Step~1 of the proof of
Theorem~\ref{linear-dro} in Appendix~\ref{app:proof-thm34}, gives zero
duality gap and dual attainment for the generalized moment problem:
\begin{subequations}\label{eq:chance-dual-prob-part}
\begin{align}
\max_{\alpha,\bm\beta}\quad
&\alpha-\bm\beta^\top\bm\gamma,\\
\text{s.t.}\quad
&\alpha-\sum_{t\in[T]}\beta_t
 \sum_{i\in I}\bm E_{it}\bm a_i
 \le s(\bm a,\bm x),
 \qquad \forall\bm a\in\mathcal A,
 \label{Chance_moment_alpha}\\
&\alpha\in\mathbb R,\qquad
 \bm\beta\in\mathbb R_+^T.
\end{align}
\end{subequations}
Consequently, the DRCCP constraint is equivalent to the existence of
$(\alpha,\bm\beta)$ satisfying
$\alpha-\bm\beta^\top\bm\gamma\ge1-\epsilon$ and
\eqref{Chance_moment_alpha}.

For fixed $\bm x$ and $\bm\beta$, write
$L_{\bm\beta}(\bm a):=
\sum_{t\in[T]}\beta_t\sum_{i\in I}\bm E_{it}\bm a_i$.
Since $s(\bm a,\bm x)=1$ on $S(\bm x)$ and
$s(\bm a,\bm x)=0$ on
$\mathcal A\setminus S(\bm x)=\bigcup_{k\in I}V_k(\bm x)$,
the semi-infinite constraint separates into one satisfaction-region
bound and one violation-region bound for each $k\in I$. Moreover, assumption~(i) and convexity of $\mathcal A$ imply $\operatorname{cl}_{\mathcal A}V_k(\bm x) = \{\bm a\in\mathcal A:\bm a_k^\top\bm x\ge b_k\}.$ Indeed, every boundary point with
$\bm a_k^\top\bm x=b_k$ can be approached by convex combinations with
a strict violator in $V_k(\bm x)$. Since $L_{\bm\beta}$ is continuous,
the infimum over $V_k(\bm x)$ equals that over its relative closure.
Hence, \eqref{Chance_moment_alpha} is equivalent to
\begin{align}
\alpha
&\le
1+\inf\!\left\{
L_{\bm\beta}(\bm a):
\bm a\in\mathcal A,\ 
\bm a_i^\top\bm x\le b_i,\ \forall i\in I
\right\},
\tag{SAT}\label{eq:regional-sat}\\
\alpha
&\le
\inf\!\left\{
L_{\bm\beta}(\bm a):
\bm a\in\mathcal A,\ 
\bm a_k^\top\bm x\ge b_k
\right\},
\qquad \forall k\in I.
\tag{VIO$_k$}\label{eq:regional-vio}
\end{align}

The key distinction is that \textup{(SAT)} contains all $I$
satisfaction inequalities and therefore introduces one multiplier
$\mu_{0i}$ for every $i\in I$. In contrast, each
\textup{(VIO$_k$)} contains only the reversed $k$th inequality and
introduces the single multiplier $\mu_{1k}$, which appears only in the
$k$th stationarity condition. Assumption~(iii) permits exact conic
dualization of these regional problems: \textup{(SAT)} yields the
$(\bm p_{0i},\mu_{0i})$ block, while each \textup{(VIO$_k$)} yields
the $(\bm p_{ki},\mu_{1k})$ block in
\eqref{chance-reformulation}. The vector $\bm\beta$ is shared across
all regional subproblems, whereas their conic certificates are
region-specific.

Combining these regional duals with
$\alpha-\bm\beta^\top\bm\gamma\ge1-\epsilon$ and the corresponding
dual-variable domains gives exactly \eqref{chance-reformulation}.
Since the argument holds for every $\bm x\in\mathcal X$, the two
optimization problems are equivalent.
\end{proof}

\begin{remark}\label{DRCCP_remark}
The reformulation yields several immediate modeling insights.
\begin{itemize}
    \item \textbf{Individual, joint, and combinatorial DRCCPs.}
    When $I_m=1$, the formulation reduces to an individual DRCCP and involves two dual subproblems. When $I_m>1$, the joint DRCCP gives rise to $I_m+1$ dual subproblems: one corresponding to the case in which all constraints in the group are satisfied ($z=1$), and $I_m$ additional subproblems, each associated with the violation of a single constraint within the group ($z=0$). For a general $h_m$-out-of-$I_m$ combinatorial DRCCP, direct eventwise dualization may require $\binom{I_m}{h_m}+\binom{I_m}{h_m-1}=\binom{I_m+1}{h_m}$ regional subproblems; this number has exponential worst-case growth in $I_m$, making the explicit dual reformulation computationally impractical as a general-purpose solution approach.
    
    \item \textbf{Single and multiple DRCCPs.}
    The reformulation can be applied groupwise to multiple DRCCPs ($M\ge 2$) when the proposition's local conditions hold for each group. Each group contributes its own dual constraints of the form~\eqref{eq:chance-dual-prob-part}, while the groups share the decision vector $\bm x$.
    
    \item \textbf{AS-DRO as a limiting case.}
    At $\epsilon=0$, the probability requirement in \eqref{chance-reformulation} becomes the corresponding AS-DRO requirement under the conditions of Proposition~\ref{prop:drccp_dual}. Moreover, under Definition~\ref{def:Marginal-ambiguity}, the joint and individual versions of AS-DRO are equivalent.
\end{itemize}
\end{remark}

\subsubsection{BiCS Approach: Primal Perspective}
The dual reformulation provides a bilinear representation of DRCCPs under the local conditions of Proposition~\ref{prop:drccp_dual}. These conditions can be restrictive in settings with moments of higher order or combinatorial structures. At $\epsilon=0$, the DRCCP probability requirement becomes an almost-sure requirement for the same group satisfaction event. These observations motivate a primal treatment that uses the exact satisfaction event directly. For each $m\in[M]$, define
\[
s_m(\bm a_m,\bm x):=
\mathbf 1_{\left\{
|\{i\in[I_m]:f_{mi}(\bm a_m,\bm x)\le b_{mi}\}|\ge h_m
\right\}}.
\]

For a finite $\hat A_m=\{\bm a_{mk}\}_{k=0}^{K_m}\subseteq\mathcal A_m$, let
$\mathcal P_m(\hat A_m):=\{\mathbb P_m\in\mathcal P_m:
\operatorname{supp}(\mathbb P_m)\subseteq\hat A_m\}$, write
$P_{mk}:=\mathbb P_m(\{\bm a_{mk}\})$, and adopt
$\inf\varnothing=+\infty$, so an empty restricted row is temporarily
vacuous. We adapt the master and subproblem structure of
Algorithm~\ref{BiCSalgorithm} as follows.
\begin{subequations}\label{DRCCP-master problem}
\begin{align}
\textbf{MP-DRCCP:}\quad
\min\quad
& \bm c^\top\bm x\\
\text{s.t.}\quad
& \inf_{\mathbb P_m\in\mathcal P_m(\hat A_m)}
\sum_{k=0}^{K_m}P_{mk}\hat s_{mk}
\ge 1-\epsilon_m,
\qquad \forall m\in[M],
\label{inner_min_constraint}\\
& \hat s_{mk}^i=1
\ \Longrightarrow\
f_{mi}(\bm a_{mk},\bm x)\le b_{mi},
\qquad
\forall i\in[I_m],\
\forall k=0,\ldots,K_m,\
\forall m\in[M],\\
& (\hat s_{mk}-1)I_m
\le
\sum_{i=1}^{I_m}\hat s_{mk}^i-h_m,
\qquad
\forall k=0,\ldots,K_m,\
\forall m\in[M],
\label{Chance-satisfy-indicator}\\
& \bm x\in\mathcal X,\qquad
\hat s_{mk}^i,\hat s_{mk}\in\{0,1\},
\qquad
\forall i\in[I_m],\
\forall k=0,\ldots,K_m,\
\forall m\in[M].
\end{align}
\end{subequations}
\begin{subequations}\label{DRCCP-SubProblem}
\begin{align}
\textbf{SP-DRCCP:}\quad
\underline{\Omega}_m^*(\bm x^*)
&=
\inf_{\tilde{\mathbb P}_m\in\mathcal P_m}
\mathbb E_{\tilde{\mathbb P}_m}
\bigl[s_m(\bm a_m,\bm x^*)\bigr],
\qquad m\in[M].
\end{align}
\end{subequations}

In MP-DRCCP, $\hat s_{mk}^i$ and $\hat s_{mk}$ are auxiliary binary variables representing componentwise and groupwise satisfaction at scenario $\bm a_{mk}$, respectively. The implication and cardinality constraints imply
$\hat s_{mk}\le s_m(\bm a_{mk},\bm x)$. Conversely, choosing
$\hat s_{mk}^i=\mathbf 1_{\{f_{mi}(\bm a_{mk},\bm x)\le b_{mi}\}}$
and $\hat s_{mk}=s_m(\bm a_{mk},\bm x)$ satisfies these constraints.
Thus, the mixed-integer formulation and the restricted chance constraint have the same feasible decisions $\bm x$. Since
$\mathcal P_m(\hat A_m)\subseteq\mathcal P_m$, MP-DRCCP is a relaxation of
\eqref{chanceDRO}.

\paragraph{Solving Master Problem}
For a nonempty restricted ambiguity set, the inner minimization in \eqref{inner_min_constraint} is a finite linear program in the probability weights. It may therefore be handled by the LP duality or KKT constructions of Section~\ref{BiCS Master Problem}, with the objective direction adjusted to minimization. This observation concerns only the problem over probability weights and does not remove the difficulty of encoding events in the sample space.

\paragraph{Solving Subproblem}
Given a candidate solution $\bm x^*$ from the master problem, the subproblem computes the worst-case satisfaction probability in the full ambiguity set. Once $\bm x^*$ is fixed, $s_m(\bm a_m,\bm x^*)$ is determined by the scenario. A distribution with finite support that proves insufficient satisfaction supplies a violated distribution cut; otherwise, feasibility is asserted only through one of the valid certificates stated below. We use the same two oracle architectures as in Standard DRO.
\emph{Oracle 1} provides a direct formulation that jointly optimizes over scenarios and their associated probability weights, while \emph{Oracle 2} exploits the structure of the ambiguity set through CG to improve scalability.

The satisfaction indicator $s_m$ is bounded. Applying the proof of Corollary~\ref{cor:SP-O1-oracle} to $-s_m$ shows that the infimum over $\mathcal P_m$ equals the infimum over admissible distributions supported on at most $T_m+1$ points. Hence Oracle~1 takes the form
\begin{subequations}\label{SP-bilinear-DRCCP}
\begin{align}
\textbf{SP-O1-DRCCP:}\quad
\underline{\Omega}_m^*(\bm{x}^*) = \inf\;& \sum_{r=1}^{T_m+1} \tilde{P}_{mr}\,
s_m(\tilde{\bm a}_{mr},\bm x^*) \label{DRCCP-O1-SP}\\
\text{s.t.}\;& \sum_{r=1}^{T_m+1}\tilde P_{mr}=1, \hspace{2mm} \sum_{r=1}^{T_m+1}\tilde P_{mr}g_t(\tilde{\bm a}_{mr})\le\gamma_{mt},
\quad \forall t\in[T_m],\\
& \tilde{\bm a}_{mr}\in\mathcal A_m,\quad \tilde P_{mr}\ge0,
\quad \forall r\in[T_m+1].
\end{align}
\end{subequations}

Oracle~2 instead separates probability assignment from scenario generation, handled by the following PMP and PSP, respectively.

\begin{subequations}\label{Chance-PricingMasterProblem}
\begin{align}
\textbf{PMP-DRCCP:}\quad \omega_m^*(\bm x^*) =
\min_{P_{mk}\ge0}\;& \sum_{k=0}^{K_m} P_{mk}\,s_{mk}^*\\
\text{s.t.}\;& \eqref{PMP-constraints},\notag
\end{align}
\end{subequations}

Here $s_{mk}^*:=s_m(\bm a_{mk},\bm x^*)$. If the restricted PMP is infeasible, the Farkas branch of Algorithm~\ref{SPCG} applies unchanged because its ray certificate depends only on the ambiguity constraints, not on the objective coefficients. 

Let $(\alpha_m^*,\bm\beta_m^*)$ be an optimal dual solution of the feasible PMP. Its objective is $\alpha_m-\bm\beta_m^\top\bm\gamma_m$, and its constraints are $\alpha_m-\sum_{t\in[T_m]}\beta_{mt}g_t(\bm a_{mk})\le s_{mk}^*$.

\begin{subequations}\label{Chance-PricingSubProblem}
\begin{align}
\textbf{PSP-DRCCP:}\quad
\pi_m^*(\bm x^*,\alpha_m^*,\bm\beta_m^*) =
\sup_{\tilde{\bm a}_m\in\mathcal A_m}\;&
\alpha_m^* - \sum_{t\in[T_m]} \beta_{mt}^*g_t(\tilde{\bm a}_m)
-s_m(\tilde{\bm a}_m,\bm x^*).
\end{align}
\end{subequations}

\begin{remark}
    In an explicit finite encoding of the pricing subproblem, when $I_m=h_m$, the violation event may be represented by enumerating the violated component. When $I_m>h_m$, enumeration of violation patterns or a binary cardinality representation may be used, but exactness of a buffered formulation requires the conditions stated above for that encoding.
\end{remark}

Finite LP duality for the restricted PMP, together with weak duality for the generalized moment problem, gives
\[
\omega_m^*(\bm x^*)-\pi_m^*
\le \underline{\Omega}_m^*(\bm x^*)\le \omega_m^*(\bm x^*).
\]
For an outer algorithmic tolerance $\eta>0$, kept distinct from the chance tolerance $\epsilon_m$, the DRCCP oracle applies the Standard DRO return conditions to the violation indicator $1-s_m$. Hence, if the outer method terminates with the corresponding certificate for every row, then $\underline{\Omega}_m^*(\bm x^*)\ge 1-\epsilon_m-\eta$ for every $m$, while the restricted master supplies a lower bound on the original objective. At $\eta=0$, this conclusion is exact conditional on exact certificates and termination.


\subsection{Local Information} \label{sec:local_info}

Standard DRO models typically impose \emph{global} distributional information, such as bounds on the mean vector and covariance matrix. While these constraints shape the overall distribution, they do not specify how probability mass is distributed across different regions of the sample space. In many practical settings, however, additional distributional information associated with prescribed regions of the sample space can be collected at relatively low cost; we refer to this as local information. Examples include histogram counts over predefined bins, empirical estimates of tail probabilities, or the fraction of observations lying in a typical range around the mean. Incorporating such local information can yield more realistic ambiguity sets. 

Formally, let $L\in\mathbb N$. Starting from the global ambiguity set $\mathcal{P}$ in \eqref{p_gm_ambiguity}, we define a \emph{local information ambiguity set} by adding constraints on measurable regions $A_l\in\mathcal F$, $l\in[L]$:
\begin{align} 
\mathcal{P}^{\text{local}}
&= 
\mathcal{P} \cap 
\Bigl\{
\mathbb P \;\big|\;
\int_{A_l} \psi_l(\bm a)\, \mathrm{d}\mathbb P(\bm a) \le \phi_l,\;\forall l\in[L]
\Bigr\} \label{Local-constraint}\\
&=
\mathcal{P} \cap
\Bigl\{
\mathbb P \;\big|\;
\int_{\mathcal A} \mathbf{1}_{\{\bm a\in A_l\}} \psi_l(\bm a)\, \mathrm{d}\mathbb P(\bm a) \le \phi_l,\;\forall l\in[L]
\Bigr\}, \notag
\end{align}

where $\psi_l:\mathcal A\to\mathbb R$ is measurable and $\phi_l\in\mathbb R$. Here, $\mathbf{1}_{\{\bm a\in A_l\}}$ is the membership indicator of the region $A_l$, equal to one when $\bm a\in A_l$ and zero otherwise. Morevoer, we assume locally that $\mathbf 1_{\{\bm a\in A_l\}}\psi_l(\bm a)$ is integrable under every $\mathbb P\in\mathcal P$ and that $\mathcal P^{\mathrm{local}}$ is nonempty. The functions $\psi_l$ may encode local moments, tail indicators, or other statistics specific to a region. In particular, a lower probability bound $\mathbb P(A_l)\ge\alpha_l$ fits this common $\le$ form by taking $\psi_l\equiv-1$ and $\phi_l=-\alpha_l$.

The same local constraints can be imposed on any baseline ambiguity set on $(\mathcal A,\mathcal F)$, provided that the relevant integrals are well defined and the resulting set is nonempty. The generalized moment specification is used below for the oracle reformulations and numerical illustration.

Local information constraints substantially increase the expressive power of the ambiguity set. Moment-based DRO restricts only a few summary statistics and leaves the allocation of probability mass across the sample space largely unconstrained. As a result, distributions with bell shapes, heavy tails, or even distributions concentrated on only two points may satisfy the same moment conditions. Local information mitigates this limitation by imposing probability bounds on specific subregions of the sample space. For instance, requiring a prescribed proportion of mass to lie within one standard deviation of the mean can encode a concentration pattern resembling a normal distribution without assuming a parametric model. More generally, asymmetric regional constraints can encode skewed patterns or patterns resembling a Gamma distribution. In this way, these local information constraints allow the ambiguity set to encode concentration and tail behavior while maintaining distributional robustness.

\begin{example}[Local Information Resembling a Normal Distribution] \label{ex.local}
Let $\bm a=(a_1,a_2)$ be defined on a measurable space $(\mathcal A,\mathcal F)$, where $\mathcal A\subseteq\mathbb R^2$. Consider the following Standard DRO problem:
\[
\begin{aligned}
\min_{\bm x \in \mathbb{R}_+^2} \quad & c_1 x_1 + c_2 x_2 \\
\text{s.t.} \quad 
& \mathbb{E}_{\mathbb{P}}[a_1 x_1 + a_2 x_2] \le b 
\qquad \forall\, \mathbb{P} \in \mathcal{P}^{\text{local-Normal}},
\end{aligned}
\]
where $\mathcal{P}^{\text{local-Normal}}$ augments global moment information with regional probability constraints to capture an ambiguity structure resembling a normal distribution. Specifically, let $\bm\mu:=(\mu_1,\mu_2)$ denote the nominal mean:
\[
\mathcal{P}^{\mathrm{local-Normal}}
=
\left\{
\mathbb{P}\in\mathcal M_+(\mathcal A,\mathcal F) :
\begin{array}{l}
\mathbb P(\mathcal A)=1,\\[1.5mm]
\mathbb{E}_{\mathbb{P}}\!\left[a_j\right] \in [\,\underline{\mu}_j,\overline{\mu}_j\,],
\quad j = 1,2,\\[1.5mm]
\mathbb{E}_{\mathbb{P}}\!\left[(a_j - \mu_j)^2\right] \le \bar{\sigma}_j^2,
\quad j = 1,2,\\[1.5mm]
\mathbb{E}_{\mathbb{P}}\!\left[(a_1 - \mu_1)(a_2 - \mu_2)\right]
\le \rho_{12}\, \bar{\sigma}_1\,\bar{\sigma}_2,\\[1.5mm]
\mathbb{P}(\bm a \in B_1) \ge \alpha_1,\quad
\mathbb{P}(\bm a \in B_2 \setminus B_1) \ge \alpha_2 - \alpha_1.
\end{array}
\right\}.
\]

The sets $B_1$ and $B_2$ represent concentration regions within one and two standard deviations of the nominal mean $\bm\mu$, where $0\le\alpha_1\le\alpha_2\le1$ and $\bar\sigma_j>0$ for $j=1,2$. 
In this example, we use coordinatewise boxes to approximate these regions:
\[
B_k
=
\left\{
\bm a\in\mathcal A :
|a_j-\mu_j| \le k \bar\sigma_j, \quad j=1,2
\right\}, \quad k=1,2.
\]
\end{example}

While local information enhances the modeling capability of the ambiguity set, it also complicates the resulting optimization problem. The following subsections present solution schemes for DRO models with such local constraints.

\subsubsection{Reformulation Approach: Dual Perspective}

For each $m\in[M]$, let $\mathcal P_m^{\mathrm{local}}$ be the rowwise counterpart of \eqref{Local-constraint}, with data $(A_{ml},\psi_{ml},\phi_{ml})$, $l\in[L_m]$. If the regularity conditions used in Theorem~\ref{linear-dro} hold and $\mathcal P_m^{\mathrm{local}}$ contains a distribution under which all global and local inequalities hold strictly, generalized moment strong duality applies to the first dualization. This strict feasibility condition is local to the dual reformulation; it is not required by the primal model below. The dual objective for row $m$ is $\alpha_m+\sum_{t\in[T_m]}\beta_{mt}\gamma_{mt}+\sum_{l\in[L_m]}\tau_{ml}\phi_{ml}$, and the $m$th DRO row requires this quantity to be at most $b_m$. For fixed $\bm x$, $\bm\beta_m\in\mathbb R_+^{T_m}$, and $\bm\tau_m\in\mathbb R_+^{L_m}$, when the following supremum is finite, the smallest feasible $\alpha_m$ is

\begin{equation}\label{local_alpha_max_expression}
\alpha_m
=
\sup_{\bm a_m\in \mathcal{A}_m}
\Bigl\{
f_m(\bm a_m,\bm x)
-\sum_{t\in[T_m]}\beta_{mt} g_t(\bm a_m)
-\sum_{l\in[L_m]}\tau_{ml}\,\mathbf{1}_{\{\bm a_m\in A_{ml}\}} \psi_{ml}(\bm a_m)
\Bigr\}.
\end{equation}

Here $\tau_{ml}\ge0$ are the dual variables associated with the local information constraints. The main difficulty lies in the problem over the sample space in \eqref{local_alpha_max_expression}, whose indicator terms are discontinuous and whose remaining nonlinear terms may also make the formulation nonconvex.

\begin{remark}
Several challenges arise when attempting to reformulate the inner maximization problem \eqref{local_alpha_max_expression} from the dual perspective:

\begin{enumerate}
\item Even without the indicator terms, indefinite cross moments or moments of higher order may make the problem over the sample space nonconvex. For instance, the covariance expression in Example~\ref{ex.local} is generally indefinite.

\item In principle, one could enumerate all patterns of regional membership for each realization $\bm a_m$. However, similar to combinatorial 
DRCCP models, the number of such patterns grows exponentially in $L_m$. Consequently, full dualization over the sample space quickly becomes impractical 
when many local information constraints are included.

\item For coordinatewise boxes, membership is described by linear inequalities but the complement is disjunctive. An exact finite representation therefore requires additional enumeration or binary variables and valid finite bounds when a Big-$M$ encoding is used.

\end{enumerate}
\end{remark}

Overall, these features can make a tractable dual reformulation over the sample space difficult; tractability depends on the regional and moment structure.

\subsubsection{BiCS Approach: Primal Perspective}

The primal BiCS framework avoids dualizing over the entire sample space by evaluating regional membership for generated scenarios. For a finite set $\hat A_m=\{\bm a_{mk}:k=0,\ldots,K_m\}$, define $\mathcal P_m^{\mathrm{local}}(\hat A_m):=\{\mathbb P_m\in\mathcal P_m^{\mathrm{local}}:\operatorname{supp}(\mathbb P_m)\subseteq\hat A_m\}$. For this variant with local information, we adopt $\sup\varnothing=-\infty$, so an empty restricted master row is temporarily vacuous; the initial scenario pool need not support a feasible local distribution.

\begin{subequations}\label{Local-DRO-master problem}
\begin{align}
\textbf{MP-Local-DRO:}\quad
\min\;& \bm{c}^\top \bm{x}\\
\text{s.t.}\;& \sup_{\mathbb P_m \in \mathcal{P}_m^{\mathrm{local}}(\hat A_m)} \sum_{k = 0}^{K_m} P_{mk}f_m(\bm{a}_{mk},\bm{x}) \le b_m, \quad \forall m\in[M]\\
& \bm{x} \in \mathcal{X}
\end{align}
\end{subequations}
\begin{subequations}\label{Local-DRO--SubProblem}
\begin{align}
\textbf{SP-Local-DRO:}\quad
\Omega_m^*(\bm x^*) = \sup_{\tilde{\mathbb{P}}_m\in\mathcal{P}^{\text{local}}_m}\;&
\mathbb{E}_{\tilde{\mathbb{P}}_m}[f_m(\bm a_m,\bm x^*)],
\quad m\in[M].
\end{align}
\end{subequations}

For a nonempty restricted local ambiguity set, the master row remains a finite linear program in the probability weights after the corresponding local constraints are added. Thus the LP duality or KKT constructions of Section~\ref{sec:DRO} remain available under their stated conditions, and $\mathcal P_m^{\mathrm{local}}(\hat A_m)\subseteq\mathcal P_m^{\mathrm{local}}$ preserves the relaxation property of the master problem.

The main modification occurs in the oracle subproblem \eqref{Local-DRO--SubProblem}, where the full local ambiguity set and exact regional memberships must be enforced. For each $m\in[M]$ and $l\in[L_m]$, define $q_{ml}(\bm a_m):=\mathbf 1_{\{\bm a_m\in A_{ml}\}}\psi_{ml}(\bm a_m)$.

\begin{corollary}\label{SP-bilinear-Local-DRO-cor}
Fix $m\in[M]$ and $\bm x^*\in\mathcal X$, and let $N_m:=T_m+L_m+1$. Under the local nonemptiness and integrability conditions above and finiteness of $\Omega_m^*(\bm x^*)$, applying the proof of Corollary~\ref{cor:SP-O1-oracle} to the augmented list $\{g_t\}_{t\in[T_m]}\cup\{q_{ml}\}_{l\in[L_m]}$ yields equality between the full supremum and the supremum over distributions with at most $N_m$ support points, giving
\begin{subequations}\label{SP-bilinear-Local-DRO}
\begin{align}
\textbf{SP-O1-Local-DRO:} \hspace{2mm}
\Omega_m^*(\bm{x}^*)
    = \sup_{\{\tilde{P}_i,\tilde{\bm a}_i,\tilde{z}_{il}\}}
    &\;\sum_{i=1}^{N_m} \tilde{P}_i\, f_m(\tilde{\bm a}_{i},\bm{x}^*) 
    \label{SP-bilinear-obj}\\[2mm]
\text{s.t.}\quad
& \sum_{i=1}^{N_m}\tilde P_i=1,\\[1mm]
& \sum_{i=1}^{N_m}\tilde P_i g_t(\tilde{\bm a}_i)\le\gamma_{mt},
    \qquad \forall t\in[T_m],\\[1mm]
& \sum_{i=1}^{N_m} \tilde{P}_i\, \tilde{z}_{il}\, \psi_{ml}(\tilde{\bm a}_i)
    \le \phi_{ml},
    \qquad \forall l\in[L_m], \label{SP-local-info}\\[1mm]
& \tilde z_{il}=\mathbf 1_{\{\tilde{\bm a}_i\in A_{ml}\}},
    \qquad \forall i\in[N_m],\;\forall l\in[L_m],
    \label{SP-indicator-1}\\[1mm]
& \tilde{\bm a}_i \in \mathcal{A}_{m},\quad
  \tilde{P}_i \ge 0,
    \qquad \forall i\in[N_m], \label{SP-feasible-a}\\[1mm]
& \tilde{z}_{il} \in \{0,1\},
    \qquad \forall i\in[N_m],\;\forall l\in[L_m]. \label{SP-binary}
\end{align}
\end{subequations}
\end{corollary}

If the supremum in \eqref{Local-DRO--SubProblem} is attained, the supremum above may be replaced by a maximum. The exact membership equalities in~\eqref{SP-indicator-1} and~\eqref{Local-PricingSubProblem} can be replaced by the usual Big-$M$ pair with nonstrict inequalities only under boundary separation specific to the model; otherwise, a different exact encoding is required. The products in \eqref{SP-local-info} can make the oracle a mixed-integer nonlinear program when $\psi_{ml}$ or the global moment functions are nonlinear.

Compared with the preceding Oracle~1 formulations, local information introduces additional coupling between scenario generation and probability assignment through \eqref{SP-local-info}. Oracle~2 retains the CG separation: probability allocation is handled in the PMP, while new scenarios and their exact regional memberships are determined in the PSP.
\begin{subequations}\label{Local-PricingMasterProblem}
\begin{align}
\textbf{PMP-Local-DRO:}\quad 
\omega_m^*(\bm x^*) =
\max\;& \sum_{k=0}^{K_m} P_{mk}\,f_m(\bm a_{mk},\bm{x}^*)\\
\text{s.t.}\;& 
\sum_{k=0}^{K_m} P_{mk} = 1, \label{local-alpha-dual}\\
& \sum_{k=0}^{K_m} P_{mk}\,g_t(\bm{a}_{mk}) \le \gamma_{mt}
\quad \forall t \in [T_m], \label{local-beta-dual}\\
& \sum_{k=0}^{K_m} P_{mk}\, z_{mkl}\, \psi_{ml}(\bm a_{mk}) \le \phi_{ml}
\quad \forall l \in [L_m], \\
& P_{mk} \ge 0
\quad k=0,\ldots,K_m.
\end{align}
\end{subequations}

Here $z_{mkl}:=\mathbf 1_{\{\bm a_{mk}\in A_{ml}\}}$. If PMP-Local-DRO is infeasible, Proposition~\ref{prop:PMP-infeasible} applies to the augmented moment list $\{g_t\}_{t\in[T_m]}\cup\{q_{ml}\}_{l\in[L_m]}$: since $\mathcal P_m^{\mathrm{local}}$ is nonempty, the associated Farkas pricing problem has a positive value, and adding a scenario with a strictly positive ray objective invalidates the current certificate. If PMP-Local-DRO is feasible, let $(\alpha_m^*,\bm\beta_m^*,\bm\tau_m^*)$ be an optimal dual solution. Its objective is $\alpha_m+\sum_{t\in[T_m]}\beta_{mt}\gamma_{mt}+\sum_{l\in[L_m]}\tau_{ml}\phi_{ml}$, and its corresponding constraints are
\[
\alpha_m+\sum_{t\in[T_m]}\beta_{mt}g_t(\bm a_{mk})
+\sum_{l\in[L_m]}\tau_{ml}z_{mkl}\psi_{ml}(\bm a_{mk})
\ge f_m(\bm a_{mk},\bm x^*),
\quad k=0,\ldots,K_m.
\]
\begin{subequations}\label{Local-PricingSubProblem}
\begin{align}
\textbf{PSP-Local-DRO:}\notag\\
\pi_m^*(\bm x^*,\alpha_m^*,\bm\beta_m^*, \bm\tau_m^*) =
\sup_{\tilde{\bm a}_m,\tilde{z}_l}\;&
f_m(\tilde{\bm a}_{m},\bm{x}^*) 
- \alpha_m^* 
- \sum_{t\in[T_m]} \beta_{mt}^*\,g_t(\tilde{\bm{a}}_m)
- \sum_{l\in[L_m]} \tau^*_{ml}\, \tilde{z}_{l}\, \psi_{ml}(\tilde{\bm a}_{m})\\[1mm]
\text{s.t.}\;&
\tilde z_l=\mathbf 1_{\{\tilde{\bm a}_m\in A_{ml}\}},
\qquad \forall l\in[L_m],\\
& \tilde{\bm a}_{m} \in \mathcal{A}_m,\quad \tilde{z}_l \in \{0,1\}
\quad \forall l\in[L_m].
\end{align}
\end{subequations}

For a feasible PMP, weak and finite LP duality yield $\omega_m^*(\bm x^*)\le\Omega_m^*(\bm x^*)
\le\omega_m^*(\bm x^*)+\pi_m^*.$ Thus a tolerance on reduced cost supplies the same valid Oracle~2 value certificate as in Standard DRO. 

We use Example~\ref{ex.local} to examine how local information changes the worst-case distribution. The four rows in Figure~\ref{fig:Local_distribution} form a $2\times2$ design: global information is described by first- or second-order moments, with or without local constraints. The displayed values are the objectives of the outer minimization. Restricting the ambiguity set weakly expands the feasible region for $\bm x$ and therefore weakly lowers the minimum objective value.

\begin{figure}[h!]
    \centering
    \includegraphics[width=0.95\textwidth]{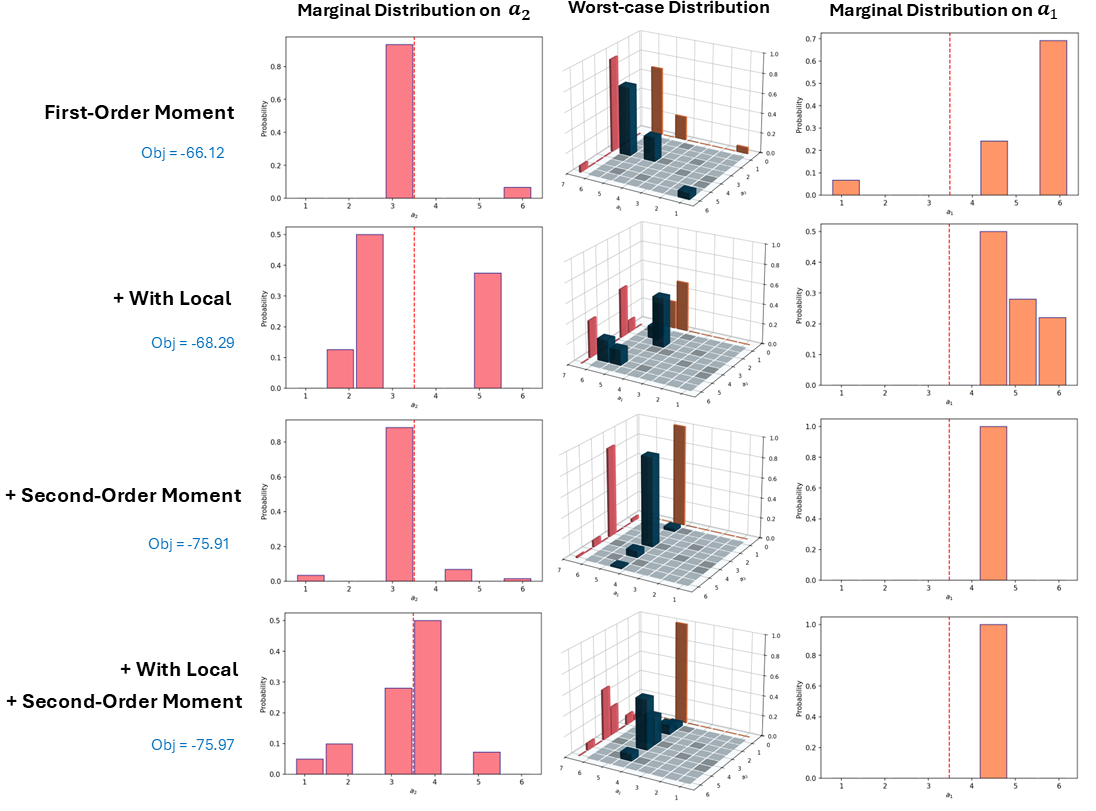}
    \caption{Outer objective values and worst-case distributions with and without local information in Example~\ref{ex.local}.}
    \label{fig:Local_distribution}
\end{figure}

The first two rows isolate the effect of local information under first-order moments. Without local constraints, the worst-case distribution concentrates much of its mass at a small number of support points. Adding local constraints redistributes probability across several support points in both marginals. The outer objective decreases from $-66.12$ to $-68.29$, consistent with the weak expansion of the $\bm x$-feasible region.

The second-order rows show a more selective effect. Second-order information already lowers the outer objective to $-75.91$ and substantially determines the first marginal. Adding local information changes the objective only slightly, to $-75.97$, while still redistributing probability in the second marginal across the remaining admissible locations. Thus, in this example, local information primarily refines distributional features not already determined by the second-order moments. The local constraints therefore exclude concentrated worst-case distributions that satisfy the global moment conditions but are inconsistent with the available local knowledge, while BiCS makes the resulting redistribution explicit through its generated support points and probability weights.

\section{Numerical Experiments} 
\label{sec:results} 

We consider a knapsack problem with $n$ items. Let $J:=[n]$ denote the item index set; throughout the numerical study, $N$ is reserved for the number of empirical sample points. For a set $J_I\subseteq J$ of integer components, let
\[
  \mathcal{X}_{0} := \{\bm x\in[0,1]^n\mid x_j\in\{0,1\}\;\forall j\in J_I\},
  \qquad
  \mathcal{X}_{\mathrm{det}} := \{\bm x\in\mathbb{R}^{n}\mid \bm A\bm x\le \bm b\},
\]
where $\bm A\in\mathbb{R}^{m\times n}$ and $\bm b\in\mathbb{R}^{m}_{+}$. For $\bm c\in\mathbb{R}^{n}_{+}$, the deterministic knapsack is defined as 
\begin{align}
  \max_{\bm x\in\mathbb{R}^{n}}\;\bm c^{\top}\bm x
  \quad\text{s.t.}\quad
  \bm x\in\mathcal{X}_{0}\cap\mathcal{X}_{\mathrm{det}},
  \tag{DET}\label{prob:DET}
\end{align}
Here $J_I=\emptyset$ gives the continuous model, $J_I=J$ gives the pure-integer model, and a proper nonempty subset gives a mixed-integer model.

In this section, we mainly use the following baseline instance as a simple illustrative example. Each item profit is drawn as \(c_j \sim \mathrm{Unif}[0,10]\), which can represent, for instance, expected revenue in thousands of dollars. Each coefficient \(a_{ij}\) of \(\bm A\), which may capture the resource consumption or salvage effect of including item \(j\) in constraint \(i\), is independently set to \(\mathrm{Unif}[0,100]\) (e.g., weight in kilograms or processing time in minutes) with probability \(d_p\), to \(\mathrm{Unif}[-50,0]\) (e.g., a resource refund or reclaimable capacity) with probability \(d_n\), and to zero otherwise, resulting in a \((d_p, d_n)\)-sparse matrix that incorporates both positive and negative impacts. Each capacity limit \(b_i\) is drawn as \(b_i \sim \mathrm{Unif}[0, b_{\max}]\), where \(b_{\max} = 50\,n\,d_p\), representing a total available resource such as budget, weight limit, or machine hours. This simple yet flexible model mirrors the setting in \cite{bertsimas2004price}, and could help us illustrate our framework clearly.

Furthermore, to make this paper self-contained, we provide definitions of the following commonly used ambiguity sets, which are special cases of \(\mathcal{P}\) in~\eqref{p_gm_ambiguity} and will be used in this section.

\begin{definition} The moment-inequality ambiguity set \(\mathcal{P}^{M}\) is defined as:
\begin{align}
    \mathcal{P}^{M} := \left\{ \mathbb{P} \in \mathcal{M}_+(\mathcal{A}, \mathcal{F}) \,\middle|\, 
    \mathbb{E}_{\mathbb P} \left[ 1 \right] = 1, \mathbb{E}_{\mathbb P} \left[ g_t(\bm a) \right] \in [l_t, u_t], \forall t \in T \right\} \label{p_m_ambiguity},
\end{align}
where \(g_t( \bm a)\) denotes a moment function corresponding to each index \(t\) in the index set \(T\), and \(l_t\) and \(u_t\) are the corresponding lower and upper bounds on these moments under candidate probability distributions.
\end{definition}

\begin{definition}\label{p_w_ambiguity} The Wasserstein ambiguity set \(\mathcal{P}^{W}\) is a closed Wasserstein ball with radius $\theta > 0$: 
\begin{subequations}
\begin{align}
    \mathcal{P}^{W} := \left\{ \mathbb{P} \in \mathcal{M}_+(\mathcal{A}, \mathcal{F}) \,\middle|\, 
    \mathbb{E}_{\mathbb P} \left[ 1 \right] = 1, \mathcal{W} (\mathbb{P}, \mathbb{P}_0) \leq \theta \right\}
\end{align}
where the Wasserstein distance \(\mathcal{W} (\mathbb{P}, \mathbb{P}_0)\) between two measures \(\mathbb{P}\) and \(\mathbb{P}_0\) is defined as: 
\begin{equation}
    \mathcal{W} (\mathbb{P}, \mathbb{P}_0) :=  \inf \left\{ \int_{\mathcal{A} \times \mathcal{A}} \| \bm a - \bm a_0\| \, d\Pi(\bm a, \bm a_0) : 
    \begin{aligned}
    \Pi \text{ is a joint distribution of } \bm a \text{ and } \bm a_0  \\
    \text{ with marginals } \mathbb{P} \text{ and } \mathbb{P}_0 \text{, respectively} 
    \end{aligned} 
    \right\}.
\end{equation}
\end{subequations}
More specifically, \( \| \cdot \|\) denotes an arbitrary norm on \(\mathbb{R}^n\). The empirical distribution is constructed from $N$ i.i.d. samples as \(\mathbb{P}_0 := \frac{1}{N} \sum_{j=1}^N \delta_{\bar{\bm a}^{\,0}_j}\), where $\delta_{\bar{\bm a}^{\,0}_j}$ is the Dirac measure concentrated at sample point $\bar{\bm a}^{\,0}_j$.
\end{definition}

All experiments are conducted on a machine with an AMD Ryzen 7 5825U processor, featuring 8 physical cores and 16 logical threads. We use Python 3.11.11 (Anaconda distribution) and Gurobi 11.0.0 as the optimization solver. The time limit for each run is set to 3600 seconds. Convergence is determined by the stated algorithmic stopping criterion.


\subsection{Standard DRO} \label{Numerical-Basic-DRO}
For any ambiguity set \(\mathcal{P}\), the distributionally robust knapsack can be written as
\begin{align}
  \max_{\bm x\in\mathcal{X}_{0}}
  \;\bigl\{\bm c^{\top}\bm x
     \mid
     \bm a_i^\top\bm x+
     \sup_{\mathbb{P}_i\in\mathcal{P}_i}\mathbb{E}_{\mathbb{P}_i}\!\left[\sum_{j\in J}r_j\hat a_{ij}x_j\right]\le b_i,
     \hspace{2mm}\forall i\in[m]
  \bigr\}.
  \tag{Standard-DRO} \label{prob:DRO}
\end{align}
Here $\bm a_i$ is the nominal coefficient vector for row $i$, $\hat a_{ij}=0.01a_{ij}$ is the deviation coefficient, and $\bm r=(r_j)_{j\in J}$ records the continuous deviation fractions.
We examine four ambiguity sets for each row \(i\), where $\mathcal F_i=\mathcal B(\mathcal A_i)$:
\begin{align}
\mathcal{P}_i^{\rm M1}
&=\bigl\{\mathbb{P}\in\mathcal{M}_+(\mathcal{A}_i,\mathcal{F}_i)
    \,\big|\,\mathbb{E}_{\mathbb{P}}[1]=1,\;
    \underline\Gamma_i \le \mathbb{E}_{\mathbb{P}}\bigl[\sum_{j\in J}r_j\bigr]
                   \le \bar\Gamma_i
    \bigr\},\\
\mathcal{P}_i^{\rm M2}
&=\bigl\{\mathbb{P}\in\mathcal{M}_+(\mathcal{A}_i,\mathcal{F}_i)
    \,\big|\,\mathbb{E}_{\mathbb{P}}[1]=1,\;
    \underline\Gamma_i \le \mathbb{E}_{\mathbb{P}}\bigl[\sum_{j\in J}r_j\bigr]
                   \le \bar\Gamma_i,\;
    \underline{\underline\Gamma}_i \le \mathbb{E}_{\mathbb{P}}\bigl[\sum_{j\in J}r_j^2\bigr]
                   \le \bar{\bar\Gamma}_i
    \bigr\},\\
\mathcal{P}_i^{\rm W1}
&=\bigl\{\mathbb{P}\in\mathcal{M}_+(\mathcal{A}_i,\mathcal{F}_i)
    \,\big|\,\mathbb{E}_{\mathbb{P}}[1]=1,\;
    \mathcal W_1(\mathbb P,\hat{\mathbb P}_i)\le\theta
    \bigr\},\\
\mathcal{P}_i^{\rm W2}
&=\bigl\{\mathbb{P}\in\mathcal{M}_+(\mathcal{A}_i,\mathcal{F}_i)
    \,\big|\,\mathbb{E}_{\mathbb{P}}[1]=1,\;
    \mathcal W_2(\mathbb P,\hat{\mathbb P}_i)\le\theta
    \bigr\}.
\end{align} \label{ambiguity_set_numerical}
where $\mathcal A_i=\{\bm r\in[0,1]^{|J|}:\sum_{j\in J}r_j\le\Gamma_i\}$. For empirical samples $\hat{\bm r}_{ik}\in\mathcal A_i$, $k\in[N]$, define $\hat{\mathbb P}_i:=\frac{1}{N}\sum_{k=1}^N\delta_{\hat{\bm r}_{ik}}$. The distances $\mathcal W_1$ and $\mathcal W_2$ are the Wasserstein distances in Definition~\ref{p_w_ambiguity} with $\ell_1$ and $\ell_2$ ground norms, respectively. This model follows the budgeted uncertainty of \cite{bertsimas2004price}: each deviation fraction lies in $[0,1]$, a full deviation changes its nominal coefficient by at most 1\%, and the rowwise total deviation fraction is bounded by $\Gamma_i$. M1 bounds the expected total deviation, while M2 additionally bounds the expected aggregate componentwise second moment $\sum_{j\in J}r_j^2$; Wasserstein sets (W1, W2) remain centered at the empirical distribution $\hat{\mathbb P}_i$. For the ambiguity set parameters, each test group is assigned a different uncertainty budget \(\Gamma_i\) or Wasserstein radius \(\theta\), as reported in the corresponding tables. For moment-inequality ambiguity sets, the normalized first-order moment bounds are drawn as \(\underline{\Gamma}_i/\Gamma_i \sim \mathrm{Unif}(0.2, 0.4)\) and \(\bar{\Gamma}_i/\Gamma_i \sim \mathrm{Unif}(0.6, 0.8)\). The second-order moment bounds are fixed at \(\underline{\underline{\Gamma}}_i = 0.1\,\Gamma_i\) and \(\bar{\bar{\Gamma}}_i = 0.4\,\Gamma_i\). 

\begin{figure}[h!]
    \centering
    \includegraphics[width=\textwidth]{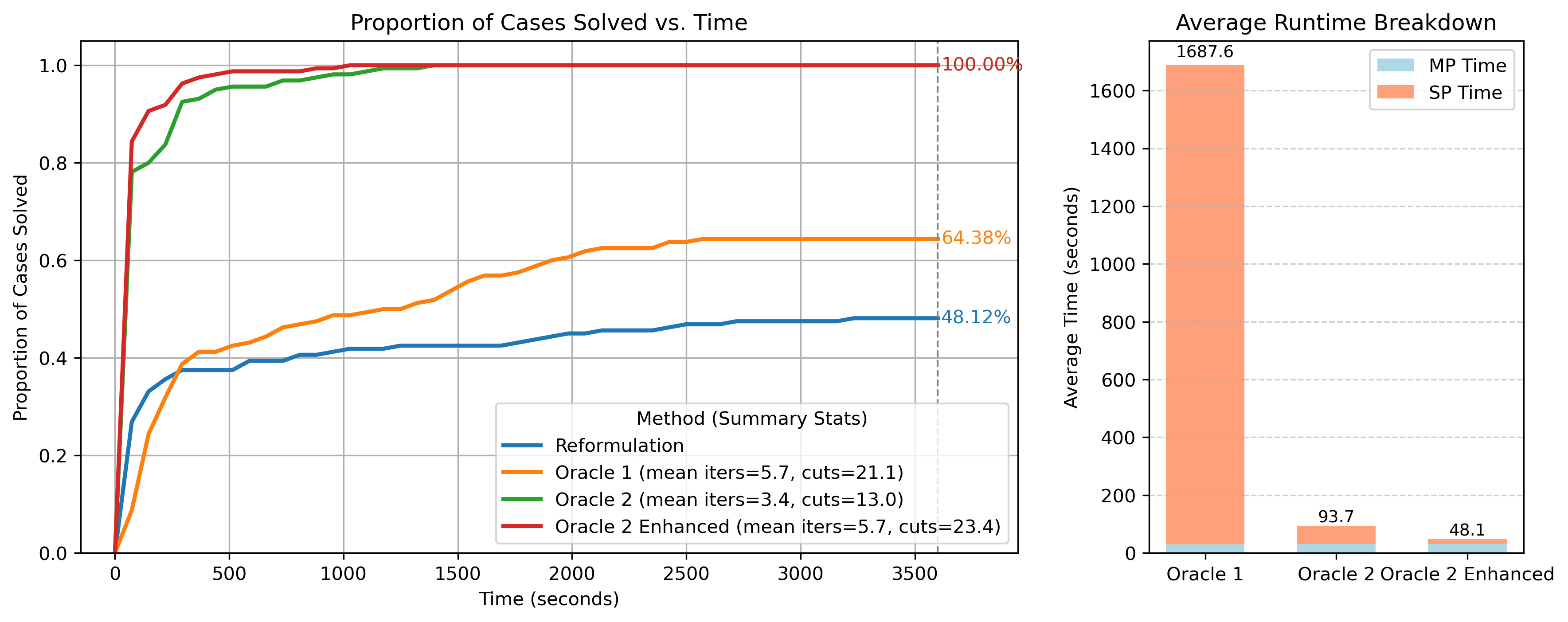}
    \caption{Standard DRO algorithm performance comparison.}
    \label{Basic DRO performance}
\end{figure}

Figure \ref{Basic DRO performance} summarizes the average computational performance across 160 randomly generated cases under the Standard DRO framework. The left panel illustrates the proportion of instances solved over time. As shown, the reformulation approach struggles to scale, solving fewer than half (48.1\%) of the instances within one hour. BiCS with Oracle 1 performs better, solving approximately 64.4\% of the cases, but still fails to converge for more than one-third of the test set. In contrast, BiCS with Oracle 2 and Oracle 2 Enhanced demonstrate strong reliability, successfully solving 100\% of all instances. Both reach full convergence within roughly 1,500 seconds, with Oracle 2 Enhanced achieving faster solution times in most cases.

The right panel decomposes runtime into MP and SP components among the instances solved by each method. For Oracle~1, most reported time is spent in the SP, which involves solving complex nonlinear programs. The detailed results show that Oracle~2 reduces SP time by separating the probability assignment linear program in the PMP from the scenario pricing PSP. All BiCS variants use scenario reduction, while Oracle~2 Enhanced denotes the configuration that additionally applies bound-based stopping. Its reported performance therefore reflects the enhanced configuration as a whole, which can require slightly longer MP time.

From the numerical results in Tables \ref{tab:DRO_first_order_moment}--\ref{tab:DRO_second_order_was}, distinct computational patterns emerge across the solution schemes. For the first-order moment sets (Table \ref{tab:DRO_first_order_moment}), the reformulation remains tractable for instances with continuous variables, but for pure-integer instances it becomes a large MILP with weak relaxations, resulting in extensive branch-and-bound exploration and frequent timeouts. In BiCS, decision variable integrality remains in the master problem; after fixing $\bm x$, Oracle~2 separates the probability assignment PMP from the scenario pricing PSP. This structure yields substantially better performance for the tested IP instances.

For the second-order moment sets (Table \ref{tab:DRO_second_order_moment}), no reformulation is available in our setting, because the second-order term induces a nonconvex objective in the sample-space dual. The difference between Oracle~1 and Oracle~2 follows directly from their subproblem structures. Oracle~1 repeatedly solves large trilinear programs involving both second-order and probability variables, making each iteration computationally expensive. Oracle~2, by contrast, separates the probability assignment PMP from the second-order PSP, avoiding the direct coupling between probabilities and scenarios handled by Oracle~1 and yielding substantially lower solution times in the tested instances.

The Wasserstein ambiguity sets exhibit similar behavior (Tables \ref{tab:DRO_first_order_was}--\ref{tab:DRO_second_order_was}). Under the $\ell_1$-norm Wasserstein metric, the reformulation remains competitive for instances with continuous variables but becomes difficult to solve for pure-integer instances, for reasons analogous to the first-order moment case. Under the $\ell_2$-norm Wasserstein metric, the formulation becomes a high-dimensional MISOCP, and Oracle~1 must repeatedly solve large trilinear subproblems, causing both approaches to scale poorly and often exceed the time limit. Oracle~2 and its enhanced variant remain more stable by separating probability assignment from scenario generation; Oracle~2 Enhanced$^*$ denotes the same enhanced configuration with empirical-sample initialization.

\subsection{Almost Sure DRO} \label{numerical:AS-DRO}

We evaluate AS-DRO using the same controlled knapsack construction as Standard DRO, retaining the data generation procedure and ambiguity set families while changing the risk criterion. Here $\bm r\in\mathcal A_i$ is the deviation pattern for row $i$, and AS-DRO (\ref{eq:ASDRO-sup}) becomes

\begin{align}
  \max_{\bm x\in\mathcal{X}_{0}}
  \;\left\{ \bm c^{\top}\bm x
     \mid
     \sup_{\mathbb{P}\in\mathcal{P}_i} \mathbb{E}_{\mathbb{P}}\!\left[t_i(\bm r,\bm x)\right] \leq 0,
     \quad t_i(\bm r,\bm x)=\left[\bm a_i^\top\bm x+\sum_{j\in J}r_j\hat a_{ij}x_j-b_i\right]_+,
     \hspace{2mm}\forall i\in[m]
  \right\}.
  \tag{AS-DRO} \label{prob:AS-DRO}
\end{align}

All remaining instance parameters follow Section~\ref{Numerical-Basic-DRO}. Figure~\ref{ASDRO performance} summarizes performance across 160 instances. The reformulation solves 35.6\%, while Oracle~1, Oracle~2, and Oracle~2 Enhanced each solve 100\%. The average total times of Oracle~1, Oracle~2, and Oracle~2 Enhanced are 205.7, 1.6, and 0.6 seconds, respectively.

\begin{figure}[h!]
    \centering
    \includegraphics[width=\textwidth]{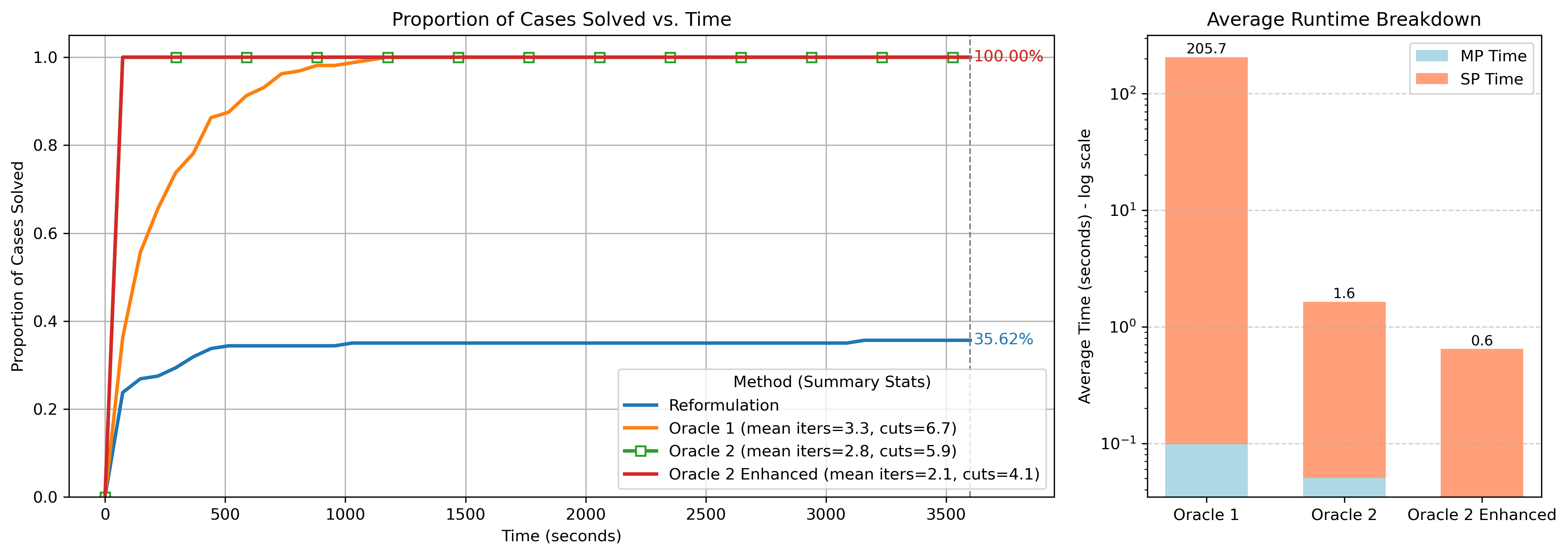}
    \caption{AS-DRO algorithm performance comparison.}
    \label{ASDRO performance}
\end{figure}

The moment results in Tables~\ref{tab:ASDRO_first_order_moment}--\ref{tab:ASDRO_second_order_moment} show a clear contrast. In the M1 cases, the reformulation times out for all displayed continuous cases and for the larger pure-integer cases; while in the M2 cases, the reformulation is still not applicable. Across the displayed moment cases, the maximum total times of Oracle~2 and Oracle~2 Enhanced are 2.15 and 0.85 seconds, respectively. Oracle~1 remains slower because its direct subproblem retains the nonlinear coupling between probabilities and scenarios, whereas Oracle~2 separates probability assignment from scenario generation.

The Wasserstein results in Tables~\ref{tab:ASDRO_first_order_was}--\ref{tab:ASDRO_second_order_was} show the same separation under the stricter positive-part criterion. For the $\ell_1$ cases, the maximum total times of Oracle~2 and Oracle~2 Enhanced are 12.80 and 5.18 seconds; for the $\ell_2$ cases, they are 3.00 and 1.36 seconds. Across all displayed AS-DRO configurations, Oracle~2 Enhanced is faster than Oracle~2. The decomposition controls the scenario search without changing the ambiguity set, while the positive part objective preserves the AS-DRO violation criterion.

\subsection{DRCCP}\label{numerical:DRCCP}

We evaluate DRCCP using the same controlled knapsack settings while changing the risk criterion to grouped chance constraints. The test cases cover individual, joint, and combinatorial chance constraint structures, all handled by the same BiCS framework. For each $m\in[M]$, let $\mathcal A_m:=\prod_{i\in[I_m]}\mathcal A_{mi}$, where each $\mathcal A_{mi}$ is constructed as in Section~\ref{Numerical-Basic-DRO}, let $\bm r_m=(\bm r_{mi})_{i\in[I_m]}\in\mathcal A_m$, and define
$f_{mi}(\bm r_m,\bm x):=\bm a_{mi}^\top\bm x+\sum_{j\in J}r_{mij}\hat a_{mij}x_j$. The test model is

\begin{align}
\max_{\bm x\in\mathcal X_0}\quad & \bm c^\top\bm x \nonumber\\
\text{s.t.}\quad
&\inf_{\mathbb P_m\in\mathcal P_m}
\mathbb P_m\!\left(
\sum_{i=1}^{I_m}\mathbf 1_{\{f_{mi}(\bm r_m,\bm x)\le b_{mi}\}}
\ge h_m
\right)
\ge1-\epsilon_m,
\quad \forall m\in[M].
\tag{DRCCP}  \label{prob:DRCCP}
\end{align}

\begin{figure}[h!]
    \centering
    \includegraphics[width=1\textwidth]{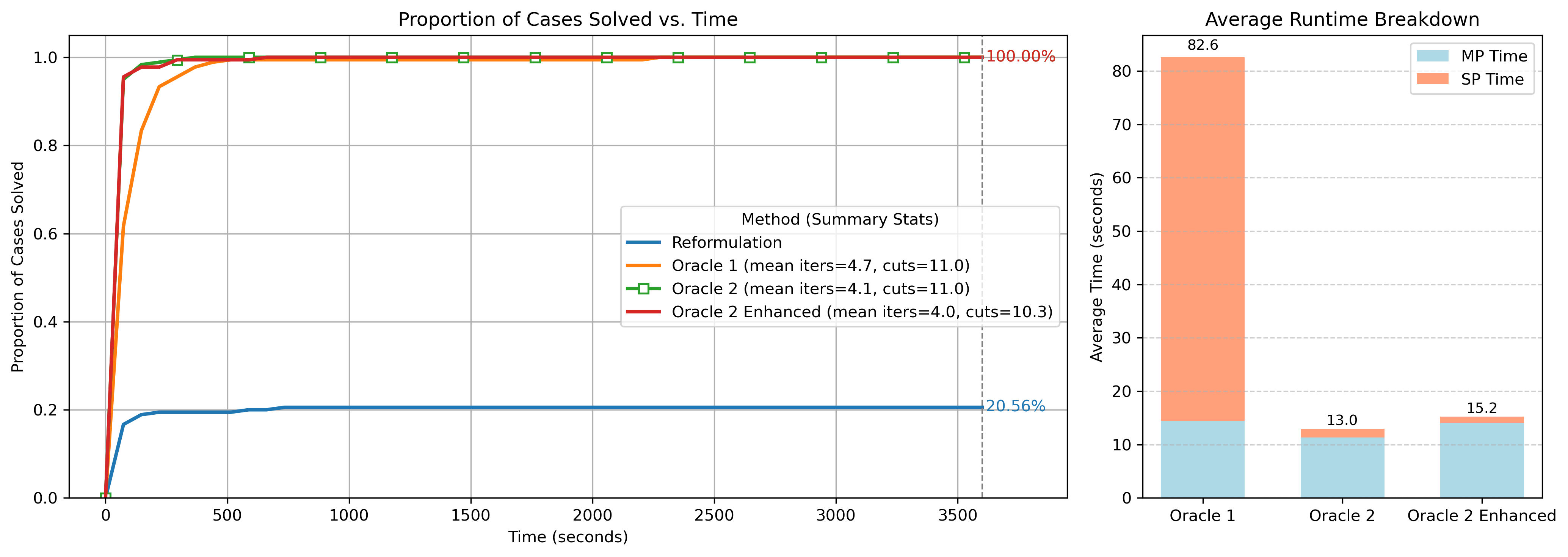}
    \caption{DRCCP algorithm performance comparison.}
    \label{DRCCP performance}
\end{figure}

Figure~\ref{DRCCP performance} summarizes performance across 180 instances. The reformulation is unavailable for the combinatorial M1 and W1 cases, all M2 cases, and the joint and combinatorial W2 cases; over the full test set, it solves 20.56\%, while all three BiCS variants solve 100\%. Average total times are 82.6 seconds for Oracle~1, 13.0 seconds for Oracle~2, and 15.2 seconds for Oracle~2 Enhanced. The enhanced variant is not uniformly faster since reductions in SP work can sometimes be offset by additional MP time.

The moment results in Tables~\ref{tab:DRCCP_first_order_moment}--\ref{tab:DRCCP_second_order_moment} identify where this difference arises. In the dense first-order combinatorial case, SP times are 458.66, 3.79, and 0.06 seconds for Oracle~1, Oracle~2, and Oracle~2 Enhanced. In the second-order cases, the corresponding SP ranges are 53.21--161.75, 0.28--2.69, and 0.19--2.28 seconds. Even when SP time is small, total time reaches 43.82 seconds in a tested combinatorial case because the MP retains the grouped cardinality logic.

The Wasserstein results in Tables~\ref{tab:DRCCP_first_order_was}--\ref{tab:DRCCP_second_order_was} show how separating probability assignment from scenario generation reduces the computational effort under both ground norms. Oracle~2 Enhanced is faster in 32 of the 36 tested DRCCP configurations, but a few large Wasserstein cases incur substantially more MP time. The MP can therefore become the computational bottleneck, especially under grouped satisfaction events, so the enhanced variant need not improve average total time.

\section{Conclusion}\label{sec:conclusion}

In this paper, we studied DRO from a primal perspective that keeps the scenario points and probability weights of a worst-case distribution explicit. This primal viewpoint clarifies how different protection criteria act on the same distributional object: Standard DRO bounds a worst-case expectation, AS-DRO requires feasibility almost surely under every admissible distribution, and DRCCP controls worst-case satisfaction probabilities. It also allows introducing additional local distributional information to refine an ambiguity set without imposing a parametric family.

Building on this perspective, we developed the BiCS framework, which constructs pooled parametric distribution cuts by identifying finite supports, pooling them, and reoptimizing probability weights as the master decision changes. A central contribution is the analysis of BiCS over closed, possibly unbounded sample spaces, where the traditional compactness-based arguments are no longer available. Under the regularity conditions, we established the validity of the generated cuts and the optimality and convergence of the overall framework in this setting. We further introduced scenario reduction and bound-based stopping to control master problem growth and subproblem effort, respectively. The same construction was adapted to AS-DRO, to variants of DRCCPs, and to DRO models whose ambiguity sets incorporate local information.

Numerical experiments with moment and Wasserstein ambiguity sets show how these structural choices affect computation. Across the tested Standard DRO, AS-DRO, and DRCCP instances, Oracle~2 and its enhanced variant, which separate probability assignment from scenario generation, were the most reliable BiCS implementations. The enhanced variant was faster in most Standard DRO cases and all tested AS-DRO cases, whereas grouped satisfaction logic in DRCCP could shift computation to the master problem and prevent a uniform reduction in total time. The local information study shows how the primal construction makes changes in the worst-case distribution directly visible. A natural next step is to establish convergence guarantees for BiCS when the ambiguity set or sample space depends on the decision.

\begin{landscape}
\begin{table}[ht!]
    \centering
    \caption{Performance metrics for Standard DRO with First-order Moment Ambiguity Sets}
    \label{tab:DRO_first_order_moment}
    \renewcommand{\arraystretch}{1.0} 
    \setlength{\tabcolsep}{4pt}       

    \begin{threeparttable}
    \begin{tabular}{cccc cc ccccc ccccc ccccc}
    \toprule
    \multirow{2}{*}{Case} 
      & \multirow{2}{*}{$d(A)$} 
      & \multirow{2}{*}{$\Gamma$} 
      & \multirow{2}{*}{Var.} 
      & \multicolumn{2}{c}{Reformulation} 
      & \multicolumn{5}{c}{Oracle 1} 
      & \multicolumn{5}{c}{Oracle 2} 
      & \multicolumn{5}{c}{Oracle 2 Enhanced} \\
    \cmidrule(lr){5-6} \cmidrule(lr){7-11} \cmidrule(lr){12-16} \cmidrule(lr){17-21}
      &  &  & 
      & Gap & \cellcolor{gray!20}$t_{\mathrm{total}}$ 
      & Iter. & Cuts & $|\hat{\mathcal{A}}|$
      & $t_{\mathrm{sp}}$ & \cellcolor{gray!20}$t_{\mathrm{total}}$ 
      & Iter. & Cuts & $|\hat{\mathcal{A}}|$
      & $t_{\mathrm{sp}}$ & \cellcolor{gray!20}$t_{\mathrm{total}}$ 
      & Iter. & Cuts & $|\hat{\mathcal{A}}|$ 
      & $t_{\mathrm{sp}}$ & \cellcolor{gray!20}$t_{\mathrm{total}}$\\
    \midrule

    1  & D & 5  & M  & 0  & \cellcolor{gray!20}16.8  & 
     2.2  & 11.4  & 11.4  & 70.5  & \cellcolor{gray!20}75.1  &
     2.2  & 11.4  & 16.2  & 32.9  & \cellcolor{gray!20}37.6  &
     6.4  & 50.8  & 50.8  & 18.3  & \cellcolor{gray!20}24.1  \\
    2  & D & 40 & M & 0  & \cellcolor{gray!20}16.2  & 
     3.2  & 21.8  & 21.8  & 136.6 & \cellcolor{gray!20}141.5 &
     3.2  & 21.8  & 23.4  & 22.4  & \cellcolor{gray!20}27.3  &
     3.6  & 22.2  & 22.2  & 10.3  & \cellcolor{gray!20}15.3  \\
    3  & D & 5  & I   & 1.53\%  & \cellcolor{gray!20}T 
    & 3.2  & 3.4   & 3.4   & 105.1 & \cellcolor{gray!20}163.7 
    & 3.2  & 3.4   & 4.6   & 48.9  & \cellcolor{gray!20}113.4 
    & 3.2  & 3.4   & 4.6   & 9.3   & \cellcolor{gray!20}72.0  \\
    4  & D & 40 & I   & 0.68\%  & \cellcolor{gray!20}T 
    & 3.8  & 6.2   & 6.2   & 123.7 & \cellcolor{gray!20}174.5 
    & 4.0  & 6.8   & 8.4   & 22.9  & \cellcolor{gray!20}82.3  
    & 4.0  & 6.8   & 7.8   &11.2  & \cellcolor{gray!20}71.3  \\
    5  & S  & 5  & M  & 0  & \cellcolor{gray!20}4.6    
    & 3.4  & 40.2  & 40.2  & 32.8    & \cellcolor{gray!20}34.3  
    & 3.4  & 39.6  & 55.2  & 14.7  & \cellcolor{gray!20}16.2 
    & 7.4  & 134.6 & 134.6 & 7.2   & \cellcolor{gray!20}9.3  \\
    6  & S & 40 & M  & 0  & \cellcolor{gray!20}4.5   
     & 3.4  & 42.8  & 45.6  & 35.0   & \cellcolor{gray!20}36.5  
     & 3.0  & 41.6  & 63.2  & 5.52   & \cellcolor{gray!20}7.0  
     & 4.4  & 60.8  & 60.8  & 4.15   & \cellcolor{gray!20}8.2  \\
    7  & S  & 5  & I   & 0.94\%  & \cellcolor{gray!20}T 
    & 6.8  & 14.0  & 14.0  & 56.3  & \cellcolor{gray!20}356.2 
    & 6.8  & 14.0  & 20.6  & 30.3  & \cellcolor{gray!20}362.3 
    & 7.6  & 15.2  & 23.8  & 7.2   & \cellcolor{gray!20}420.2 \\
    8  & S  & 40 &  I    & 0.12\%  & \cellcolor{gray!20}T      
    & 7.0  & 15.6  & 16.4  & 73.4  & \cellcolor{gray!20}442.3 
    & 7.2  & 16.8  & 26.8  & 15.7  & \cellcolor{gray!20}455.2 
    & 7.0  & 16.2  & 23.0  & 6.5   & \cellcolor{gray!20}431.2  \\

    \bottomrule
    \end{tabular}
    \end{threeparttable}
\end{table}

\begin{table}[ht!]
    \centering
    \caption{Performance metrics for Standard DRO with Second-order Moment Ambiguity Sets}
    \label{tab:DRO_second_order_moment}
    \renewcommand{\arraystretch}{1.0} 
    \setlength{\tabcolsep}{4.5pt}

    \begin{threeparttable}
    \begin{tabular}{ccccc ccccc ccccc ccccc}
    \toprule
    \multirow{2}{*}{Case} 
      & \multirow{2}{*}{Size} 
      & \multirow{2}{*}{$d(A)$} 
      & \multirow{2}{*}{$\Gamma$} 
      & \multirow{2}{*}{Var.} 
      & \multicolumn{5}{c}{Oracle 1} 
      & \multicolumn{5}{c}{ Oracle 2} 
      & \multicolumn{5}{c}{Oracle 2 Enhanced} \\
    \cmidrule(lr){6-10} \cmidrule(lr){11-15} \cmidrule(lr){16-20} 
      &  &  &  &
      & Iter. & Cuts & $|\hat{\mathcal{A}}|$
      & $t_{\mathrm{sp}}$ & \cellcolor{gray!20}$t_{\mathrm{total}}$ & Iter. & Cuts & $|\hat{\mathcal{A}}|$
      & $t_{\mathrm{sp}}$ & \cellcolor{gray!20}$t_{\mathrm{total}}$ & Iter. & Cuts & $|\hat{\mathcal{A}}|$ &
      $t_{\mathrm{sp}}$ & \cellcolor{gray!20}$t_{\mathrm{total}}$ \\
    \midrule

   1  & (10,3) & D  & 2  & I  
   & 2.4  & 3.0  & 15.0  & 291.45  & \cellcolor{gray!20}291.45  
   & 2.0  & 2.4  & 7.2  & 0.18  & \cellcolor{gray!20}0.18  
   & 2.0  & 2.4  & 4.8  & 0.01  & \cellcolor{gray!20}0.02  \\
    2  & (10,3) & D & 2  & M  
    & 30.0 & 43.0 & 215.0 & T       & \cellcolor{gray!20}T      
    & 3.6  & 4.8  & 13.4 & 0.26  & \cellcolor{gray!20}0.27  
    & 19.4 & 23.8 & 26.2  & 0.08  & \cellcolor{gray!20}0.12  \\
    3  & (10,3) & S  & 2  & I  
    & 2.2  & 2.8  & 14.0  & 251.19  & \cellcolor{gray!20}251.20  
    & 2.0  & 2.6  & 7.0   & 0.11  & \cellcolor{gray!20}0.12  
    & 2.0  & 2.6  & 2.8   & 0.01  & \cellcolor{gray!20}0.02  \\
    4  & (10,3) & S & 2  & M  
    & 30  & 66.4 & 332.0 & T & \cellcolor{gray!20}T 
    & 3.6  & 5.8  & 13.4 & 0.16  & \cellcolor{gray!20}0.17 
    & 13.8 & 26.6 & 26.2 & 0.07  & \cellcolor{gray!20}0.14  \\
    5  & (100,20) & D  & 10  & I  
    & 2.0  & 3.2 & 16.0  & 1517.79 & \cellcolor{gray!20}1518.32  
     & 2.2  & 3.2  & 8.4 & 3.06  & \cellcolor{gray!20}3.26 
     & 2.2  & 3.4  & 5.0 & 0.23  & \cellcolor{gray!20}0.46 \\
    6  & (100,20) & D & 10  & M 
    & 5.2  & 18.4 & 92.0 & T       & \cellcolor{gray!20}T     
     & 4.8  & 14.4 & 32.6 & 7.20  & \cellcolor{gray!20}7.29 
     & 19.6 & 57.4 & 57.8 & 2.18  & \cellcolor{gray!20}2.43 \\
    7  & (100,20) & S  & 10  & I  
    & 2.6  & 3.8  & 16.0  & 1854.52 & \cellcolor{gray!20}1854.52 
    & 2.6  & 3.8  & 7.6  & 2.09  & \cellcolor{gray!20}2.29  
    & 2.4  & 3.6  & 4.4 & 0.11  & \cellcolor{gray!20}0.31  \\
    8  & (100,20) & S & 10  & M  
    & 4.6  & 28.8 & 144.0 & T       & \cellcolor{gray!20}T     
    & 3.6  & 16.6 & 34.8 & 2.90  & \cellcolor{gray!20}2.95  
    & 20.6 & 77.0 & 77.0  & 1.04  & \cellcolor{gray!20}1.24  \\

    \bottomrule
    \end{tabular}

    \begin{tablenotes}\footnotesize
    \item \textbf{Note:} For Tables~\ref{tab:DRO_first_order_moment} and~\ref{tab:DRO_second_order_moment}, T indicates a timeout at 3600 seconds and each row reports averages over 5 random seeds. D denotes 75\% positive and 5\% negative entries, while S denotes 40\% positive entries. The first-order table has 600 variables and 300 constraints; in the second-order table, Size reports (number of variables, number of constraints). The parameter $\Gamma$ is the rowwise uncertainty budget. In the Var. column, M denotes instances with continuous decision variables and I denotes pure-integer instances.
    \end{tablenotes}
    
    \end{threeparttable}
\end{table}

\begin{table}[ht!]
    \centering
    \caption{Performance metrics for Standard DRO with $\ell_1$-norm Wasserstein Ambiguity Sets}
    \label{tab:DRO_first_order_was}
    \renewcommand{\arraystretch}{1.0}
    \setlength{\tabcolsep}{2.5pt}

    \begin{threeparttable}
    \begin{tabular}{ccccccccccccccccccccc}
    \toprule
    \multirow{2}{*}{Case} 
      & \multirow{2}{*}{$\theta$} 
      & \multirow{2}{*}{$N$} 
      & \multirow{2}{*}{Var.} 
      & \multicolumn{2}{c}{Reformulation} 
      & \multicolumn{5}{c}{Oracle 1} 
      & \multicolumn{5}{c}{Oracle 2} 
      & \multicolumn{5}{c}{Oracle~2 Enhanced$^*$} \\
    \cmidrule(lr){5-6} \cmidrule(lr){7-11} \cmidrule(lr){12-16} \cmidrule(lr){17-21}
      &  &  & 
      & Gap & \cellcolor{gray!20}$t_{\mathrm{total}}$
      & Iter. & Cuts & $|\hat{\mathcal{A}}|$
      & $t_{\mathrm{sp}}$ & \cellcolor{gray!20}$t_{\mathrm{total}}$ & Iter. & Cuts & $|\hat{\mathcal{A}}|$
      & $t_{\mathrm{sp}}$ & \cellcolor{gray!20}$t_{\mathrm{total}}$ & Iter. & Cuts & $|\hat{\mathcal{A}}|$ 
      & $t_{\mathrm{sp}}$ & \cellcolor{gray!20}$t_{\mathrm{total}}$ \\
    \midrule
    1 & 1    & 30  & I & 0       & \cellcolor{gray!20}136.6 
    & 2.2 & 12.6 & 526.6   & 103.1 & \cellcolor{gray!20}108.4 
    & 2.0 & 12.4 & 274.2  & 20.3 & \cellcolor{gray!20}24.4  
     & 1.2 & 0.2 & 6.0 & 4.3  & \cellcolor{gray!20}6.9  \\
    2 & 1    & 300 & I & 0.6\%   & \cellcolor{gray!20}T 
     & 2.0 & 12.6 & 5453.4  & 1453.8 & \cellcolor{gray!20}1514.6
    & 2.0 & 12.6 & 2519.2  & 233.0 & \cellcolor{gray!20}275.1 
     & 1.0 & 0.0 & 0 & 38.2 & \cellcolor{gray!20}54.8 \\
    3 & 0.01 & 30  & I & 0       & \cellcolor{gray!20}299.5 
     & 2.0 & 12.4 & 541.0   & 105.1 & \cellcolor{gray!20}109.2
     & 2.0 & 12.4 & 256.4   & 13.7 & \cellcolor{gray!20}17.3
     & 1.0 & 0.0 & 0 & 3.5  & \cellcolor{gray!20}5.8  \\
    4 & 0.01 & 300 & I & 0.8\%   & \cellcolor{gray!20}T 
    & 2.0 & 12.4 & 7424.4  & 611.0 & \cellcolor{gray!20}699.1 
    & 2.0 & 12.4 & 2472.2  & 174.7 & \cellcolor{gray!20}222.9 
     & 1.0 & 0.0 & 0 & 37.3 & \cellcolor{gray!20}54.0 \\
    5 & 1    & 30  & M & 0       & \cellcolor{gray!20}5.0   
    & 3.0 & 19.2 & 758.0   & 110.7 & \cellcolor{gray!20}112.5 
    & 6.0 & 31.0 & 479.6   & 52.2 & \cellcolor{gray!20}54.2 
     & 7.0 & 44.0 & 1057.6 & 20.9 & \cellcolor{gray!20}23.0  \\
    6 & 1    & 300 & M & 0       & \cellcolor{gray!20}101.1 
    & 3.0 & 18.8 & 8158.6  & 1886.7 & \cellcolor{gray!20}1905.4 
    & 10.4 & 39.0 & 4773.2  & 955.7 & \cellcolor{gray!20}983.3 
    & 7.0 & 43.0 & 10318.6 & 217.0 & \cellcolor{gray!20}242.5 \\
    7 & 0.01 & 30  & M & 0   & \cellcolor{gray!20}4.9   
    & 3.4 & 18.6 & 979.8   & 179.2 & \cellcolor{gray!20}181.1 
    & 3.0 & 16.4 & 260.4   & 18.6 & \cellcolor{gray!20}20.4 
     & 4.0 & 20.4 & 588.4 & 11.8 & \cellcolor{gray!20}13.7 \\
    8 & 0.01 & 300 & M & 0      & \cellcolor{gray!20}86.8  
    & 3.2 & 20.6 & 11995.2 & 921.6 & \cellcolor{gray!20}943.9 
    & 3.0 & 16.4 & 2492.8  & 250.9 & \cellcolor{gray!20}269.1 
    & 4.0 & 18.6 & 5385.4 & 119.8 & \cellcolor{gray!20}138.8 \\
    \bottomrule
    \end{tabular}

    \begin{tablenotes}\footnotesize
    \item \textbf{Note:} T indicates a timeout at 3600 seconds; each row reports averages over 5 random seeds. All cases have 150 variables and 20 constraints, $\Gamma=30$, and $(d_p,d_n)=(0.3,0.1)$. $\theta$ is the Wasserstein radius and $N$ is the number of empirical sample points. In the Var. column, M denotes instances with continuous decision variables and I denotes pure-integer instances. The superscript $^*$ indicates that the empirical points defining $\hat{\mathbb P}_i$ are also used to initialize the scenario pool; this warm start does not otherwise alter the ambiguity set.
    \end{tablenotes}
    \end{threeparttable}
\end{table}

\begin{table}[ht!]
    \centering
    \caption{Performance metrics for Standard DRO with $\ell_2$-norm Wasserstein Ambiguity Sets}
    \label{tab:DRO_second_order_was}
    \renewcommand{\arraystretch}{1.0}
    \setlength{\tabcolsep}{2.5pt}

    \begin{threeparttable}
    \begin{tabular}{ccccccccccccccccccccc}
    \toprule
    \multirow{2}{*}{Case} 
      & \multirow{2}{*}{$\theta$} 
      & \multirow{2}{*}{$N$} 
      & \multirow{2}{*}{Var.} 
      & \multicolumn{2}{c}{Reformulation} 
      & \multicolumn{5}{c}{Oracle 1} 
      & \multicolumn{5}{c}{Oracle 2} 
      & \multicolumn{5}{c}{Oracle~2 Enhanced$^*$} \\
    \cmidrule(lr){5-6} \cmidrule(lr){7-11} \cmidrule(lr){12-16} \cmidrule(lr){17-21}
      &  &  & 
      & Gap & \cellcolor{gray!20}$t_{\mathrm{total}}$
      & Iter. & Cuts & $|\hat{\mathcal{A}}|$
      & $t_{\mathrm{sp}}$ & \cellcolor{gray!20}$t_{\mathrm{total}}$ & Iter. & Cuts & $|\hat{\mathcal{A}}|$
      & $t_{\mathrm{sp}}$ & \cellcolor{gray!20}$t_{\mathrm{total}}$ & Iter. & Cuts & $|\hat{\mathcal{A}}|$ & $t_{\mathrm{sp}}$ & \cellcolor{gray!20}$t_{\mathrm{total}}$ \\
    \midrule
    1 & 1    & 30  & I & 0.27\% & \cellcolor{gray!20}T      
    & 5.4 & 11.2 & 442.0  & T & \cellcolor{gray!20}T 
    & 2.6 & 1.8 & 19.0 & 6.0 & \cellcolor{gray!20}7.0 
    & 2.6 & 2.0 & 78.0 & 1.2 & \cellcolor{gray!20}4.0 \\
    2 & 1    & 100 & I & 5.68\% & \cellcolor{gray!20}T      
    & 5.4 & 13.8 & 1859.6 & T & \cellcolor{gray!20}T 
    & 2.6 & 1.8 & 61.8 & 19.5 & \cellcolor{gray!20}22.0 
    & 2.6 & 2.0 & 239.6 & 3.9 & \cellcolor{gray!20}11.9 \\
    3 & 0.01 & 5   & I & 0 & \cellcolor{gray!20}29.9   
    & 6.0 & 11.8 & 110.8  & T & \cellcolor{gray!20}T 
    & 1.0 & 0.8 & 2.2 & 0.7 & \cellcolor{gray!20}1.0  
    & 1.0 & 0.0 & 0.0 & 0.1 & \cellcolor{gray!20}0.4 \\
    4 & 0.01 & 30  & I & 0.16\% & \cellcolor{gray!20}T      
    & 6.0 & 17.7 & 733.0  & T & \cellcolor{gray!20}T 
    & 3.7 & 3.7 & 15.0 & 5.0 & \cellcolor{gray!20}6.0
    & 1.0 & 0.0 & 0.0 & 1.0 & \cellcolor{gray!20}2.1 \\
    5 & 1    & 30  & M & 0.10\% & \cellcolor{gray!20}T      
    & 6.0 & 36.0 & 1421.0 & T & \cellcolor{gray!20}T 
    & 4.0 & 20.0 & 66.7 & 4.1 & \cellcolor{gray!20}4.8 
    & 6.0 & 24.3 & 700.7 & 2.5 & \cellcolor{gray!20}7.4 \\
    6 & 1    & 100 & M & 0.31\% & \cellcolor{gray!20}T      
    & 6.0 & 36.0 & 5373.3 & T & \cellcolor{gray!20}T 
    & 6.0 & 219.0 & 0.0 & 14.0 & \cellcolor{gray!20}16.2 
    & 5.7 & 22.0 & 2126.7 & 7.2 & \cellcolor{gray!20}23.2 \\
    7 & 0.01 & 5   & M & 0 & \cellcolor{gray!20}20.2   
    & 6.0 & 37.0 & 314.4  & T & \cellcolor{gray!20}T 
    & 2.7 & 7.7 & 9.0  & 0.8 & \cellcolor{gray!20}1.0 
    & 5.3 & 24.3 & 119.7 & 0.4 & \cellcolor{gray!20}1.2 \\
    8 & 0.01 & 30  & M & 0.01\% & \cellcolor{gray!20}T   
    & 6.0 & 38.0 & 1469.0 & T & \cellcolor{gray!20}T 
    & 3.0 & 9.7 & 21.0 & 5.0 & \cellcolor{gray!20}5.7 
    & 5.0 & 22.7 & 662.0 & 2.0 & \cellcolor{gray!20}6.3 \\
    \bottomrule
    \end{tabular}

    \begin{tablenotes}\footnotesize
    \item \textbf{Note:} T indicates a timeout at 3600 seconds; each row reports averages over 5 random seeds. All cases have 80 variables and 15 constraints, $\Gamma=10$, and $(d_p,d_n)=(0.3,0.1)$. $\theta$ is the Wasserstein radius and $N$ is the number of empirical sample points. In the Var. column, M denotes instances with continuous decision variables and I denotes pure-integer instances. The superscript $^*$ indicates that the empirical points defining $\hat{\mathbb P}_i$ are also used to initialize the scenario pool; this warm start does not otherwise alter the ambiguity set.
    \end{tablenotes}
    \end{threeparttable}
\end{table}

\end{landscape}

\begin{landscape}

\begin{table}[ht!]
    \centering
    \caption{Performance metrics for AS-DRO with First-order Moment Ambiguity Sets}
    \label{tab:ASDRO_first_order_moment}
    \renewcommand{\arraystretch}{1.0}
    \setlength{\tabcolsep}{2pt}

    \begin{threeparttable}
    \begin{tabular}{ccccc cc ccccc ccccc ccccc}
    \toprule
    \multirow{2}{*}{Case} 
      & \multirow{2}{*}{Size} 
      & \multirow{2}{*}{$d(A)$} 
      & \multirow{2}{*}{$\Gamma$} 
      & \multirow{2}{*}{Var.} 
      & \multicolumn{2}{c}{Reformulation} 
      & \multicolumn{5}{c}{Oracle 1} 
      & \multicolumn{5}{c}{Oracle 2} 
      & \multicolumn{5}{c}{Oracle 2 Enhanced} \\
      \cmidrule(lr){6-7} \cmidrule(lr){8-12} \cmidrule(lr){13-17} \cmidrule(lr){18-22}
      &  &  &  &
      & Gap & \cellcolor{gray!20}$t_{\mathrm{total}}$
      & Iter. & Cuts & $|\hat{\mathcal{A}}|$
      & $t_{\mathrm{sp}}$ & \cellcolor{gray!20}$t_{\mathrm{total}}$ & Iter. & Cuts & $|\hat{\mathcal{A}}|$
      & $t_{\mathrm{sp}}$ & \cellcolor{gray!20}$t_{\mathrm{total}}$ & Iter. & Cuts & $|\hat{\mathcal{A}}|$ 
      & $t_{\mathrm{sp}}$ & \cellcolor{gray!20}$t_{\mathrm{total}}$ \\
    \midrule

    1 & (20,5)   & D  & 2  & M & 31.69\% & \cellcolor{gray!20}T 
    & 4.8 & 11.6 & 32.0 & 22.7  & \cellcolor{gray!20}22.8 
    & 3.0 & 6.8 & 13.4 &  0.10  & \cellcolor{gray!20}0.12
    & 3.0 & 6.8 & 6.8 & 0.04  & \cellcolor{gray!20}0.07 \\
    2 & (20,5)   & D  & 2  & I & 0 & \cellcolor{gray!20}3.0 
    &3.2  & 7.2 & 21.2 & 20.3  & \cellcolor{gray!20}20.3 
    & 2.0 & 4.4 & 8.8 & 0.06  & \cellcolor{gray!20}0.07 
    & 2.0 & 4.4 & 4.4 & 0.04 & \cellcolor{gray!20}0.04 \\
    3 & (20,5)   & S  & 2  & M & 23.11\% & \cellcolor{gray!20}T 
    & 4.2 & 8.8 & 24.6 & 6.4 & \cellcolor{gray!20}6.4 
    & 3.2 & 6.2 & 11.6 & 0.07 & \cellcolor{gray!20}0.09 
    & 3.4 & 6.8 & 6.8 & 0.04 & \cellcolor{gray!20}0.05 \\
    4 & (20,5)   & S  & 2  & I & 0 & \cellcolor{gray!20}1.4 
    & 2.6 & 4.6 & 13.4 & 5.8 & \cellcolor{gray!20}5.8   
     & 2.0   & 4.0 & 8.0 & 0.05   & \cellcolor{gray!20}0.06 
    & 2.0   & 4.0  & 4.0  & 0.02 & \cellcolor{gray!20}0.03  \\
    5 & (100,30) & D  & 10 & M & 491.94\% & \cellcolor{gray!20}T 
    & 3.8 & 9.0 & 20.6 & 23.4 & \cellcolor{gray!20}23.5 
    & 3.4 & 8.6 & 12.0  & 0.81 & \cellcolor{gray!20}0.92
    & 2.4 & 5.4 & 5.4 & 0.43 & \cellcolor{gray!20}0.53 \\
    6 & (100,30) & D  & 10 & I & 26.56\% & \cellcolor{gray!20}T 
    & 1.6  & 1.0 & 2.6  & 1.5 & \cellcolor{gray!20}1.7 
    & 1.6 & 1.0 & 1.4  & 0.30 & \cellcolor{gray!20}0.53 
    & 1.6 & 1.0 & 1.0 & 0.27& \cellcolor{gray!20}0.43 \\
    7 & (100,30) & S  & 10 & M & 199.0\% & \cellcolor{gray!20}T 
    & 4.2 & 15.6 & 33.0 & 10.8 & \cellcolor{gray!20}10.9 
    & 4.2 & 15.4 & 20.4 & 0.71 & \cellcolor{gray!20}0.79 
    & 2.8 & 9.8 & 9.8 & 0.31 & \cellcolor{gray!20}0.37 \\
    8 & (100,30) & S  & 10 & I & 33.52\% & \cellcolor{gray!20}T  
    & 2.0  & 1.6 & 3.2 & 0.8   & \cellcolor{gray!20}1.0 
    & 2.0 & 1.6 & 2.6 & 0.22 & \cellcolor{gray!20}0.43
    & 2.0   & 1.6 & 1.6 & 0.19 & \cellcolor{gray!20}0.34 \\

    \bottomrule
    \end{tabular}
    
    \end{threeparttable}
\end{table}

\begin{table}[ht!]
    \centering
    \caption{Performance metrics for AS-DRO with Second-order Moment Ambiguity Sets}
    \label{tab:ASDRO_second_order_moment}
    \renewcommand{\arraystretch}{1.0}
    \setlength{\tabcolsep}{3pt}

    \begin{threeparttable}
    \begin{tabular}{ccccc ccccc ccccc ccccc}
    \toprule
    \multirow{2}{*}{Case} 
      & \multirow{2}{*}{Size} 
      & \multirow{2}{*}{$d(A)$} 
      & \multirow{2}{*}{$\Gamma$} 
      & \multirow{2}{*}{Var.} 
      & \multicolumn{5}{c}{Oracle 1} 
      & \multicolumn{5}{c}{Oracle 2} 
      & \multicolumn{5}{c}{Oracle 2 Enhanced} \\
      \cmidrule(lr){6-10} \cmidrule(lr){11-15} \cmidrule(lr){16-20}
      &  &  &  &
      & Iter. & Cuts & $|\hat{\mathcal{A}}|$ & $t_{\mathrm{sp}}$ & \cellcolor{gray!20}$t_{\mathrm{total}}$ & Iter. & Cuts & $|\hat{\mathcal{A}}|$
      & $t_{\mathrm{sp}}$ & \cellcolor{gray!20}$t_{\mathrm{total}}$ & Iter. & Cuts & $|\hat{\mathcal{A}}|$
      & $t_{\mathrm{sp}}$ & \cellcolor{gray!20}$t_{\mathrm{total}}$ \\
    \midrule
    1 & (10,3)   & D  & 2  & M & 
     4.8 & 6.2 & 31.0 & 211.07 & \cellcolor{gray!20}211.10 &
    2.8 & 4.0 & 9.6 & 0.38 & \cellcolor{gray!20}0.40 & 
    3.0 & 4.2 & 6.6 & 0.06 & \cellcolor{gray!20}0.08 \\
    2 & (10,3)   & D  & 2  & I & 
     2.2 & 2.8 & 14.0 & 74.09  & \cellcolor{gray!20}74.10 &
     3.0 & 4.6 & 9.4 & 0.33 & \cellcolor{gray!20}0.34 &
     2.0 & 3.0 & 5.4 & 0.04 & \cellcolor{gray!20}0.05 \\
    3 & (10,3)   & S  & 2  & M & 
     4.0 & 6.2 & 29.2 &  216.10 & \cellcolor{gray!20}216.12 &
     2.6 & 3.8 & 8.4 & 0.30 & \cellcolor{gray!20}0.32 & 
     3.0 & 4.6 & 4.8 & 0.05 & \cellcolor{gray!20}0.06 \\
    4 & (10,3)   & S  & 2  & I 
     & 2.0 & 2.6 & 13.0 & 97.16  &  \cellcolor{gray!20}97.17 
     & 2.8 & 4.2 & 8.6  & 0.28 & \cellcolor{gray!20}0.29 
     & 2.2 & 3.2 & 3.4 & 0.03 & \cellcolor{gray!20}0.04 \\
    5 & (100,20) & D  & 10 & M 
    & 6.2 & 17.0 & 80.2 & 661.46 & \cellcolor{gray!20}661.70 
    & 2.8 & 9.6 & 20.0 & 2.04 & \cellcolor{gray!20}2.15 
    & 3.8 & 12.6 & 13.8 & 0.74 & \cellcolor{gray!20}0.85 \\
    6 & (100,20) & D  & 10 & I 
    & 2.4 & 4.6 & 22.2 & 175.78 & \cellcolor{gray!20}176.16 
    & 2.6 & 6.2 & 11.6 & 1.06 & \cellcolor{gray!20}1.43 
    & 2.8 & 7.6 & 8.2  & 0.49 & \cellcolor{gray!20}0.85 \\
    7 & (100,20) & S  & 10 & M 
    & 5.6 & 22.0 & 103.0 & 803.08 & \cellcolor{gray!20}803.28 
    & 2.8 & 11.4 & 20.6 & 1.27 & \cellcolor{gray!20}1.35 
    & 3.6 & 15.8 & 17.8 & 0.49 & \cellcolor{gray!20}0.57 \\
    8 & (100,20) & S  & 10 & I 
    & 2.6 & 4.6 & 21.4 & 169.40 & \cellcolor{gray!20}169.94 
    & 2.4 & 6.0 & 9.8  & 0.66 & \cellcolor{gray!20}0.98 
    & 2.8 & 7.0 & 7.4 & 0.33 & \cellcolor{gray!20}0.66 \\
    \bottomrule
    \end{tabular}

    \begin{tablenotes}\footnotesize
    \item \textbf{Note:} For Tables~\ref{tab:ASDRO_first_order_moment} and~\ref{tab:ASDRO_second_order_moment}, T indicates a timeout at 3600 seconds and each row reports averages over 5 random seeds. Size reports (number of variables, number of constraints). D denotes 75\% positive and 5\% negative entries, while S denotes 40\% positive and 0\% negative entries. The parameter $\Gamma$ is the rowwise uncertainty budget. In the Var. column, M denotes instances with continuous decision variables and I denotes pure-integer instances.
    \end{tablenotes}
    \end{threeparttable}
\end{table}
\end{landscape}

\begin{landscape}
    \begin{table}[ht!]
    \centering
    \caption{Performance metrics for AS-DRO with $\ell_1$-norm Wasserstein Ambiguity Sets}
    \label{tab:ASDRO_first_order_was}
    \renewcommand{\arraystretch}{1.0}
    \setlength{\tabcolsep}{2.5pt}

    \begin{threeparttable}
    \begin{tabular}{ccccccccccccccccccccc}
    \toprule
    \multirow{2}{*}{Case} 
      & \multirow{2}{*}{$\theta$} 
      & \multirow{2}{*}{$N$} 
      & \multirow{2}{*}{Var.} 
      & \multicolumn{2}{c}{Reformulation} 
      & \multicolumn{5}{c}{Oracle 1} 
      & \multicolumn{5}{c}{Oracle 2} 
      & \multicolumn{5}{c}{Oracle~2 Enhanced$^*$} \\
    \cmidrule(lr){5-6} \cmidrule(lr){7-11} \cmidrule(lr){12-16} \cmidrule(lr){17-21}
      &  &  & 
      & Gap & \cellcolor{gray!20}$t_{\mathrm{total}}$
      & Iter. & Cuts & $|\hat{\mathcal{A}}|$
      & $t_{\mathrm{sp}}$ & \cellcolor{gray!20}$t_{\mathrm{total}}$ & Iter. & Cuts & $|\hat{\mathcal{A}}|$
      & $t_{\mathrm{sp}}$ & \cellcolor{gray!20}$t_{\mathrm{total}}$ & Iter. & Cuts & $|\hat{\mathcal{A}}|$
      & $t_{\mathrm{sp}}$ & \cellcolor{gray!20}$t_{\mathrm{total}}$ \\
    \midrule
    1 & 0.01 & 5  & M & 27.96\% & \cellcolor{gray!20}T 
    & 3.6 & 11.0 & 110 & 446.90 & \cellcolor{gray!20}447.01 
    & 4.2 & 12.8 & 13.4 & 1.05 & \cellcolor{gray!20}1.16 
    & 2.2 & 5.2 & 26 & 0.36 & \cellcolor{gray!20}0.45 \\
    2 & 0.01 & 30 & M & 30.06\% & \cellcolor{gray!20}T 
    & 4.8 & 13.6 & 816 & 672.96 & \cellcolor{gray!20}673.56 
    & 4.6 & 13.8 & 18.4 & 6.23 & \cellcolor{gray!20}6.78 
    & 2.2 & 5.2 & 156 & 2.06 & \cellcolor{gray!20}2.56 \\
    3 & 1  & 30 & M & 28.27\% & \cellcolor{gray!20}T 
    & 4.0 & 11.0 & 660 & 460.50 & \cellcolor{gray!20}461.12 
    & 4.0 & 11.4 & 168.8 & 6.17 & \cellcolor{gray!20}6.75 
    & 2.2 & 5.2 & 156 & 2.21 & \cellcolor{gray!20}2.77 \\
    4 & 1  & 60 & M & 27.84\% & \cellcolor{gray!20}T 
    & 4.2 & 12.0 & 1440 & 528.48 & \cellcolor{gray!20}529.75 
    & 4.2 & 12.0 & 324.8 & 11.70 & \cellcolor{gray!20}12.80 
    & 2.2 & 5.0 & 300 & 4.15 & \cellcolor{gray!20}5.18 \\
    5 & 0.01 & 5  & I & 0  & \cellcolor{gray!20}61.90 
    & 1.8 & 1.4 & 14 & 57.00 & \cellcolor{gray!20}57.10 
    & 1.8 & 1.4 & 2 & 0.26 & \cellcolor{gray!20}0.35 
    & 1.2 & 0.2 & 1 & 0.14 & \cellcolor{gray!20}0.24 \\
    6 & 0.01 & 30 & I & 0  & \cellcolor{gray!20}261.19 
    & 2.2 & 2.4 & 144 & 99.72 & \cellcolor{gray!20}100.28 
    & 2.2 & 2.4 & 5.4 & 2.07 & \cellcolor{gray!20}2.60 
    & 1.2 & 0.2 & 6 & 0.84 & \cellcolor{gray!20}1.34 \\
    7 & 1  & 30 & I & 0  & \cellcolor{gray!20}280.64 
    & 1.4 & 0.6 & 36 & 26.16 & \cellcolor{gray!20}26.69 
    & 1.4 & 0.6 & 13.2 & 1.11 & \cellcolor{gray!20}1.61 
    & 1.2 & 0.2 & 6 & 0.98 & \cellcolor{gray!20}1.48 \\
    8 & 1  & 60 & I & 0  & \cellcolor{gray!20}1073.97 
    & 1.2 & 0.2 & 24 & 11.41 & \cellcolor{gray!20}12.42 
    & 1.2 & 0.2 & 12 & 2.37 & \cellcolor{gray!20}3.37 
    & 1.0 & 0.0 & 0 & 1.83 & \cellcolor{gray!20}2.84 \\
    \bottomrule
    \end{tabular}

    \end{threeparttable}
\end{table}

\begin{table}[ht!]
    \centering
    \caption{Performance metrics for AS-DRO with $\ell_2$-norm Wasserstein Ambiguity Sets}
    \label{tab:ASDRO_second_order_was}
    \renewcommand{\arraystretch}{1.0}
    \setlength{\tabcolsep}{2.5pt}

    \begin{threeparttable}
    \begin{tabular}{ccccccccccccccccccccc}
    \toprule
    \multirow{2}{*}{Case} 
      & \multirow{2}{*}{$\theta$} 
      & \multirow{2}{*}{$N$} 
      & \multirow{2}{*}{Var.} 
      & \multicolumn{2}{c}{Reformulation} 
      & \multicolumn{5}{c}{Oracle 1} 
      & \multicolumn{5}{c}{Oracle 2} 
      & \multicolumn{5}{c}{Oracle~2 Enhanced$^*$} \\
    \cmidrule(lr){5-6} \cmidrule(lr){7-11} \cmidrule(lr){12-16} \cmidrule(lr){17-21}
      &  &  & 
      & Gap & \cellcolor{gray!20}$t_{\mathrm{total}}$
      & Iter. & Cuts & $|\hat{\mathcal{A}}|$ & $t_{\mathrm{sp}}$ & \cellcolor{gray!20}$t_{\mathrm{total}}$ & Iter. & Cuts & $|\hat{\mathcal{A}}|$
      & $t_{\mathrm{sp}}$ & \cellcolor{gray!20}$t_{\mathrm{total}}$ & Iter. & Cuts & $|\hat{\mathcal{A}}|$
      & $t_{\mathrm{sp}}$ & \cellcolor{gray!20}$t_{\mathrm{total}}$  \\
    \midrule
    1 & 0.01 & 5  & M & 24.87\% & \cellcolor{gray!20}T 
    & 3.8 & 5.2 & 51.2 & 311.88 & \cellcolor{gray!20}311.96 
    & 4.0 & 6.6 & 15.4 & 0.29 & \cellcolor{gray!20}0.32 
    & 2.0 & 3.0 & 15 & 0.09 & \cellcolor{gray!20}0.12 \\
    2 & 0.01 & 30 & M & 32.34\% & \cellcolor{gray!20}T 
    & 4.4 & 7.2 & 432.0 & 410.03 & \cellcolor{gray!20}410.37 
    & 4.0 & 7.8 & 71.0 & 1.71 & \cellcolor{gray!20}1.83 
    & 2.0 & 3.0 & 90 & 0.54 & \cellcolor{gray!20}0.65 \\
    3 & 1  & 30 & M & 29.05\% & \cellcolor{gray!20}T 
    & 4.0 & 7.0 & 408.8 & 288.64 & \cellcolor{gray!20}288.88 
    & 3.2 & 6.0 & 141.6 & 1.43 & \cellcolor{gray!20}1.59 
    & 2.0 & 3.0 & 90 & 0.53 & \cellcolor{gray!20}0.65 \\
    4 & 1  & 60 & M & 28.57\% & \cellcolor{gray!20}T 
    & 3.2 & 6.2 & 744.0 & 293.90 & \cellcolor{gray!20}294.32 
    & 3.0 & 5.8 & 265.0 & 2.70 & \cellcolor{gray!20}3.00 
    & 2.0 & 3.0 & 180 & 1.13 & \cellcolor{gray!20}1.36 \\
    5 & 0.01 & 5  & I & 0  & \cellcolor{gray!20}0.17 
    & 2.4 & 2.8 & 28.0 & 113.44 & \cellcolor{gray!20}113.48 
    & 2.4 & 2.8 & 12.2 & 0.21 & \cellcolor{gray!20}0.24 
    & 1.0 & 0.0 & 0 & 0.05 & \cellcolor{gray!20}0.07 \\
    6 & 0.01 & 30 & I & 0  & \cellcolor{gray!20}1.35 
    & 2.4 & 3.4 & 204.0 & 138.53 & \cellcolor{gray!20}138.71 
    & 2.4 & 3.4 & 64.6 & 1.25 & \cellcolor{gray!20}1.38 
    & 1.0 & 0.0 & 0 & 0.26 & \cellcolor{gray!20}0.40 \\
    7 & 1  & 30 & I & 0  & \cellcolor{gray!20}1.37 
    & 2.4 & 2.6 & 156.0 & 105.03 & \cellcolor{gray!20}105.20 
    & 2.4 & 2.6 & 78.4 & 1.20 & \cellcolor{gray!20}1.35 
    & 1.0 & 0.0 & 0 & 0.26 & \cellcolor{gray!20}0.38 \\
    8 & 1  & 60 & I & 0  & \cellcolor{gray!20}3.48 
    & 2.4 & 2.8 & 336.0 & 113.82 & \cellcolor{gray!20}114.15 
    & 2.4 & 2.8 & 159.2 & 2.43 & \cellcolor{gray!20}2.73 
    & 1.0 & 0.0 & 0 & 0.51 & \cellcolor{gray!20}0.74 \\
    \bottomrule
    \end{tabular}

    \begin{tablenotes}\footnotesize
    \item \textbf{Note:} For Tables~\ref{tab:ASDRO_first_order_was} and~\ref{tab:ASDRO_second_order_was}, T indicates a timeout at 3600 seconds and each row reports averages over 5 random seeds. The $\ell_1$ table has 50 variables and 10 constraints, while the $\ell_2$ table has 20 variables and 5 constraints; all cases use $\Gamma=5$ and $(d_p,d_n)=(0.3,0.1)$. $\theta$ is the Wasserstein radius and $N$ is the number of empirical sample points. In the Var. column, M denotes instances with continuous decision variables and I denotes pure-integer instances. The superscript $^*$ indicates that the empirical points defining $\hat{\mathbb P}_i$ are also used to initialize the scenario pool; this warm start does not otherwise alter the ambiguity set.
    \end{tablenotes}
    \end{threeparttable}
\end{table}    
\end{landscape}

\begin{landscape}
\begin{table}[ht!]
    \centering
    \caption{Performance metrics for DRCCP with First-order Moment Ambiguity Sets}
    \label{tab:DRCCP_first_order_moment}
    \renewcommand{\arraystretch}{1.0}
    \setlength{\tabcolsep}{2pt}

    \begin{threeparttable}
    \begin{tabular}{cccccc cc ccccc ccccc ccccc}
    \toprule
    \multirow{2}{*}{Case} 
      & \multirow{2}{*}{Size} 
      & \multirow{2}{*}{$d(A)$} 
      & \multirow{2}{*}{Model} 
      & \multirow{2}{*}{Type} 
      & \multirow{2}{*}{$I_{\max}$} 
      & \multicolumn{2}{c}{Reformulation} 
      & \multicolumn{5}{c}{Oracle 1} 
      & \multicolumn{5}{c}{Oracle 2} 
      & \multicolumn{5}{c}{Oracle 2 Enhanced} \\
      \cmidrule(lr){7-8} \cmidrule(lr){9-13} \cmidrule(lr){14-18} \cmidrule(lr){19-23}
      &  &  &  &  & 
      & Gap & \cellcolor{gray!20}$t_{\mathrm{total}}$ & Iter. & Cuts & $|\hat{\mathcal{A}}|$ 
      & $t_{\mathrm{sp}}$ & \cellcolor{gray!20}$t_{\mathrm{total}}$ & Iter. & Cuts & $|\hat{\mathcal{A}}|$
      & $t_{\mathrm{sp}}$ & \cellcolor{gray!20}$t_{\mathrm{total}}$ & Iter. & Cuts & $|\hat{\mathcal{A}}|$
      & $t_{\mathrm{sp}}$ & \cellcolor{gray!20}$t_{\mathrm{total}}$ \\
    \midrule
    1 & (20,5) & S & MIP & I & 1  & 0     & \cellcolor{gray!20}20.2  
    & 36.4 & 52.6 & 136 & 3.89 & \cellcolor{gray!20}4.22 
    & 5.2 & 9.4 & 10.2 & 0.05 & \cellcolor{gray!20}0.06 
    & 3.4 & 6.2 & 6.2 & 0.02 & \cellcolor{gray!20}0.03 \\
    2 & (40,5)   & D & IP &I & 1  & 0 & \cellcolor{gray!20}9.22 
    & 1.8 & 0.8  & 2.2   & 0.06  & \cellcolor{gray!20}0.09 
    & 1.8 & 1.0  & 1.0 & 0.02 & \cellcolor{gray!20}0.04 
    & 1.8 & 1.0  & 1.0  & 0.01 & \cellcolor{gray!20}0.03 \\
    3 & (200,30) & S & IP & I & 1  & 750\% & \cellcolor{gray!20}T    
     & 6.0 & 16.2 & 46.4  & 3.46  & \cellcolor{gray!20}9.83 
     & 5.4 & 24.8 & 25.6  & 0.67 & \cellcolor{gray!20}5.92 
     & 5.6 & 25.0 & 25.4 & 0.63 & \cellcolor{gray!20}4.94 \\
    4 & (20,5) & S & MIP & J & 2  & 0  & \cellcolor{gray!20}19.45 
    & 14.4 & 20.4 & 58.6 & 0.42 & \cellcolor{gray!20}0.47 
    & 4   & 5.6 & 5.8 & 0.04 & \cellcolor{gray!20}0.06 
    & 3.8 & 5.4 & 5.4 & 0.03 & \cellcolor{gray!20}0.04 \\
    5 & (40,5)   & D & IP & J & 2  & 0 & \cellcolor{gray!20}52.59 
    & 3.0 & 3.4  & 9.2   & 0.27 & \cellcolor{gray!20}0.36  
    & 3.2 & 3.8  & 4.0  & 0.06 & \cellcolor{gray!20}0.12 
    & 3.2 & 3.8  & 4.0  & 0.05 & \cellcolor{gray!20}0.11 \\
    6 & (200,30) & S & IP &J & 4  & 374\% & \cellcolor{gray!20}T    
    & 2.6 & 25.8 & 76.8  & 4.38  & \cellcolor{gray!20}4.97 
    & 2.8 & 28.6 & 29.0 & 0.95 & \cellcolor{gray!20}1.62 
    & 2.8 & 28.6 & 28.8 & 0.76 & \cellcolor{gray!20}1.45 \\
    7 & (200,30) & D & IP& C & 10 & --     & \cellcolor{gray!20}--    
    & 3.4 & 32.0 & 79.6 & 458.66 & \cellcolor{gray!20}463.15 
    & 3.4 & 32.0 & 32.0 & 3.79 & \cellcolor{gray!20}6.82 
    & 3.4 & 32.0 & 32.0 & 0.06 & \cellcolor{gray!20}6.19 \\
    8 & (40,5)   & D &IP& C & 5  & --     & \cellcolor{gray!20}--    
    & 3.4 & 31.0 & 92.0 & 13.92 & \cellcolor{gray!20}16.94 
    & 3.2 & 31.0 & 31.0 & 2.08 & \cellcolor{gray!20}5.21 
    & 3.2 & 31.0 & 31.0 & 0.01 & \cellcolor{gray!20}4.53 \\
    9 & (20,5) & S & MIP & C & 10 & --     & \cellcolor{gray!20}--    
    & 2.2 & 3.6 & 10.6 & 0.15 & \cellcolor{gray!20}0.17 
    & 2.2 & 4.0 & 4.0 & 0.04 & \cellcolor{gray!20}0.07 
    & 2.2 & 4.0 & 4.0 & 0.01 & \cellcolor{gray!20}0.06 \\
    \bottomrule
    \end{tabular}
    \begin{tablenotes}\footnotesize
    \item \textbf{Note:} For Table~\ref{tab:DRCCP_first_order_moment}, T indicates a timeout at 3600 seconds and a dash denotes an unavailable entry; each row reports averages over 5 random seeds. Size reports (number of decision variables, number of chance-constraint groups). D denotes 75\% positive and 5\% negative entries, while S denotes 40\% positive and 0\% negative entries; $\Gamma=5$. MIP and IP denote mixed-integer and pure-integer instances. Type denotes individual (I), joint (J), or combinatorial (C) chance constraints. Here $I_{\max}:=\max_{m\in[M]}I_m$ is the largest number of component constraints in a group.
    \end{tablenotes}
    \end{threeparttable}
\end{table}
\end{landscape}

\begin{landscape}
\begin{table}[ht!]
    \centering
    \caption{Performance metrics for DRCCP with Second-order Moment Ambiguity Sets}
    \label{tab:DRCCP_second_order_moment}
    \renewcommand{\arraystretch}{1.0}
    \setlength{\tabcolsep}{3pt}

    \begin{threeparttable}
    \begin{tabular}{ccc c c ccccc ccccc ccccc}
    \toprule
    \multirow{2}{*}{Case} 
      & \multirow{2}{*}{$d(A)$} 
      & \multirow{2}{*}{Model} 
      & \multirow{2}{*}{Type} 
      & \multirow{2}{*}{$I_{\max}$} 
      & \multicolumn{5}{c}{Oracle 1} 
      & \multicolumn{5}{c}{Oracle 2} 
      & \multicolumn{5}{c}{Oracle 2 Enhanced} \\
      \cmidrule(lr){6-10} \cmidrule(lr){11-15} \cmidrule(lr){16-20}
      &  &  &  &
      & Iter. & Cuts & $|\hat{\mathcal{A}}|$ & $t_{\mathrm{sp}}$ & \cellcolor{gray!20}$t_{\mathrm{total}}$ & Iter. & Cuts & $|\hat{\mathcal{A}}|$
      & $t_{\mathrm{sp}}$ & \cellcolor{gray!20}$t_{\mathrm{total}}$ & Iter. & Cuts & $|\hat{\mathcal{A}}|$
      & $t_{\mathrm{sp}}$ & \cellcolor{gray!20}$t_{\mathrm{total}}$ \\
    \midrule

1 & D & IP & I & 1 
  & 5.6 & 11.4 & 24.4& 146.51 & \cellcolor{gray!20}148.06 
  & 5.2 & 16.2 & 19.4& 0.44 & \cellcolor{gray!20}1.20 
  & 3.6 & 10.2 & 10.2& 0.26 & \cellcolor{gray!20}0.67  \\
2 & S & IP & I & 1
  & 5.0 & 18.2 & 50.6& 161.75 & \cellcolor{gray!20}162.59 
  & 4.6 & 22.2 & 25.8& 0.28 & \cellcolor{gray!20}0.69 
  & 4.4 & 20.0 & 20.2 & 0.21 & \cellcolor{gray!20}0.60 \\
3 & S & MIP & I & 1 
  & 4.4 & 15.2 & 44.0& 89.56 & \cellcolor{gray!20}90.13 
  & 4.4 & 20.0 & 22.0& 0.29 & \cellcolor{gray!20}0.72 
  & 4.0 & 18.4 & 18.6 & 0.19 & \cellcolor{gray!20}0.57 \\
4 & D & IP & J & 3
  & 5.8 & 19.4 & 30.6& 141.85 & \cellcolor{gray!20}144.99 
  & 6.4 & 25.4 & 30.4& 1.03 & \cellcolor{gray!20}2.97 
  & 5.8 & 19.0 & 19.0 & 0.75 & \cellcolor{gray!20}2.22 \\
5 & S & IP & J & 8
  & 2.2 & 17.0 & 30.4& 56.26 & \cellcolor{gray!20}56.77 
 & 2.6 & 19.0 & 20.6 & 0.55 & \cellcolor{gray!20}1.08 
  & 2.4 & 17.8 & 18.2 & 0.41 & \cellcolor{gray!20}0.80 \\
6 & S & MIP & J & 3
  & 3.0 & 11.0 & 22.4& 53.21 & \cellcolor{gray!20}53.79 
  & 3.0 & 14.2 & 16.0& 0.34 & \cellcolor{gray!20}0.89 
  & 2.6 & 12.2 & 12.2 & 0.23 & \cellcolor{gray!20}0.64 \\

7 & D & IP & C & 10
  & 6.2 & 24.4 & 41.0& 138.35 & \cellcolor{gray!20}154.18 
  & 6.2 & 25.0 & 25.6& 2.69 & \cellcolor{gray!20}12.09 
  & 6.2 & 24.0 & 24.2 & 2.28 & \cellcolor{gray!20}12.05 \\

8 & S & IP & C & 10
  & 4.0 & 20.6 & 31.6& 68.31 & \cellcolor{gray!20}114.69 
  & 3.8 & 21.4 & 23.4& 1.04 & \cellcolor{gray!20}43.82 
  & 3.8 & 20.2 & 20.6 & 0.83 & \cellcolor{gray!20}36.69 \\

9 & S & MIP & C & 5
  & 5.0 & 16.0 & 32.2& 119.37 &  \cellcolor{gray!20}123.90 
  & 4.4 & 18.4 & 20.8& 0.64 & \cellcolor{gray!20}3.60 
  & 4.2 & 14.8 & 15.0& 0.50 & \cellcolor{gray!20}3.15  \\

    \bottomrule
    \end{tabular}

    \begin{tablenotes}\footnotesize
    \item \textbf{Note:} For Table~\ref{tab:DRCCP_second_order_moment}, each row reports averages over 5 random seeds. All cases have 100 decision variables and 20 chance-constraint groups. D denotes 50\% positive and 10\% negative entries, while S denotes 20\% positive and 5\% negative entries; $\Gamma=10$. MIP and IP denote mixed-integer and pure-integer instances. Type denotes individual (I), joint (J), or combinatorial (C) chance constraints. Here $I_{\max}:=\max_{m\in[M]}I_m$ is the largest number of component constraints in a group.
    \end{tablenotes}
    \end{threeparttable}
\end{table}
\end{landscape}

\begin{landscape}
\begin{table}[ht!]
    \centering
    \caption{Performance metrics for DRCCP with $\ell_1$-norm Wasserstein Ambiguity Sets}
    \label{tab:DRCCP_first_order_was}
    \renewcommand{\arraystretch}{1.0}
    \setlength{\tabcolsep}{2pt}

    \begin{threeparttable}
    \begin{tabular}{cccccc cc ccccc ccccc ccccc}
    \toprule
    \multirow{2}{*}{Case} 
      & \multirow{2}{*}{Size} 
      & \multirow{2}{*}{$N$} 
      & \multirow{2}{*}{Model} 
      & \multirow{2}{*}{$I_{\max}$} 
      & \multirow{2}{*}{Type} 
      & \multicolumn{2}{c}{Reformulation} 
      & \multicolumn{5}{c}{Oracle 1} 
      & \multicolumn{5}{c}{Oracle 2} 
      & \multicolumn{5}{c}{Oracle~2 Enhanced$^*$} \\
      \cmidrule(lr){7-8} \cmidrule(lr){9-13} \cmidrule(lr){14-18} \cmidrule(lr){19-23}
      &  &  &  &  & 
      & Gap & \cellcolor{gray!20}$t_{\mathrm{total}}$
      & Iter. & Cuts & $|\hat{\mathcal{A}}|$& $t_{\mathrm{sp}}$ & \cellcolor{gray!20}$t_{\mathrm{total}}$ & Iter. & Cuts & $|\hat{\mathcal{A}}|$
      & $t_{\mathrm{sp}}$ & \cellcolor{gray!20}$t_{\mathrm{total}}$ & Iter. & Cuts & $|\hat{\mathcal{A}}|$
      & $t_{\mathrm{sp}}$ & \cellcolor{gray!20}$t_{\mathrm{total}}$  \\
    \midrule

1 & (100,10) & 30 & IP & 1  & I 
  & 66.4\% & \cellcolor{gray!20}T 
  & 2.4 & 1.8 & 102 & 72.4 & \cellcolor{gray!20}111.5 
  & 2.2 & 1.4 & 33.4& 3.6 & \cellcolor{gray!20}33.1 
  & 3.4 & 2.8 & 72.4& 5.1 & \cellcolor{gray!20}146.5  \\
2 & (20,4) & 5 & MIP & 1 & I
  & 3.16\% & \cellcolor{gray!20}T
  & 5.8 & 6.6 & 66& 0.84 & \cellcolor{gray!20}1.31 
  & 2.2 & 2.4 & 11.0& 0.44 & \cellcolor{gray!20}0.61
  & 2.2 & 2.4 & 11.0 & 0.38 & \cellcolor{gray!20}0.54 \\
3 & (40,4) & 10 & IP & 1 & I
  & 0 & \cellcolor{gray!20}8.1
  & 2.0 & 1.4 & 16& 8.5 & \cellcolor{gray!20}8.55 
  & 2.0 & 1.4 & 12.8& 0.32 & \cellcolor{gray!20}0.39 
  & 2.0 & 1.4 & 12.6& 0.25 & \cellcolor{gray!20}0.32 \\
4 & (100,10) & 30 & IP & 4 & J
  & -- & \cellcolor{gray!20}T
  & 1.8 & 2.2 & 120& 95.22 & \cellcolor{gray!20}277.11 
  & 1.8 & 2.2 & 52.0& 7.76 & \cellcolor{gray!20}80.80 
  & 2.0 & 3.4 & 84.6 & 7.14 & \cellcolor{gray!20}125.61 \\
5 & (20,4) & 5 & MIP & 4 & J
  & 0.13\% & \cellcolor{gray!20}T
  & 4.6 & 6.4 & 41.8& 33.57 & \cellcolor{gray!20}33.66 
  & 2.2 & 2.8 & 13.4& 0.25 & \cellcolor{gray!20}0.30 
  & 2.2 & 2.8 & 12.8& 0.16 & \cellcolor{gray!20}0.21  \\
6 & (40,4) & 10 & IP & 4 & J
  & 0 & \cellcolor{gray!20}314.4
  & 2.4 & 1.4 & 22.4& 33.37 & \cellcolor{gray!20}33.64 
  & 3.0 & 2.0 & 16.4& 0.89 & \cellcolor{gray!20}1.13 
  & 2.6 & 1.6 & 14.4& 0.55 & \cellcolor{gray!20}0.78  \\
7 & (100,10) & 5 & IP & 10 & C
  & -- & \cellcolor{gray!20}--
  & 4.4 & 5.2 & 40.2& 128.04 & \cellcolor{gray!20}157.44 
  & 4.6 & 5.4 & 25.6& 6.62 & \cellcolor{gray!20}26.43 
  & 5.4 & 6.2 & 28.8& 5.75 & \cellcolor{gray!20}27.41  \\
8 & (20,4) & 5 & MIP & 5 & C
  & -- & \cellcolor{gray!20}--
  & 6.2 & 8.4 & 59.2& 76.68 & \cellcolor{gray!20}76.79 
  & 2.4 & 2.4 & 10.8& 0.25 & \cellcolor{gray!20}0.29 
  & 2.0 & 2.0 & 9.0& 0.16 & \cellcolor{gray!20}0.20  \\
9 & (40,4) & 10 & IP & 10 & C
  & -- & \cellcolor{gray!20}--
  & 3.0 & 2.4 & 40& 68.13 & \cellcolor{gray!20}73.45 
  & 3.0 & 2.4 & 19.2& 1.87 & \cellcolor{gray!20}9.93 
  & 3.0 & 2.4 & 22.0& 1.39 & \cellcolor{gray!20}6.92  \\

    \bottomrule
    \end{tabular}
    \begin{tablenotes}\footnotesize \item \textbf{Note:} For Table~\ref{tab:DRCCP_first_order_was}, T indicates a timeout at 3600 seconds and a dash denotes an unavailable entry; each row reports averages over 5 random seeds. Size reports (number of decision variables, number of chance-constraint groups). Matrices have 30\% positive and 10\% negative entries, $\Gamma=5$, and $\theta=5$; $N$ is the number of empirical sample points. MIP and IP denote mixed-integer and pure-integer instances. Type denotes individual (I), joint (J), or combinatorial (C) chance constraints, and $I_{\max}:=\max_{m\in[M]}I_m$. The superscript $^*$ indicates that the empirical points defining the relevant empirical distribution are also used to initialize the scenario pool; this warm start does not otherwise alter the ambiguity set. \end{tablenotes}
    \end{threeparttable}
\end{table}

\end{landscape}

\begin{landscape}
    \begin{table}[ht!]
    \centering
    \caption{Performance metrics for DRCCP with $\ell_2$-norm Wasserstein Ambiguity Sets}
    \label{tab:DRCCP_second_order_was}
    \renewcommand{\arraystretch}{1.0}
    \setlength{\tabcolsep}{2pt}

    \begin{threeparttable}
    \begin{tabular}{cccccc cc ccccc ccccc ccccc}
    \toprule
    \multirow{2}{*}{Case} 
      & \multirow{2}{*}{Size} 
      & \multirow{2}{*}{$N$} 
      & \multirow{2}{*}{Model} 
      & \multirow{2}{*}{$I_{\max}$} 
      & \multirow{2}{*}{Type} 
      & \multicolumn{2}{c}{Reformulation} 
      & \multicolumn{5}{c}{Oracle 1} 
      & \multicolumn{5}{c}{Oracle 2} 
      & \multicolumn{5}{c}{Oracle~2 Enhanced$^*$} \\
      \cmidrule(lr){7-8} \cmidrule(lr){9-13} \cmidrule(lr){14-18} \cmidrule(lr){19-23}
      &  &  &  &  & 
      & Gap & \cellcolor{gray!20}$t_{\mathrm{total}}$
      & Iter. & Cuts & $|\hat{\mathcal{A}}|$& $t_{\mathrm{sp}}$ & \cellcolor{gray!20}$t_{\mathrm{total}}$ & Iter. & Cuts & $|\hat{\mathcal{A}}|$
      & $t_{\mathrm{sp}}$ & \cellcolor{gray!20}$t_{\mathrm{total}}$ & Iter. & Cuts & $|\hat{\mathcal{A}}|$
      & $t_{\mathrm{sp}}$ & \cellcolor{gray!20}$t_{\mathrm{total}}$  \\
    \midrule

1 & (100,10) & 30 & IP & 1  & I 
  & 53.8\% & \cellcolor{gray!20}T 
  & 3 & 3 & 162.6& 112.5 & \cellcolor{gray!20}135.9 
  & 3 & 2.8 & 75.8& 6.2 & \cellcolor{gray!20}23.6 
  & 2.8 & 2.6 & 74.8 & 3.7 & \cellcolor{gray!20}25 \\

2 & (20,4) & 5 & MIP & 1 & I
  & 0.93\% & \cellcolor{gray!20}T
  & 3.2 & 2.8 & 25.6& 0.2 & \cellcolor{gray!20}0.2 
  & 9.2 & 8.4 & 41.6& 0.6 & \cellcolor{gray!20}0.7 
  & 9.2 & 8.4 & 41.6 & 0.5 & \cellcolor{gray!20}0.6 \\

3 & (40,4) & 20 & IP & 1 & I
  & 0 & \cellcolor{gray!20}9.7
  & 2 & 1.2 & 33.6& 19.6 & \cellcolor{gray!20}19.7 
  & 2.4 & 1.6 & 28.2& 0.7 & \cellcolor{gray!20}0.9 
  & 2 & 1.2 & 22.2 & 0.5 & \cellcolor{gray!20}0.7 \\

4 & (100,10) & 30 & IP & 4 & J
  & -- & \cellcolor{gray!20}--
  & 2 & 2.8 & 150.8& 121.5 & \cellcolor{gray!20}195.5 
  & 2 & 3 & 72.4& 6.4 & \cellcolor{gray!20}97.1 
  & 2 & 3 & 78.8& 4.8 & \cellcolor{gray!20}64.1 \\

5 & (20,4) & 5 & MIP & 4 & J
  & -- & \cellcolor{gray!20}--
  & 2.4 & 1.6 & 12.6 & 8.5 & \cellcolor{gray!20}8.6 
  & 13.2 & 12.6 & 51.6 & 1.2 & \cellcolor{gray!20}1.3 
  & 14 & 13.4 & 54.6 & 1.0 & \cellcolor{gray!20}1.2 \\

6 & (40,4) & 10 & IP & 4 & J
  & -- & \cellcolor{gray!20}--
  & 2.8 & 2 & 37.2 & 65.8 & \cellcolor{gray!20}66.0 
  & 2.8 & 2 & 17.2 & 0.9 & \cellcolor{gray!20}1.2 
  & 2.8 & 2 & 18.2 & 0.6 & \cellcolor{gray!20}0.9 \\

7 & (100,10) & 5 & IP & 10 & C
  & -- & \cellcolor{gray!20}--
  & 3.4 & 4.6 & 36.2& 103.8 & \cellcolor{gray!20}184.2 
  & 4.2 & 5.2 & 22.6& 4.5 & \cellcolor{gray!20}103.4 
  & 4 & 4.8 & 21.6 & 3.2 & \cellcolor{gray!20}79.5 \\

8 & (20,4) & 5 & MIP & 5 & C
  & -- & \cellcolor{gray!20}--
  & 2.8 & 2.2 & 19.2& 0.6 & \cellcolor{gray!20}0.7 
  & 14.6 & 14.2 & 67.4& 1.2 & \cellcolor{gray!20}1.3 
  & 14.6 & 14.2 & 67 & 0.9 & \cellcolor{gray!20}1.0 \\

9 & (40,4) & 10 & IP & 10 & C
  & -- & \cellcolor{gray!20}--
  & 2.2 & 1.6 & 28.6& 43.1 & \cellcolor{gray!20}44.8 
  & 2.4 & 1.6 & 13.8& 1.1 & \cellcolor{gray!20}2.9 
  & 2.2 & 1.4 & 13 & 0.8 & \cellcolor{gray!20}2.3 \\ 

    \bottomrule
    \end{tabular}
    \begin{tablenotes}\footnotesize \item \textbf{Note:} For Table~\ref{tab:DRCCP_second_order_was}, T indicates a timeout at 3600 seconds and a dash denotes an unavailable entry; each row reports averages over 5 random seeds. Size reports (number of decision variables, number of chance-constraint groups). Matrices have 30\% positive and 10\% negative entries, $\Gamma=5$, and $\theta=2$; $N$ is the number of empirical sample points. MIP and IP denote mixed-integer and pure-integer instances. Type denotes individual (I), joint (J), or combinatorial (C) chance constraints, and $I_{\max}:=\max_{m\in[M]}I_m$. The superscript $^*$ indicates that the empirical points defining the relevant empirical distribution are also used to initialize the scenario pool; this warm start does not otherwise alter the ambiguity set. \end{tablenotes}
    \end{threeparttable}
\end{table}
\end{landscape}

\appendix
\section{Proofs and Supplementary Formulations}
\subsection{Proof of Theorem~\ref{linear-robust}: RO Reformulation}
\label{app:proof-thm32}
\begin{proof}
Fix $m \in [M]$ and $\bm x\in\mathcal X$. Whenever the $m$th robust row is feasible, its inner supremum is finite. The conic Slater condition then yields strong duality and dual attainment:
\[
\sup_{\bm a_m\in\mathcal A_m}\bm a_m^\top\bm x
=\min_{\substack{\bm p_m\in\mathcal K_m^*\\
\bm D_m^\top\bm p_m=\bm x}}
\bm p_m^\top\bm d_m.
\]
Thus a feasible robust row admits the certificate in~\eqref{eq_RO_duality_reform}. Conversely, any $\bm p_m\in\mathcal K_m^*$ satisfying $\bm D_m^\top\bm p_m=\bm x$ and $\bm p_m^\top\bm d_m\le b_m$ implies the robust row by conic weak duality. Applying this equivalence to every $m \in [M]$ proves~\eqref{eq_RO_duality_reform}.
\end{proof}

\subsection{Proof of Theorem~\ref{linear-dro}: Standard DRO Reformulation}
\label{app:proof-thm34}
\begin{proof}
\textbf{Step 1: Dualization over the probability measure.}
Fix $m \in [M]$ and $\bm x \in \mathcal X$. Let $\mathfrak C_m$ be the cone of finite nonnegative measures $\mu$ on $(\mathcal A_m,\mathcal F_m)$ such that $\bm e_{mt}^\top\bm a_m$ is $\mu$-integrable for every $t\in[T_m]$. Assumption~\ref{assumption:regularity}(ii) then also ensures the integrability of $\bm a_m^\top\bm x$. The inner DRO problem is
\begin{equation}\label{primal_continuous}
\begin{aligned}
\sup_{\mathbb P_m\in\mathfrak C_m}\quad
&\int_{\mathcal A_m}\bm a_m^\top\bm x\,d\mathbb P_m \\
\text{s.t.}\quad
&\mathbb P_m(\mathcal A_m)=1,\quad
\int_{\mathcal A_m}\bm e_{mt}^\top\bm a_m\,d\mathbb P_m\le\gamma_{mt},
\quad t\in[T_m].
\end{aligned}
\end{equation}
Define $\mathcal L_m\mu:=\bigl(\mu(\mathcal A_m),(\int_{\mathcal A_m}\bm e_{mt}^\top\bm a_m\,d\mu)_{t\in[T_m]}\bigr)$, $\bm r_m:=(1,\bm\gamma_m^\top)^\top$, and $\mathcal Q_m:=\{0\}\times\mathbb R_-^{T_m}$. The constraints in~\eqref{primal_continuous} are $\mathcal L_m\mu-\bm r_m\in\mathcal Q_m$. If $\bar{\mathbb P}_m$ is the strictly feasible distribution from Theorem~\ref{linear-dro}, then, for every $\bm y$ sufficiently close to $\bm r_m$, $y_0>0$ and $\mu=y_0\bar{\mathbb P}_m$ satisfies $\mathcal L_m\mu-\bm y\in\mathcal Q_m$. Hence $\bm r_m\in\operatorname{int}[\mathcal L_m(\mathfrak C_m)-\mathcal Q_m]$. Proposition~3.4 of \cite{shapiro2001duality} therefore gives zero duality gap and, whenever the common value is finite, dual attainment. The dual is
\begin{equation}\label{dual-continuous}
\begin{aligned}
\inf_{\substack{\alpha_m\in\mathbb R\\\bm\beta_m\in\mathbb R_+^{T_m}}}\quad
&\alpha_m+\bm\beta_m^\top\bm\gamma_m \\
\text{s.t.}\quad
&\alpha_m+\sum_{t\in[T_m]}\beta_{mt}\bm e_{mt}^\top\bm a_m
\ge\bm a_m^\top\bm x,\quad \forall \bm a_m\in\mathcal A_m.
\end{aligned}
\end{equation}

\textbf{Step 2: Dualization over the sample space.}
If the $m$th DRO row is feasible, the moment problem has a finite value, so~\eqref{dual-continuous} attains an optimum $(\alpha_m,\bm\beta_m)$. Its pointwise constraint gives an upper bound of $\alpha_m$ on the sample-space supremum. By the conic Slater condition and Appendix~\ref{app:proof-thm32}, this finite supremum satisfies
\[
\sup_{\bm a_m\in\mathcal A_m}
\bm a_m^\top(\bm x-\bm E_m^\top\bm\beta_m)
=\min_{\substack{\bm p_m\in\mathcal K_m^*\\
\bm D_m^\top\bm p_m=\bm x-\bm E_m^\top\bm\beta_m}}
\bm p_m^\top\bm d_m.
\]
Because $\bm p_m^\top\bm d_m\le\alpha_m$ and $\alpha_m+\bm\beta_m^\top\bm\gamma_m\le b_m$, the minimizing $\bm p_m$ gives the certificate in~\eqref{eq_DRO_reform}. Conversely, given any certificate in~\eqref{eq_DRO_reform}, setting $\alpha_m:=\bm p_m^\top\bm d_m$ makes $(\alpha_m,\bm\beta_m)$ feasible in~\eqref{dual-continuous} by conic weak duality; moment weak duality then implies the original DRO row. Applying both directions to every $m\in[M]$ proves~\eqref{eq_DRO_reform}.
\end{proof}

\subsection{Proof of Theorem~\ref{thm:discretize_unbounded}: Finite-Support Approximation}
\label{app:proof-thm3}
\begin{proof}
For each integer $n \geq 1$, define the truncation set $T_n := \{\bm{a} \in \mathcal{A} : |h(\bm{a})| \leq n\}$, so that $T_n \uparrow \mathcal{A}$ as $n \to \infty$. Since $|h| \in L^1(\mathbb{P})$, the dominated convergence theorem yields $\mathbb{P}(T_n) \to 1$ and $\int_{\mathcal{A} \setminus T_n} |h| \, d\mathbb{P} \to 0$. We restrict attention to $n$ large enough that $\mathbb{P}(T_n) > 0$.

On $T_n$, the function $h$ takes values in $[-n, n]$. Partition this interval into $K_n = 2n^2$ subintervals $\{I_k^{(n)}\}_{k=1}^{K_n}$ of equal length $1/n$, inducing a measurable partition $B_k^{(n)} := \{\bm{a} \in T_n : h(\bm{a}) \in I_k^{(n)}\}$ of $T_n$. For each $k$ with $\mathbb{P}(B_k^{(n)}) > 0$, pick any representative $\bm{a}_k^{(n)} \in B_k^{(n)}$; cells with zero probability are omitted. Let $\mathcal{A}_n^* := \{\bm{a}_k^{(n)} : \mathbb{P}(B_k^{(n)}) > 0\}$, set weights $w_k^{(n)} := \mathbb{P}(B_k^{(n)})/\mathbb{P}(T_n)$ so that $\sum_k w_k^{(n)} = 1$, and define the finitely supported probability measure $\mathbb{P}_n \;:=\; \sum_{\bm{a}_k^{(n)} \in \mathcal{A}_n^*} w_k^{(n)}\, \delta_{\bm{a}_k^{(n)}}$.

Since $h(\bm{a})$ and $h(\bm{a}_k^{(n)})$ both lie in $I_k^{(n)}$ whenever $\bm{a} \in B_k^{(n)}$, we have $|h(\bm{a}) - h(\bm{a}_k^{(n)})| \leq 1/n$ on each $B_k^{(n)}$. Noting that
\[
\mathbb{P}(T_n) \cdot \mathbb{E}_{\mathbb{P}_n}[h] \;=\; \sum_k \mathbb{P}(B_k^{(n)})\, h(\bm{a}_k^{(n)}) \;=\; \sum_k \int_{B_k^{(n)}} h(\bm{a}_k^{(n)}) \, d\mathbb{P}(\bm{a}),
\]
we obtain
\[
\left| \int_{T_n} h \, d\mathbb{P} \;-\; \mathbb{P}(T_n) \cdot \mathbb{E}_{\mathbb{P}_n}[h] \right| \;\leq\; \sum_k \int_{B_k^{(n)}} \bigl|h(\bm{a}) - h(\bm{a}_k^{(n)})\bigr|\, d\mathbb{P}(\bm{a}) \;\leq\; \frac{1}{n}.
\]
Decomposing $\mathbb{E}_{\mathbb{P}}[h] = \int_{T_n} h \, d\mathbb{P} + \int_{\mathcal{A} \setminus T_n} h \, d\mathbb{P}$ and applying the triangle inequality gives
\[
\bigl| \mathbb{E}_{\mathbb{P}}[h] - \mathbb{P}(T_n) \cdot \mathbb{E}_{\mathbb{P}_n}[h] \bigr| \;\leq\; \int_{\mathcal{A} \setminus T_n} |h| \, d\mathbb{P} \;+\; \frac{1}{n}.
\]
Moreover, since $|h(\bm{a}_k^{(n)})| \leq n$ for all $k$, we have $|\mathbb{E}_{\mathbb{P}_n}[h]| \leq n$, and using $|h| > n$ on $\mathcal{A}\setminus T_n$,
\[
\bigl(1 - \mathbb{P}(T_n)\bigr) \cdot |\mathbb{E}_{\mathbb{P}_n}[h]| \;\leq\; n \cdot \mathbb{P}(\mathcal{A}\setminus T_n) \;\leq\; \int_{\mathcal{A}\setminus T_n} |h|\, d\mathbb{P}.
\]
Combining the last two displays with the triangle inequality,
\[
\bigl|\mathbb{E}_{\mathbb{P}}[h] - \mathbb{E}_{\mathbb{P}_n}[h]\bigr| \;\leq\; 2\int_{\mathcal{A}\setminus T_n} |h|\, d\mathbb{P} \;+\; \frac{1}{n} \;\longrightarrow\; 0,
\]
which completes the proof.
\end{proof}

\subsection{Master Problem Reformulations 
(Section~\ref{BiCS Master Problem})}
\label{appendix:master}

The inner maximization in the master problem 
\eqref{MasterProblem} optimizes the probability measure $\mathbb P_m$, represented by the weight vector $\bm P_m:=(P_{m0},\ldots,P_{mK_m})^\top\in\mathbb R_+^{K_m+1}$, over the finite scenario set 
$\hat{A}_m = \{\bm{a}_{mk}\}_{k=0}^{K_m}$. By Step~1 of Algorithm~\ref{BiCSalgorithm}, $\mathcal P_m(\hat A_m)$ is nonempty initially and remains nonempty as $\hat A_m$ expands. Since the ambiguity set 
$\mathcal{P}_m$ (Definition~\ref{ambiguity_set_define}) imposes only 
linear constraints on the probability weights, this inner problem 
takes the following primal--dual pair for each constraint $m \in [M]$.

\noindent\textbf{Primal (inner maximization):}
\begin{equation} \label{primal-discrete}
\begin{aligned}
    \max_{\bm P_m} \quad
    & \sum_{k=0}^{K_m}P_{mk}f_m(\bm a_{mk},\bm x) \\
    \text{s.t.} \quad
    & \sum_{k=0}^{K_m}P_{mk}=1, \qquad
      \sum_{k=0}^{K_m}P_{mk}g_t(\bm a_{mk})\le\gamma_{mt},
      \quad \forall t\in[T_m].
\end{aligned}
\end{equation}

\noindent\textbf{Dual:}\; For $\alpha_m\in\mathbb R$ and $\bm\beta_m\in\mathbb R_+^{T_m}$,
\begin{equation} \label{dual-discrete}
\begin{aligned}
    \min_{\alpha_m,\,\bm\beta_m} \quad
    & \alpha_m+\bm\beta_m^\top\bm\gamma_m \\
    \text{s.t.} \quad
    & \alpha_m+\sum_{t\in[T_m]}\beta_{mt}g_t(\bm a_{mk})
      \ge f_m(\bm a_{mk},\bm x),
      \quad \forall k=0,\ldots,K_m.
\end{aligned}
\end{equation}

The linear structure of the probability measure allows us to apply 
standard optimality-based reformulations to eliminate the inner 
maximization and obtain a single-level representation of the master 
problem with an embedded worst-case expected value. We present two such reformulations below.

\medskip
\noindent\textbf{Dual reformulation.}\; 
By substituting the dual formulation \eqref{dual-discrete} into the 
left-hand side of constraint \eqref{inner_max_constraint} and imposing
$\alpha_m+\bm\beta_m^\top\bm\gamma_m\le b_m$ for each $m\in[M]$, the inner
maximization can be eliminated explicitly. This substitution reduces 
the master problem with an embedded worst-case expected value \eqref{MasterProblem} to a single-level 
optimization problem that can be solved using standard commercial 
solvers. Although formulation \eqref{dual-discrete} may appear similar 
to the continuous dual formulation \eqref{dual-continuous}, the two are 
fundamentally different. The discrete dual \eqref{dual-discrete} is 
well suited for an iterative algorithm, as the scenario set 
$\hat{A}_m$ is a finite subset of the sample space generated by the 
subproblem and can be enumerated explicitly. In contrast, the 
continuous dual formulation \eqref{dual-continuous} requires 
enumeration over the entire sample space, which is intractable.

\medskip
\noindent\textbf{KKT reformulation.}\; 
Alternatively, the inner maximization can be reformulated using the 
KKT conditions \cite{karush1939minima, kuhn2013nonlinear}. For linear 
programs, the KKT conditions provide necessary and sufficient 
conditions for optimality, comprising primal feasibility 
\eqref{primal-discrete}, dual feasibility \eqref{dual-discrete}, and 
the following complementary slackness conditions:
\begin{subequations} \label{complementary-slackness}
\begin{align}
    & P_{mk} \!\left( \alpha_m + \sum_{t \in [T_m]} \beta_{mt}\, 
      g_t(\bm{a}_{mk}) - f_m(\bm{a}_{mk}, \bm{x}) \right) = 0, 
      \quad \forall k=0,\ldots,K_m,\; \forall m \in [M], \\
    & \beta_{mt} \!\left( \sum_{k=0}^{K_m} P_{mk}\, 
      g_t(\bm{a}_{mk}) - \gamma_{mt} \right) = 0, 
      \quad \forall t \in [T_m],\; \forall m \in [M].
\end{align}
\end{subequations}
Any solution satisfying these conditions is optimal for the inner 
linear program, so enforcing them together with
$\sum_{k=0}^{K_m}P_{mk}f_m(\bm a_{mk},\bm x)\le b_m$ for each $m\in[M]$
eliminates the inner maximization operator. When valid finite bounds are available, the bilinear complementary slackness terms can then be 
linearized via a standard Big-$M$ technique, yielding a single-level 
mixed-integer formulation.

\subsection{Scenario Reduction Model}\label{appendix:scenario_reduction}

Let $\mathcal I_m^{\mathrm{new}}\subseteq\{0,\ldots,K_m\}$ index the scenarios generated in the current oracle call at $\bm x^*$, where $\hat A_m=\{\bm a_{mk}\}_{k=0}^{K_m}$. The following model preserves the returned oracle value $\hat\Omega_m(\bm x^*)$ using the fewest new scenarios:
\begin{align*}
\min_{\bm P_m,\,\bm y} \quad & \sum_{k\in\mathcal I_m^{\mathrm{new}}} y_k\\
\text{s.t.} \quad 
& \sum_{k=0}^{K_m} P_{mk}f_m(\bm a_{mk},\bm x^*)=\hat\Omega_m(\bm x^*),
  \qquad \sum_{k=0}^{K_m}P_{mk}=1, \\
& \sum_{k=0}^{K_m}P_{mk}g_t(\bm a_{mk})\le\gamma_{mt},
  \quad \forall t\in[T_m], \qquad
  P_{mk}\ge0, \quad \forall k=0,\ldots,K_m, \\
& P_{mk}\le y_k,\quad y_k\in\{0,1\},
  \quad \forall k\in\mathcal I_m^{\mathrm{new}}.
\end{align*}
Here, $y_k$ indicates whether new scenario $\bm a_{mk}$ is retained; the existing support remains available because the linking constraints apply only to $k\in\mathcal I_m^{\mathrm{new}}$. The value equality preserves the returned oracle value, while the objective minimizes the number of retained new scenarios. Only new scenarios with $y_k^*=1$ are added to the master problem.

\bibliographystyle{abbrv}
\bibliography{ref}

\end{document}